\documentclass[12pt]{amsart}

\usepackage[matrix,arrow,curve,frame]{xy}
\usepackage{amsmath,amsthm,amssymb,enumerate}
\usepackage{latexsym}
\usepackage{amscd}
\usepackage[colorlinks=false]{hyperref}
\usepackage{euscript}
\usepackage{appendix}

\newtheorem{thm}{Theorem}[section]
\newtheorem{prop}{Proposition}[section]
\newtheorem{cor}{Corollary}[section]
\newtheorem{lemma}{Lemma}[section]
\newtheorem{mainthm}{Main Theorem}

\theoremstyle{definition}
\newtheorem{defn}{Definition}[section]
\newtheorem{conj}{Conjecture}[section]

\theoremstyle{remark}
\newtheorem{remark}{Remark}[section]

\numberwithin{equation}{section}

\def\cO{{\cal O}}

\def\cU{{\cal U}}
\def\cG{{\cal G}}

\def\cA{{\cal A}}

\def\ra{\rightarrow}

\def\A{{\bf A}}

\def\cA{{\mathcal A}}

\def\cC{{\mathcal C}}
\def\cD{{\mathcal D}}
\def\cE{{\mathcal E}}
\def\cF{{\mathcal F}}
\def\cG{{\mathcal G}}
\def\cH{{\mathcal H}}
\def\cI{{\mathcal I}}

\def\cM{{\mathcal M}}

\def\cO{{\mathcal O}}

\def\cR{{\mathcal R}}
\def\cS{{\mathcal S}}

\def\cU{{\mathcal U}}
\def\cV{{\mathcal V}}
\def\cW{{\mathcal W}}

\def\cZ{{\mathcal Z}}
\def\ga{{\mathfrak a}}
\def\gb{{\mathfrak b}}

\def\gd{{\mathfrak d}}

\def\gg{{\mathfrak g}}

\def\gl{{\mathfrak l}}

\def\go{{\mathfrak o}}
\def\gp{{\mathfrak p}}

\def\gs{{\mathfrak s}}

\DeclareMathOperator{\id}{id}

\DeclareMathOperator{\Hom}{Hom}

\newcommand{\btimes}{\boxtimes}
\newcommand\one{\mathbf{1}}

\newfont{\german}{eufm10}

\begin{document}
\pagestyle{plain}

\title
{Trialities of orthosymplectic $\cW$-algebras}

\author{Thomas Creutzig}
\address{University of Alberta}
\email{creutzig@ualberta.ca}

\author{Andrew R. Linshaw}
\address{University of Denver}
\email{andrew.linshaw@du.edu}

\thanks{T. C. is supported by NSERC Discovery Grant \#RES0048511. A. L. is supported by Simons Foundation Grant \#635650 and NSF Grant DMS-2001484. We thank Tomoyuki Arakawa, Boris Feigin, Naoki Genra, Shashank Kanade, Shigenori Nakatsuka, and Ryo Sato for many illuminating discussions on topics related to this paper.}

{\abstract \noindent Trialities of $\cW$-algebras are isomorphisms between the affine cosets of three different $\cW$-(super)algebras, and were first conjectured in the physics literature by Gaiotto and Rap\v{c}\'ak. In this paper we prove trialities among eight families of $\cW$-(super)algebras of types $B$, $C$, and $D$. The key idea is to identify the affine cosets of these algebras with one-parameter quotients of the universal two-parameter even spin $\cW_{\infty}$-algebra which was recently constructed by Kanade and the second author. Our result is a vast generalization of both Feigin-Frenkel duality in types $B$, $C$, and $D$, and the coset realization of principal $\cW$-algebras of type $D$ due to Arakawa and us. It also provides a new coset realization of principal $\cW$-algebras of types $B$ and $C$. As an application, we prove the rationality of the affine vertex superalgebra $L_k(\go\gs\gp_{1|2n})$, the minimal $\cW$-algebra $\cW_{k-1/2}(\gs\gp_{2n+2}, f_{\text{min}})$, and the coset $\text{Com}(L_k(\gs\gp_{2m}), L_k(\gs\gp_{2n}))$, for all integers $k,n,m \geq 1$ with $m<n$. We also prove the rationality of some families of principal $\cW$-superalgebras of $\go\gs\gp_{1|2n}$ and $\go\gs\gp_{2|2n}$, and subregular $\cW$-algebras of $\gs\go_{2n+1}$.}

\maketitle
\section{Introduction} Trialities of $\cW$-algebras are isomorphisms between the affine cosets of three different $\cW$-(super)algebras, as one-parameter vertex algebras. In recent work \cite[Thm. 1.1]{CL3}, we proved a family of trialities among $\cW$-(super)algebras of type $A$ which was conjectured by Gaiotto and Rap\v{c}\'ak in \cite{GR}. This theorem is a common generalization of both Feigin-Frenkel duality and the coset realization of  principal $\cW$-algebras of $\gs\gl_n$ \cite{FF2,ACL}, as well as Feigin-Semikhatov duality between subregular $\cW$-algebras of $\gs\gl_n$, and principal $\cW$-superalgebras of $\gs\gl_{n|1}$ \cite{FS,CGN}. The key idea of the proof was to identify all these affine cosets with one-parameter quotients of the universal two-parameter vertex algebra $\cW(c,\lambda)$ of type $\cW(2,3,\dots)$. The existence and uniqueness of this structure was conjectured for many years in the physics literature \cite{YW,GG,Pro1,Pro2}, and was recently proven by the second author in \cite{L2}. One-parameter quotients of $\cW(c,\lambda)$ are in bijection with a family of curves in the parameter space $\mathbb{C}^2$ called truncation curves, and \cite[Thm. 1.1]{CL3} follows from the explicit ideals that define these curves.

In this paper, we will prove an analogous triality theorem which involves eight families of $\cW$-(super)algebras of types $B$, $C$, and $D$. First, $\gg$ will be either $\mathfrak{so}_{2n+1}$, $\mathfrak{sp}_{2n}$, $\mathfrak{so}_{2n}$, or $\mathfrak{osp}_{n|2r}$, and will decompose as $$\mathfrak{g} = \mathfrak{a} \oplus \mathfrak{b} \oplus \rho_{\mathfrak{a}} \otimes \rho_{\mathfrak{b}}.$$
Here $\mathfrak{a}$ and $\mathfrak{b}$ are Lie sub(super)algebras of $\mathfrak{g}$, and $\rho_{\mathfrak{a}}$, $\rho_{\mathfrak{b}}$ transform as the standard representations of $\mathfrak{a}$, $\mathfrak{b}$, respectively, and have the same parity which can be either even or odd.

Let $f_{\mathfrak{b}} \in \mathfrak{g}$ be the nilpotent element which is principal in $\mathfrak{b}$ and trivial in $\mathfrak{a}$, and let $\cW^k(\gg,f_{\gb})$ be the corresponding $\cW$-(super)algebra. In all cases, $\cW^k(\gg,f_{\gb})$ is of type
$$\cW \bigg(1^{\text{dim}\ \mathfrak{a}}, 2,4,\dots, 2m, \bigg(\frac{{d_{\mathfrak{b}}} + 1}{2} \bigg)^{d_{\mathfrak{a}}}\bigg).$$ In particular, there are $\text{dim}\ \mathfrak{a}$ fields in weight $1$ which generate an affine vertex (super)algebra of $\ga$. The fields in weights $2,4,\dots, 2m$ are even and are invariant under $\mathfrak{a}$. The $d_{\mathfrak{a}}$ fields in weight $\frac{d_{\mathfrak{b}} + 1}{2}$ can be even or odd, and transform as the standard $\mathfrak{a}$-module.

For $n,m\geq 0$ we have the following cases where $\mathfrak{b} = \mathfrak{so}_{2m+1}$.
\begin{enumerate}
\item {\bf Case 1B}: $\ \ \displaystyle \mathfrak{g} =\mathfrak{so}_{2n+2m+2},\quad \mathfrak{a} = \mathfrak{so}_{2n+1}$. 
\item {\bf Case 1C}: $\ \ \displaystyle\mathfrak{g} =\mathfrak{osp}_{2m+1|2n}, \quad \mathfrak{a} = \mathfrak{sp}_{2n}$.
\item{\bf Case 1D}: $\ \ \displaystyle\mathfrak{g}  =\mathfrak{so}_{2n+2m+1},  \quad \mathfrak{a} = \mathfrak{so}_{2n}$.
\item{\bf Case 1O}:  $\ \ \displaystyle\mathfrak{g} =\mathfrak{osp}_{2m+2|2n},  \quad \mathfrak{a} = \mathfrak{osp}_{1|2n}$.
\end{enumerate}
For $n\geq 0$ and $m\geq 1$ we have the following cases where $\mathfrak{b} = \mathfrak{sp}_{2m}$.
\begin{enumerate}
\item {\bf Case 2B}:  $\ \ \displaystyle\mathfrak{g} =\mathfrak{osp}_{2n+1|2m},\quad \mathfrak{a} = \mathfrak{so}_{2n+1}$.
\item{\bf Case 2C}: $\ \ \displaystyle\mathfrak{g} =\mathfrak{sp}_{2n+2m}, \quad \mathfrak{a} = \mathfrak{sp}_{2n}$.
\item{\bf Case 2D}:  $\ \ \displaystyle\mathfrak{g} =\mathfrak{osp}_{2n|2m}, \quad \mathfrak{a} = \mathfrak{so}_{2n}$.
\item{\bf Case 2O}:  $\ \ \displaystyle\mathfrak{g} =\mathfrak{osp}_{1|2n+2m},  \quad \mathfrak{a} = \mathfrak{osp}_{1 | 2 n}$. 
\end{enumerate}
  
It is convenient to replace the level $k$ with the critically shifted level $\psi = k+h^{\vee}$, where $h^{\vee}$ is the dual Coxeter number of $\gg$.
For $i = 1,2$ and $X = B,C,D,O$, we denote the corresponding $\cW$-algebra by $\cW^{\psi}_{iX}(n,m)$. In the cases $i = 1,2$ and $X = C$, we denote the corresponding affine cosets by $\cC^{\psi}_{iC}(n,m)$. In the cases $X = B,D,O$, there is an additional action of $\mathbb{Z}_2$, and $\cC^{\psi}_{iX}(n,m)$ denotes the $\mathbb{Z}_2$-orbifold of the affine coset. We will also define the algebras $\cW^{\psi}_{2X}(n,0)$ in a different way so that our results hold uniformly for all $n,m\geq 0$. 

Our main result is that there are four families of trialities among the algebras $\cC^{\psi}_{iX}(n,m)$.
\begin{mainthm} \label{main:intro} (Theorem \ref{main}) For all integers $m\geq n \geq 0$, we have the following isomorphisms of one-parameter vertex algebras.
\begin{equation} \label{2b2b2o:intro}  \cC^{\psi}_{2B}(n,m) \cong  \cC^{\psi'}_{2O}(n,m-n) \cong \cC^{\psi''}_{2B}(m,n), \qquad \psi' = \frac{1}{4\psi},  \qquad \frac{1}{\psi}  + \frac{1}{\psi''} = 2,
 \end{equation}
\begin{equation} \label{1c1c2c:intro} \cC^{\psi}_{1C}(n,m) \cong \cC^{\psi'}_{2C}(n,m-n) \cong  \cC^{\psi''}_{1C}(m,n), \qquad \psi'= \frac{1}{2\psi}, \qquad \frac{1}{\psi} + \frac{1}{\psi''} = 1,
\end{equation}
\begin{equation}  \label{2d1d1o:intro} \cC^{\psi}_{2D}(n,m) \cong  \cC^{\psi'}_{1D}(n,m-n ) \cong \cC^{\psi''}_{1O}(m,n-1),\qquad \psi' = \frac{1}{2\psi}, \qquad \frac{1}{2\psi} + \frac{1}{\psi''} = 1,
\end{equation}
\begin{equation} \label{1o1b2d:intro} \cC^{\psi}_{1O}(n,m) \cong \cC^{\psi'}_{1B}(n,m-n) \cong  \cC^{\psi''}_{2D}(m+1,n) ,\qquad \psi' = \frac{1}{\psi},  \qquad  \frac{1}{\psi}  + \frac{1}{2\psi''} = 1.
\end{equation}
\end{mainthm}

Special cases of this result include Feigin-Frenkel duality in types $B$, $C$, and $D$ \cite{FF2}, as well as a version for principal $\cW$-superalgebras of $\go\gs\gp_{1|2n}$ which was recently proven in \cite{CGe}. The special case 
$$\cC^{\psi}_{2D}(n,0) \cong \cC^{\psi''}_{1O}(0,n-1), \qquad \frac{1}{2\psi} + \frac{1}{\psi''} =1, \qquad n\geq 2$$ of \eqref{2d1d1o:intro} provides a new proof of the coset realization of principal $\cW$-algebras of type $D$ \cite{ACL}. 
The special case 
$$\cC^{\psi}_{2B}(n,0)  \cong \cC^{\psi''}_{2B}(0,n),\qquad \frac{1}{\psi}  + \frac{1}{\psi''} = 2$$ of \eqref{2b2b2o:intro} recovers the coset realization of the principal $\cW$-superalgebra of $\go\gs\gp_{1|2n}$ \cite{CGe}.
More importantly, the special case 
 $$\cC^{\psi}_{1C}(n,0) \cong  \cC^{\psi''}_{1C}(0,n), \qquad \frac{1}{\psi} + \frac{1}{\psi''} =1$$ of \eqref{1c1c2c:intro} provides a new coset realization of principal $\cW$-algebras of type $B$ (and type $C$ by Feigin-Frenkel duality), since
 \begin{equation*} \begin{split}  & \cC^{\psi}_{1C}(n,0) = \text{Com}(V^k(\gs\gp_{2n}), V^k(\go\gs\gp_{1|2n})),\qquad k = - \frac{1}{2} (\psi +2n +1).\\ & \cC^{\psi''}_{1C}(0,n) = \cW^{\psi''-2n+1}(\mathfrak{so}_{2n+1}).
\end{split}\end{equation*} This is quite different from the coset realizations of $\cW^k(\gg)$ for simply-laced $\gg$ given in \cite{ACL} since it involves affine vertex superalgebras. Finally, the special case
$$ \cC^{\psi}_{2D}(1,m) \cong \cC^{\psi'}_{1D}(1,m-1), \qquad  \psi' = \frac{1}{2\psi}$$ of \eqref{2d1d1o:intro} provides an alternative proof of the duality between the Heisenberg cosets of $\cW^{\psi'-2m+1}(\gs\go_{2m+1}, f_{\text{subreg}})$ and $\cW^{\psi-m}(\go\gs\gp_{2|2m})$ appearing in \cite{CGN}. 

The key idea in the proof of Main Theorem \ref{main:intro} is to identify all the algebras $\cC^{\psi}_{iX}(n,m)$ as one-parameter quotients of the universal even spin two-parameter $\cW_{\infty}$-algebra $\cW^{\text{ev}}(c,\lambda)$ constructed by Kanade and the second author in \cite{KL}. Such quotients of $\cW^{\text{ev}}(c,\lambda)$ are in bijection with a family of plane curves called truncation curves. By explicitly computing the truncation curves for these algebras,  Main Theorem \ref{main:intro} follows from symmetries of our formulas for the defining ideals. We thus identify four distinct $\mathbb{N} \times \mathbb{N}$ families of truncations of $\cW^{\text{ev}}(c,\lambda)$, which we conjecture to account for all of its truncations. In fact, we can give a uniform description of all these truncations by replacing the integer parameters $n$ and $m$ in one of the formulas by half-integers. We mention that the algebras $\cC^{\psi}_{iX}(n,m)$ were called orthosymplectic $Y$-algebras by Gaiotto and Rap\v{c}\'ak in \cite{GR}, and some of the trialities we prove were conjectured in \cite{GR}, as well as the paper \cite{Pro3} of Proch\'azka.

\subsection{Rationality results} As an application of Main Theorem \ref{main:intro}, we prove many new rationality results in Section \ref{sec:rationality}.

\begin{mainthm} \label{main:rational} Denote by $\cF(n)$ the vertex superalgebra of $n$ free fermions, by $L_k(\gg)$ the simple affine vertex superalgebra of the simple Lie superalgebra $\gg$ at level $k$, and by $\cW_\ell(\gg)$ the simple principal $\cW$-superalgebra of $\gg$ at level $\ell$. 
\begin{enumerate} 
\item  For all $n, k \in \mathbb Z_{\geq 1}$, $L_k(\go\gs\gp_{1|2n})$ is lisse and rational (Theorem \ref{thm:osp}).
\smallskip
\item For all $m\geq 1$,  $\cW_{\psi - m -1/2}(\go\gs\gp_{1|2m})$ is lisse and rational at the following levels:
\begin{enumerate} 
\item $\psi =  \frac{2 m-1}{4 (m + r)}$, where $m+r$ and $1+2r$ are coprime (Theorem \ref{thm:Wosp1}),
\item $\psi =  \frac{1 + 2 m}{2 (1 + 2 m + 2 r)}$, where $r$ and $1+2m$ are coprime (Theorem \ref{thm:Wosp1}),
\item $\psi =  \frac{m}{2 m + 2 r-1}$, where $2r-1$ and $2m$ are coprime (Theorem \ref{thm:Wosp2}).
\end{enumerate}
\smallskip
\item For all $m\geq 1$, $\cW_{\psi-2m-1}(\gs\go_{2m+3}, f_{\text{subreg}})$ is lisse and rational at the following levels:
\begin{enumerate}
\item $\psi = \frac{3 + 2 m + 2 r}{2m+2}$, where $m+1$ and $2r+1$ are coprime (Theorem \ref {Brational1}),
\item  $\psi = \frac{2 m + 2 r+1}{2m+1}$ where $r$ and $2m+1$ are coprime (Theorem \ref {Brational3}),
\item  $\psi = \frac{2 (2 + m)}{1 + 2 m}$ (Corollary \ref{Brational5}),
\item  $ \psi = \frac{2 m}{2 m-1}$ (Corollary \ref{cor:newrationalb}).
\end{enumerate}
\smallskip
\item For all $m\geq 1$, $\cW_{\psi-m}(\mathfrak{osp}_{2|2m+2})$ is lisse and rational at the following levels:
\begin{enumerate}
\item  $\psi = \frac{1 + m}{3 + 2 m + 2 r}$, where $m+1$ and $2r+1$ are coprime (Corollary \ref {Brational2}),
\item  $\psi = \frac{2 m+1}{2(2 m + 2 r+1)}$, where $r$ and $2m+1$ are coprime (Corollary \ref {Brational4}),
\item  $\psi = \frac{2 m+1}{4 (2 + m)}$ (Corollary \ref{Brational5}),
\item  $\psi = \frac{2 m-1}{4 m}$ (Corollary \ref{cor:newrationalb}).
\end{enumerate}
\smallskip
\item For all $n\geq 1$ and $r\geq 1$, $\cW_{r-1/2}(\gs\gp_{2n+2}, f_{\text{min}})$ is lisse and rational. 
\smallskip
\item For all $k,n,m \in \mathbb{Z}_{\geq 1}$ with $n > m$, the following cosets are lisse and rational:
\begin{enumerate}
\item $\text{Com}(L_{k-1/2}(\gs\gp_{2n}), L_k(\gs\gp_{2n}) \otimes  L_{-1/2}(\gs\gp_{2n}))$ (Theorem \ref{cosetC1}),
\item $\text{Com}( L_k(\gs\gp_{2n-2m}),  L_k(\gs\gp_{2n}))$ (Corollary \ref{cor:typeC-coset}),
 \item $\text{Com}\left(   L_n(\gs\gp_{2k}),  L_{n-m}(\gs\gp_{2k}) \otimes \cF(4mk) \right)$ (Corollary \ref{cosetC2}).
 \end{enumerate}
\end{enumerate}
\end{mainthm}

 Main Theorem \ref{main:rational} (1) completes the classification of lisse and rational affine vertex superalgebras $L_k(\gg)$, for $\gg$ a simple Lie superalgebra. When $\gg$ is a Lie algebra, it is a celebrated result of Frenkel and Zhu \cite{FZ} that $L_k(\gg)$ is lisse and rational if and only if $k \in \mathbb{N}$. When $\gg$ is not a Lie algebra, Gorelik and Kac \cite{GK} claimed that $L_k(\gg)$ is lisse only when $\gg = \go\gs\gp_{1|2n}$ and $k \in \mathbb{N}$; see \cite{AL} for the recent proof. However, the rationality was previously known only for $\go\gs\gp_{1|2}$ \cite{CFK}. The proof of Main Theorem \ref{main:rational} (1) involves exhibiting $L_k(\go\gs\gp_{1|2n})$ as an extension of the rational vertex algebra $L_k(\gs\gp_{2n}) \otimes \cW_{\ell}(\gs\gp_{2n})$ for $\ell = -(n + 1) + \frac{1 + k + n}{1 + 2 k + 2 n}$. 

A celebrated result of Arakawa \cite{Ar1, Ar2} is that for a simple Lie algebra $\gg$, $\cW_k(\gg)$ is lisse and rational when $k$ is a nondegenerate admissible level for $\widehat{\gg}$. Again, a similar statement is expected to be true for $\cW_k(\go\gs\gp_{1|2m})$. Based on the coset realization of $\cW^k(\go\gs\gp_{1|2m})$, we make the following conjecture.

\begin{conj} \label{conj:intro:ospcoset} (Conjecture \ref{conj:ospcoset}) For all $m \geq 1$, $\cW_{\psi -m -1/2}(\mathfrak{osp}_{1|2m})$ where $\psi= \frac{p}{2 (p+q)}$, is lisse and rational if 
\begin{enumerate}
\item $p,q \in \mathbb{N}$ are coprime,
\item $p\geq 2m-1$ if $q$ is odd,
\item $p\geq 2m$ if $q$ is even.
\end{enumerate}
By \eqref{2b2b2o:intro}, this conjecture implies that $\cW_{\psi' -m-1/2}(\mathfrak{osp}_{1|2m})$ is also lisse and rational at the Feigin-Frenkel dual level, where $\psi' = \frac{1}{4\psi}  = \frac{p + q}{2 p}$.
\end{conj} 
Main Theorem \ref{main:rational} (2) proves several cases of Conjecture \ref{conj:intro:ospcoset} by identifying $\cW_{\psi -m -1/2}(\mathfrak{osp}_{1|2m})$ in these cases with a simple current extension of a known rational vertex algebra of the form $\cW_s(\gs\gp_{2r})$ or $\cW_s(\gs\go_{2r})^{\mathbb{Z}_2}$. In the case $n=1$, where $\cW_k(\go\gs\gp_{1|2})$ is just the $N=1$ superconformal algebra, Conjecture \ref{conj:intro:ospcoset} is already known (\cite{KWan, Ad, M, BMRW}), and we give an alternative proof; see Theorem \ref{thm:osp12}.

Main Theorem \ref{main:rational} (3) proves several cases of the Kac-Wakimoto rationality conjecture \cite{KW4}, which was later refined by Arakawa \cite{Ar1}\footnote{In a recent preprint that appeared a few months after this paper was submitted, Robert McRae has proven the Kac-Wakimoto-Arakawa conjecture in full generality \cite[Main Theorem 4]{McRae2}.}. Let $\gg$ be a simple Lie algebra and $k = -h^{\vee} + \frac{p}{q}$ an admissible level for $\widehat{\gg}$. The associated variety of $L_k(\gg)$ is then the closure of a nilpotent orbit $\mathbb{O}_q$ which depends only on the denominator $q$. If $f\in \gg$ is a nilpotent lying in $\mathbb{O}_q$, the simple $\cW$-algebra $\cW_k(\gg,f)$ is known to be non-zero and lisse \cite{Ar1}. Such pairs $(f,q)$ are called {\it exceptional pairs}, and they generalize the notion of exceptional pair due to Kac and Wakimoto \cite{KW4} and Elashvili, Kac, and Vinberg \cite{EKV}. The corresponding $\cW$-algebras are also called exceptional, and were conjectured by Arakawa to be rational in \cite{Ar1}, generalizing the original conjecture of \cite{KW4}. Very recently, Arakawa and van Ekeren proved rationality of all exceptional $\cW$-algebras in type $A$, and all exceptional subregular $\cW$-algebras of simply-laced types \cite{AvE}.

The type $B$ subregular $\cW$-algebra $\cW_k(\gs\go_{2m+3}, f_{\text{subreg}})$ for $m\geq 1$ is exceptional when $k = -(2m+1) + \frac{p}{q}$ is admissible and $q = 2m+2$ or $2m+1$; see Table 1 of \cite{AvE}. Main Theorem \ref{main:rational} (3) proves rationality in all cases where $q = 2m+2$ and all cases where $q = 2m+1$ and $p$ is odd, generalizing the result for $m=1$ of Fasquel \cite{F}. In these cases, we identify $\text{Com}(\cH(1), \cW_k(\gs\go_{2m+3}, f_{\text{subreg}}))^{\mathbb{Z}_2}$ with a known rational vertex algebra of the form $\cW_s(\gs\gp_{2r})$ or $\cW_{s}(\gs\go_{2r})^{\mathbb{Z}_2}$. In the cases where $q = 2m+1$ and $p$ is even, $\text{Com}(\cH(1), \cW_k(\gs\go_{2m+3}, f_{\text{subreg}}))^{\mathbb{Z}_2}$ is identified with a vertex algebra of the form $\cW_s(\go\gs\gp_{1|2r})^{\mathbb{Z}_2}$. Using the methods of this paper, we are only able to prove rationality of $\cW_k(\gs\go_{2m+3}, f_{\text{subreg}})$ when $r=1$. However, the rationality for all such $p$ and $q$ is a special case of McRae's result \cite{McRae2}. As a consequence, we obtain further examples where Conjecture \ref{conj:intro:ospcoset} is true; see Remark \ref{rem:morerationalosp}. Note that Main Theorem \ref{main:rational} (4) follows immediately from part (3) together with the duality between subregular $\cW$-algebras of type $B$ and principal $\cW$-superalgebras of $\go\gs\gp_{2|2m}$ appearing in \cite{CGN}.

It turns out that the cases $q = 2m+1$ and $q = 2m+2$ do not account for all rational algebras of the form $\cW_{k}(\gs\go_{2m+3}, f_{\text{subreg}})$. We will show that 
$$\text{Com}(\cH(1), \cW_k(\gs\go_{2m+3}, f_{\text{subreg}}))^{\mathbb{Z}_2}\cong \cW_{s}(\go\gs\gp_{1|2r})^{\mathbb{Z}_2},$$ for
$k = -(2m+1) + \frac{2 (m - r +1)}{1 + 2 m - 2 r}$ and $s = -(r+\frac{1}{2}) +  \frac{m + 1 - r}{2m +1 -2 r }$. In the case $m\geq 2r-1$, Conjecture \ref{conj:intro:ospcoset} would then imply the rationality of $\cW_k(\gs\go_{2m+3}, f_{\text{subreg}})$ even though $k$ is not an admissible level; see Remark \ref{rem:newtypebsubreg}. In the case $r=1$, these examples are rational by Main Theorem \ref{main:rational} (4d). These are new examples of rational $\cW$-algebras that are not covered by the Kac-Wakimoto-Arakawa conjecture, and it is an interesting problem to find more such examples and ultimately to classify them.

The rationality of $\cW_{r-1/2}(\gs\gp_{2n+2}, f_{\text{min}})$ for $r \in \mathbb{Z}_{\geq 1}$ given by Main Theorem \ref{main:rational} (5), is another case of the Kac-Wakimoto-Arakawa rationality conjecture. We will show that $\cW_{r-1/2}(\gs\gp_{2n+2}, f_{\text{min}})$ is an extension of the rational vertex algebra $L_r(\gs\gp_{2n}) \otimes \cW_s(\gs\gp_{2r})$ for $s = -(r+1) + \frac{1+n+r}{3+2n+2r}$. This was conjectured in \cite{ACKL} and proven in \cite{KL} in the case $n=1$. 

Main Theorem \ref{main:rational} (6a) is proven by showing that $\text{Com}(L_{k-1/2}(\gs\gp_{2n}), L_k(\gs\gp_{2n}) \otimes  L_{-1/2}(\gs\gp_{2n}))$ is isomorphic to $\cW_{\ell} (\gs\gp_{2k})$ for $\ell = -(k + 1) + \frac{1 + n + k}{1 + 2 n + 2 k}$, which is a new level-rank duality. Similarly, Main Theorem \ref{main:rational} (6b) and (6c) are proven by showing that both cosets are extensions of the rational vertex algebra $$ \bigotimes_{i=1}^m \left(\cW_{\ell_i} (\gs\gp_{2k}) \otimes \cW_{s_i} (\gs\gp_{2k})\right) $$ with $ \ell_i  = -(k + 1) + \frac{2 + n - i + k}{3+ 2 n -2i  + 2 k}$ and $s_i = -(k+1) + \frac{1+ n - i +k}{3+2n -2i +2k}$. Note that (6b) implies that $L_k(\gs\gp_{2n})$ is an extension of $ \bigotimes_{i=1}^n \left(\cW_{\ell_i} (\gs\gp_{2k}) \otimes \cW_{s_i} (\gs\gp_{2k})\right)$ with $\ell_i, s_i$ as above. This is analogous to the statement in \cite{ACL} that $L_k(\gg\gl_n)$ is an extension of $\bigotimes_{i=1}^n \cW_{\ell_i}(\gg\gl_k)$ with $\ell_i =- k = \frac{k+n-i}{k+n-i+1}$, which is an analogue of the Gelfand-Tsetlin subalgebra of $U(\gg\gl_n)$.

\subsection{Triality from kernel vertex algebras}
Motivated from four-dimensional GL-twisted $N=4$ supersymmetric Yang-Mills theories, certain kernel vertex algebras in type $A$ were conjecturally introduced in \cite{CGa}. Here, we will introduce analogues for orthosymplectic type and explain how this provides another perspective on triality. The kernel vertex algebras also play an important role in the context of the quantum geometric Langlands program and our conjectures are closely related to the ideas sketched in Section 10.4 of \cite{FG}.

Let $\gg$ be either a simple Lie algebra of type $B, C, D$ or $\go\gs\gp_{1|2n}$\footnote{The case $\gg=\gs\go_2$ can also be included and we refer to \cite{CGNS} for discussion of the kernel vertex algebra and its relative semi-infinite cohomology.}. Let $P^+$ denote the set of dominant weights of $\gg$, and $R \subseteq P^+$ the subset corresponding to the tensor ring generated by the standard representation of $\gg$; that is, 
 $\lambda \in R$ if and only if the irreducible highest-weight representation $\rho_\lambda$ is a submodule of some iterated tensor product of the standard representation of $\gg$.
Let $\gg' = \gs\go_{2n+1}$ if $\gg = \go\gs\gp_{1|2n}$ and vice versa, and let $\gg' = \gg$ otherwise. 
Note that there is a one-to-one correspondence of irreducible finite dimensional non-spin representations of $\gs\go_{2n+1}$ and $\go\gs\gp_{1|2n}$, such that characters agree. A similar statement also holds for the quantum (super)groups \cite{CFLW}. Let $\tau$ denote the induced map on dominant weights. If $\gg$ is neither of type $\gs\go_{2n+1}$ nor of type $\go\gs\gp_{1|2n}$ then let $\tau$ be the identity on $P^+$. Motivated from Theorem \ref{main} we let $\phi$ be generic and $\phi'$ related to $\phi$ by the formula 
\begin{equation} \nonumber
\begin{split}
\frac{1}{\phi} + \frac{1}{\phi'} &= 2, \qquad \text{if} \ \gg \ \text{is of type} \ C, \\
\frac{1}{\phi} + \frac{1}{\phi'} &= 1, \qquad \text{if} \ \gg \ \text{is of type}\  D, \\
\frac{1}{2\phi} + \frac{1}{\phi'} &= 1, \qquad \text{if} \ \gg \ \text{is of type} \  \go\gs\gp_{1|2n}, \\
\frac{1}{\phi} + \frac{1}{2\phi'} &= 1, \qquad \text{if} \ \gg \ \text{is of type}\ B.
\end{split}
\end{equation}
Let $V^{\phi - h^\vee_\gg}(\gg)$ and $V^{\phi' - h^\vee_{\gg'}}(\gg')$ be the universal affine vertex (super)algebras of $\gg$ and $\gg'$ at levels $\phi - h^\vee_\gg$ and $\phi' - h^\vee_{\gg'}$. Let $M^\phi(\lambda) = V^{\phi - h^\vee_\gg}(\lambda)$ and $M^{\phi'}(\lambda) = V^{\phi' - h^\vee_{\gg'}}(\tau(\lambda))$ be the Weyl modules at these levels whose top level is the irreducible highest-weight representation of $\gg$ of highest weight $\lambda$, respectively of $\gg'$ of highest weight $\tau(\lambda)$. 
Then set
\[
A[\gg, \phi] := \bigoplus_{\lambda \in R} M^\phi(\lambda) \otimes M^{\phi'}(\lambda).
\]
More generally, let $f, f'$ be nilpotent elements in $\gg, \gg'$ and let $M^\phi_f(\lambda)$ and $M^{\phi'}_{f'}(\lambda)$ be the images under quantum Hamiltonian reduction of $M^\phi(\lambda)$ and $M^{\phi'}(\lambda)$ corresponding to $f$ and $f'$, respectively. Then set 
\[
A[\gg, \phi, f, f'] := \bigoplus_{\lambda \in R} M_f^\phi(\lambda) \otimes M_{f'}^{\phi'}(\lambda), 
\]
so that $A[\gg, \phi] = A[\gg, \phi, 0, 0]$. The conjecture is that these objects can be given the structure of a simple vertex superalgebra for generic $\phi$. 
In the case that $\gg$ is of type $D$, $f =0$, and $f'$ is a principal nilpotent, this is the coset theorem of type $D$ of \cite{ACL}. Moreover, the case of arbitrary $f$ and $f'$ the principal nilpotent is the main theorem of \cite{ACF} applied to the coset theorem. In that paper also many similar algebras are studied. 
Note that in order for $\gg = \gs\go_{3}$ to fit in the $\gs\go_{2n+1}$-series, we use the Killing form rescaled by two as bilinear form, i.e., in our convention $V^k(\gs\go_{3}) = V^{2k}(\gs\gl_2) = V^{2k}(\gs\gp_2)$. This can be thought of as to setting the dual Coxeter number of $\gs\go_{3}$ to one.

For $f=f'=0$, the case of $\gg =\gs\go_3$ and $\gg = \go\gs\gp_{1|2}$ is proven in \cite{CGa} and the case of  $\gg = \gs\gp_2$ in \cite{CGaL} and these are the affine vertex superalgebra of $\gd(2, 1; 1- \phi)$ at level one, respectively, the minimal quantum Hamiltonian reduction of $\gd(2, 1; (1- \phi)/2)$ at level $1/2$, which is the one-parameter family of 
 large $N=4$ superconformal algebras at central charge $-6$.

Cosets can often also be characterized as relative semi-infinite Lie algebra cohomologies. It seems that the cohomology approach is suitable to put our trialities into a more general perspective. 
Let $\gg$ be a simple Lie algebra, $B$ a basis for $\gg$ and $B'$ a dual basis. Let $\mathcal F(\gg)$ be two copies of free fermions in the adjoint representation of $\gg$ with generators $\{ b^x, c^{x'} \,  | \, x\in B, \, x'\in B' \}$ and operator products
$b^x(z) c^{y'}(w) \sim \delta_{x, y} (z-w)^{-1}$.
Consider the affine vertex algebra of $\gg$ at level $-2h^\vee$,  $V^{-2h^\vee}(\gg)$, and let $x(z)$ be the field corresponding to $x\in \gg$. Let  $d:=d_0$ be the zero-mode  of the field
\[
d(z) := \sum_{x\in  B} :x(z) c^{x'}(z): - \frac{1}{2} \sum_{x, y \in  B} :(:b^{[x, y]}(z) c^{x'}(z):)c^{y'}(z):,
\] 
which squares to zero. Let $\widetilde{\mathcal F}(\gg)$ denote the subalgebra of $\mathcal F(\gg)$ generated by the $b^x$ and $\partial c^{x'}$ (these are just $\dim \gg$ pairs of symplectic fermions).
For a module $M$  for $V^{-2h^\vee}(\gg)$
the relative complex is
\[
C^{\text{rel}}(\gg, d) =  \left(M\otimes \widetilde{\mathcal F}(\gg)\right)^\gg
\]
and this relative complex is preserved by $d$  \cite[Prop. 1.4.]{FGZ}. The cohomology is denoted by $H^{\text{rel},\bullet }_{\infty}(\gg, M)$. 
As shown in \cite{FGZ} and explained in Section 2.5 of \cite{CFL}, it satisfies
\begin{equation}\label{eq:relcoho}
H^{\text{rel}, 0 }_{\infty}(\gg, V^k(\lambda) \otimes V^{-2h^\vee-k}(\mu)) = \begin{cases} \mathbb C & \ \text{if} \ \mu=-\omega_0(\lambda), \\ 0 & \ \text{otherwise} .\end{cases}
\end{equation}
Here $\omega_0$ is the unique Weyl group element that interchanges the fundamental Weyl chamber with its negative.
It is reasonable to expect that one can construct similar complexes with similar properties for affine vertex superalgebras and it would be very interesting to do so for at least $\gg = \go\gs\gp_{1|2n}$. 
In order to include the case $n=0$, we define master chiral algebra and cohomology to be trivial, that is $\gs\go_0 := \gs\go_1  := \gs\gp_0  := \go\gs\gp_{1|0} = \{ \ \ \}$ and 
$A[ \{ \ \ \}, \phi] :=\mathbb C$  as well as $H^{\text{rel}, 0 }_{\infty}(\{ \ \ \}, \mathbb C) =\mathbb C$. 
\begin{conj} 
With the above set-up, and $f, f'$ nilpotent in $\gg, \gg'$, respectively:
\begin{enumerate}
\item 
The object $A[\gg, \phi, f, f']$ can be given the structure of a one-parameter vertex superalgebra.
\item For generic $\phi$, $A[\gg, \phi, f, f']$ is a simple vertex operator superalgebra
 extending  $\cW^{\phi- h^\vee_\gg}(\gg, f) \otimes \cW^{\phi'- h^\vee_{\gg'}}(\gg', f')$. 
\item There exists a generalization of relative semi-infinite Lie superalgebra cohomology for $\gg = \go\gs\gp_{1|2n}$ satisfying \eqref{eq:relcoho}.
\item 
 For all integers $m\geq n \geq 0$, we have the following isomorphisms of one-parameter vertex algebras.
 \begin{enumerate}
\item For $\psi' = \frac{1}{4\psi}$ and $\frac{1}{\psi} + \frac{1}{\psi''} = 2$,
\begin{equation}  \nonumber
\begin{split}
\cW^\psi_{2B}(n, m) &\cong H^{\text{rel}, 0 }_{\infty}\left(\go\gs\gp_{1|2n}, \cW^{\psi'}_{2O}(n,m-n )  \otimes A[\go\gs\gp_{1|2n}, \frac{1}{2}- \psi']\right) ,  \\
\cW^{\psi''}_{2B}(m, n) &\cong \text{Com}\left( V^{\psi' -n-1}(\go\gs\gp_{1|2n}), A[\go\gs\gp_{1|2m}, \psi', f_{\gs\gp_{2m-2n}}, 0]\right).
\end{split}
\end{equation} 
\item 
For $\psi' = \frac{1}{2\psi}$ and $\frac{1}{\psi} + \frac{1}{\psi''} = 1$,
\begin{equation}  \nonumber
\begin{split}
\cW^\psi_{1C}(n, m) &\cong H^{\text{rel}, 0 }_{\infty}\left(\gs\gp_{2n}, \cW^{\psi'}_{2C}(n,m-n )  \otimes A[\gs\gp_{2n}, \frac{1}{2}- \psi']\right) ,  \\
\cW^{\psi''}_{1C}(m, n) &\cong \text{Com}\left( V^{\psi' -n-\frac{3}{2}}(\gs\gp_{2n}), A[\gs\gp_{2m}, \psi', f_{\gs\gp_{2m-2n}}, 0]\right).
\end{split}
\end{equation} 
\item For $\psi' = \frac{1}{2\psi}$ and $\frac{1}{2\psi} + \frac{1}{\psi''} = 1$,
\begin{equation}  \nonumber
\begin{split}
\cW^\psi_{2D}(n, m) &\cong H^{\text{rel}, 0 }_{\infty}\left(\gs\go_{2n}, \cW^{\psi'}_{1D}(n,m-n )  \otimes A[\gs\go_{2n}, 1- \psi']\right) ,  \\
\cW^{\psi''}_{1O}(m, n) &\cong \text{Com}\left( V^{\psi' -2n+1}(\gs\go_{2n}), A[\gs\go_{2m+1}, \psi', f_{\gs\go_{2m-2n+1}}, 0]\right).
\end{split}
\end{equation}
\item 
For $\psi' = \frac{1}{\psi}$ and $\frac{1}{\psi} + \frac{1}{2\psi''} = 1$ and $m>n$ in the second case,
\begin{equation}  \nonumber
\begin{split}
\cW^\psi_{1O}(n, m) &\cong H^{\text{rel}, 0 }_{\infty}\left(\gs\go_{2n+1}, \cW^{\psi'}_{1B}(n,m-n )  \otimes A[\gs\go_{2n+1}, 1- \psi']\right) ,  \\
\cW^{\psi''}_{2D}(m, n) &\cong \text{Com}\left( V^{\psi' -2n}(\gs\go_{2n+1}), A[\gs\go_{2m}, \psi', f_{\gs\go_{2m-2n-1}}, 0]\right).
\end{split}
\end{equation} 
\end{enumerate}
\end{enumerate}
\end{conj}
Both sides of the conjectured isomorphisms have the same affine vertex superalgebra and the same quotient of $\cW^{\text{ev}}(c,\lambda)$ as commuting pair of subalgebras.

The relative semi-infinite cohomology part of (4) of the conjecture for $n=0$ is Feigin-Frenkel duality \cite{FF2}. 
The coset part of (4) should be viewed as a generalization of coset realization of principal $\cW$-algebras, e.g. the case $n=0$ of (4)(a) corresponds to the coset realization of prinicpal $\cW$-superalgebras of $\go\gs\gp_{1|2n}$ of \cite{CGe} and the case (4)(d) to the coset realization of principal $\cW$-algebras of type $D$ of \cite{ACL}. Our conjecture is a natural extension to orthosymplectic type of our conjectures for type $A$ made in Section 10 of \cite{CL3}. The relative semi-infinite cohomology part of this conjecture is proven for subregular $\cW$-algebras of type $A$ and $B$ \cite{CGNS}.
We sketched a proof strategy in Section 10 of \cite{CL3} for type $A$, and we expect that all our conjectures for type $A$ as well as orthosymplectic type can be proven uniformly. 

Besides its importance in physics and quantum geometric Langlands, our conjectures provide a way to relate representation categories in a nice way. For example, if Conjecture 4 (a) is correct, then the functor $H^{\text{rel}, 0 }_{\infty}\left(\go\gs\gp_{1|2n},   \ ? \  \otimes A[\go\gs\gp_{1|2n}, \frac{1}{2}- \psi']\right)$ maps $\cW^{\psi'}_{2O}(n,m-n )$-modules to 
$\cW^\psi_{2B}(n, m)$-modules. It is reasonable that such a functor has nice monoidal properties. Indeed in the type $A$ and subregular $\cW$-algebra case, this is true, namely there is a block-wise equivalence of categories and an isomorphism between the superspaces of logarithmic intertwining operators  \cite{CGNS}. 
Also the coset part of the conjecture should be useful to connect representation categories. For example, the coset realization of principal $\cW$-algebras of type $ADE$ has been used to prove a braided monoidal equivalence between a simple current twist of ordinary modules of the affine vertex algebra at admissible level, and a subcategory of modules of the principal $\cW$-algebra appearing as the coset; see \cite[Thm. 7.1]{C} for the precise statement. In the case $\gg = \gs\gl_2$ the admissible level result appeared in \cite[Thm. 7.4]{CHY}, and a theorem at generic level was also established in \cite[Prop. 5.5.2]{CJORY}. Note that these results prove variants of a conjecture made in the context of quantum geometric Langlands, see \cite[Conj. 6.3]{AFO}. It is work in progress to study the categories of ordinary modules of $L_k(\go\gs\gp_{1|2n})$ at admissible level, and especially to show that they are braided equivalent to subcategories of the principal $\cW$-algebras of type $B$ that appear as the cosets. 

\subsection{Geometry and conformal field theory}

The Alday-Gaiotto-Tachikawa (AGT) correspondence is a relation between four-dimensional gauge theories and two-dimensional conformal field theories \cite{AGT}. It yields interesting connections to geometry, for example a celebrated result of Schiffmann and Vasserot \cite{SV} asserts that the principal $\cW$-algebra of type $A$ acts on the equivariant cohomology of the moduli space of instantons on
 $\mathbb C^2$. There is a generalization to equivariant cohomology of Uhlenbeck spaces where  principal $\cW$-algebras of simply-laced type act \cite{BFN}, and the authors expect that their construction can be generalized to the non simply-laced case. This is interesting as conjecturally the $Y_{M, N, L}$-algebras of Gaiotto and Rap\v{c}\'ak, that is the cosets of $\cW$-superalgebras of the triality of type $A$, act on moduli spaces of spiked instantons \cite{RSYZ}, and we hope that a nice geometric interpretation also exists in the orthosymplectic case. Note that at least a nice four-dimensional physics interpretation exists for them \cite{GR}.
  In all the above mentioned works, it has been shown that a subalgebra of a Heisenberg vertex algebra, characterized as the kernel of certain screening operators, acts on an equivariant cohomology. Another crucial problem is thus to develop explicit screening realizations. Naoki Genra has already provided nice screening realizations of $\cW$-superalgebras \cite{Ge} and the main obstacle of making them more explicit are screening realizations of affine vertex superalgebras. 
    Screening realizations should also provide a different proof of the trialities with the advantage that it should give further insights on the connection of representation theories. 
 From the physics perspective the screening charges provide the interaction term in the action of the conformal field theory. A duality of conformal field theories is then an isomorphism of symmetry algebras, that is underlying vertex algebras, together with a matching of correlation functions. Our trialities in type $A$ at low rank have already led to new dualities of conformal field theories \cite{CH}  and likely there are more to be discovered.

\subsection{Outline} This paper is organized as follows. In Section \ref{sec:VOA} we review the basic terminology and examples of vertex algebras that we need. We also prove that extensions of a rational vertex algebra are rational under hypotheses that hold in all our examples. In Section \ref{sec:hooktype}, we introduce the $\cW$-(super)algebras $\cW^{\psi}_{iX}(n,m)$ for $i = 1,2$ and $X = B,C,D,O$ that we need. In Section \ref{sec:mainresult}, we state our main result and also discuss the special cases which recover Feigin-Frenkel duality and various coset realizations of principal $\cW$-algebras. In Section \ref{sec:largelevel} we discuss the free field limits of $\cW^{\psi}_{iX}(n,m)$ and the strong generating types of the algebras $\cC^{\psi}_{iX}(n,m)$. In Section \ref{sec:proofmain} we prove Main Theorem \ref{main:intro} by explicitly computing the truncation curves realizing $\cC^{\psi}_{iX}(n,m)$ as quotients of $\cW^{\text{ev}}(c,\lambda)$. In Section \ref{sec:rationality} we prove Main Theorem \ref{main:rational}. In Appendix \ref{appendixA}, we give the explicit truncation curve for $\cC^{\psi}_{2B}(n,m)$, from which all other truncation curves can be derived. Finally, in Appendices \ref{appendixB}, \ref{appendixC}, and \ref{appendixD}, we classify the pointwise coincidences between the simple quotients $\cC_{\psi,iX}(n,m)$, and the algebras $\cW_s(\gs\gp_{2r})$, $\cW_s(\gs\go_{2r})^{\mathbb{Z}_2}$, and $\cW_s(\go\gs\gp_{1|2r})^{\mathbb{Z}_2}$. These coincidences are needed in the proof of Main Theorem \ref{main:rational}.

\section{Vertex algebras} \label{sec:VOA} We will assume that the reader is familiar with vertex algebras, and we use the same notation as our previous paper \cite{CL3}. In this section, we briefly recall the definition and basic properties of free field algebras, $\cW$-algebras, and the two-parameter even spin algebra $\cW^{\text{ev}}(c,\lambda)$. We then prove some general results on extensions of rational vertex algebras which are needed in the proof of Main Theorem \ref{main:rational}.

\subsection{Free field algebras} 
Recall that a free field algebra is a vertex superalgebra $\cV$ with weight grading
$$\cV = \bigoplus_{d \in \frac{1}{2} \mathbb{Z}_{\geq 0} }\cV[d],\qquad \cV[0] \cong \mathbb{C},$$ with strong generators $\{X^i|\ i \in I\}$ satisfying OPE relations
$$ X^i(z) X^j(w) \sim a_{i,j} (z-w)^{-\text{wt}(X^i) - \text{wt}(X^j)},\quad a_{i,j} \in \mathbb{C},\ \quad a_{i,j} = 0\ \ \text{if} \ \text{wt}(X^i) +\text{wt}(X^j)\notin \mathbb{Z}.$$
Note $\cV$ is not assumed to have a conformal structure. We now recall the four families of standard free field algebras that were introduced in \cite{CL3}.

\smallskip

\noindent {\it Even algebras of orthogonal type}. For each $n\geq 1$ and even $k \geq 2$, $\cO_{\text{ev}}(n,k)$ is the vertex algebra with even generators $a^1,\dots, a^n$ of weight $\frac{k}{2}$, which satisfy
$$a^i(z) a^j(w) \sim \delta_{i,j} (z-w)^{-k}.$$ 
In the case $k=2$, $\cO_{\text{ev}}(n,k)$ is just the rank $n$ Heisenberg algebra $\cH(n)$. Note that $\cO_{\text{ev}}(n,k)$ has no conformal vector for $k>2$, but for all $k$ it is a simple vertex algebra and has full automorphism group the orthogonal group $\text{O}_n$.

\smallskip

\noindent {\it Even algebras of symplectic type}. For each $n\geq 1$ and odd $k \geq 1$,  $\cS_{\text{ev}}(n,k)$ is the vertex algebra with even generators $a^i, b^i$ for $i=1,\dots, n$ of weight $\frac{k}{2}$, which satisfy \begin{equation} \begin{split} a^i(z) b^{j}(w) &\sim \delta_{i,j} (z-w)^{-k},\qquad b^{i}(z)a^j(w)\sim -\delta_{i,j} (z-w)^{-k},\\ a^i(z)a^j(w) &\sim 0,\qquad\qquad\qquad \ \ \ \ b^i(z)b^j (w)\sim 0.\end{split} \end{equation} In the case $k=1$, $\cS_{\text{ev}}(n,k)$ is just the rank $n$ $\beta\gamma$-system $\cS(n)$. For $k>1$, $\cS_{\text{ev}}(n,k)$ has no conformal vector, but for all $k$ it is simple and has full automorphism group the symplectic group $\text{Sp}_{2n}$.

\smallskip

\noindent {\it Odd algebras of symplectic type}. For each $n\geq 1$ and even $k \geq 2$, $\cS_{\text{odd}}(n,k)$ is the vertex superalgebra with odd generators $a^i, b^i$ for $i=1,\dots, n$ of weight $\frac{k}{2}$, which satisfy
\begin{equation} \begin{split} a^{i} (z) b^{j}(w)&\sim \delta_{i,j} (z-w)^{-k},\qquad b^{j}(z) a^{i}(w)\sim - \delta_{i,j} (z-w)^{-k},\\ a^{i} (z) a^{j} (w)&\sim 0,\qquad\qquad\qquad\ \ \ \  b^{i} (z) b^{j} (w)\sim 0. \end{split} \end{equation} In the case $k=2$, $\cS_{\text{odd}}(n,k)$ is just the rank $n$ symplectic fermion algebra $\cA(n)$. Note that $\cS_{\text{odd}}(n,k)$ has no conformal vector for $k>2$, but for all $k$ it is simple and has full automorphism group $\text{Sp}_{2n}$.

\smallskip

\noindent {\it Odd algebras of orthogonal type}. For each $n\geq 1$ and odd $k \geq 1$, we define $\cO_{\text{odd}}(n,k)$ to be the vertex superalgebra with odd generators $a^i$ for $i=1,\dots, n$ of weight $\frac{k}{2}$, satisfying
\begin{equation}a^i(z) a^j(w) \sim \delta_{i,j} (z-w)^{-k}.\end{equation} For $k=1$, $\cO_{\text{odd}}(n,k)$ is just the free fermion algebra $\cF(n)$. As above, $\cO_{\text{odd}}(n,k)$ has no conformal vector for $k>1$, but it is simple and has full automorphism group $\text{O}_n$.

\subsection{$\cW$-algebras}\label{sec:W-algebras} 
Let $\gg$ be a simple, finite-dimensional Lie (super)algebra equipped with a nondegenerate, invariant (super)symmetric bilinear form $( \ \ | \ \ )$, and let $f$ be a nilpotent element in the even part of $\gg$. Associated to $\gg$ and $f$ and any complex number $k$, is the $\cW$-(super)algebra $\cW^k(\gg,f)$. The definition is due to Kac, Roan, and Wakimoto \cite{KRW}, and it generalizes the definition for $f$ a principal nilpotent and $\gg$ a Lie algebra given by Feigin and Frenkel \cite{FF1}.

First, let $\{q^\alpha\}_{\alpha \in S}$ be a basis of $\gg$ which is homogeneous with respect to parity. We define the corresponding structure constants and parity by
$$
[q^\alpha, q^\beta] = \sum_{\gamma \in S}{f^{\alpha\beta}}_\gamma q^\gamma,\qquad {|\alpha|} = \begin{cases} 0 & \ q^\alpha \ \text{even}, \\ 1 & \ q^\alpha \ \text{odd}. \end{cases}$$
The affine vertex algebra of $\gg$ associated to the bilinear form $( \ \ | \ \ )$  at level $k$ is strongly generated by $\{X^\alpha\}_{\alpha \in S}$  with OPEs
$$
X^\alpha(z)X^\beta(w) \sim k(q^\alpha|q^\beta) (z-w)^{-2} + \sum_{\gamma\in S}{f^{\alpha \beta}}_\gamma X^\gamma (w)(z-w)^{-1}.
$$
We define $X_\alpha$ to be the field corresponding to $q_\alpha$ where $\{q_\alpha\}_{\alpha \in S}$ is the dual basis of $\gg$ with respect to $( \ \ | \ \  )$.

Let $f$ be a nilpotent element in the even part of $\gg$, which we complete to an $\gs\gl_2$-triple $\{f, x, e\} \subseteq \gg$ satisfying $[x, e] =e, [x, f]=-f, [e, f]= 2x$. Then $\gg$ decomposes as
$$\gg = \bigoplus_{k \in  \frac{1}{2}\mathbb Z} \gg_k, \qquad \gg_k = \{  a \in \gg | [x, a] = ka \}. $$
Write $S = \bigcup_k S_k$ and $S_+ = \bigcup_{k>0} S_k$, where $S_k$ corresponds to a basis of $\gg_k$.

As in \cite{KW3}, one defines a complex $C(\gg, f, k) = V^k(\gg) \otimes F(\gg_+) \otimes F(\gg_{\frac{1}{2}})$, where $F(\gg_+)$ is a free field superalgebra associated to the vector superspace $\gg_+ = \bigoplus_{k \in  \frac{1}{2}\mathbb Z_{>0}} \gg_k$, and $F(\gg_{\frac{1}{2}})$ is the neutral vertex superalgebra associated to $\gg_{\frac{1}{2}}$ with bilinear form $\langle a, b \rangle = ( f | [a, b] )$. 
$F(\gg_+)$  is strongly generated by fields $\{\varphi_\alpha, \varphi^\alpha\}_{\alpha \in S_+}$, where $\varphi_\alpha$ and $\varphi^\alpha$ have opposite parity to $q^\alpha$. The operator products are
\[
\varphi_\alpha(z) \varphi^\beta(w) \sim \delta_{\alpha, \beta}(z-w)^{-1}, \qquad \varphi_\alpha(z) \varphi_\beta(w) \sim 0 \sim \varphi^\alpha(z) \varphi^\beta(w).
\]
$F(\gg_{\frac{1}{2}})$   is strongly generated by fields $\{\Phi_\alpha\}_{\alpha \in S_{\frac{1}{2}}}$ and $\Phi^\alpha$ and $q^\alpha$ have the same parity. Their operator products are
\begin{equation} \nonumber
\Phi_\alpha(z) \Phi_\beta(w) \sim \langle q^\alpha , q^\beta \rangle (z-w)^{-1}  \sim (f | [q^\alpha, q^\beta])(z-w)^{-1}.
\end{equation}
There is a $\mathbb{Z}$-grading on $C(\gg, f, k)$ by charge, and a weight one odd field $d(z)$ of charge minus one, 
\begin{equation}
\begin{split}
d(z) &= \sum_{\alpha \in S_+} (-1)^{|\alpha|} :X^\alpha \varphi^\alpha:  
-\frac{1}{2} \sum_{\alpha, \beta, \gamma \in S_+}  (-1)^{|\alpha||\gamma|} {f^{\alpha \beta}}_\gamma :\varphi_\gamma \varphi^\alpha \varphi^\beta: + \\
&\quad  \sum_{\alpha \in S_+} (f | q^\alpha) \varphi^\alpha + \sum_{\alpha \in S_{\frac{1}{2}}} :\varphi^\alpha \Phi_\alpha:,
\end{split} 
\end{equation}
whose zero-mode $d_0$ is a square-zero differential on $C(\gg, f, k)$. The $\cW$-algebra $\cW^k(\gg, f)$ is defined to be the homology $H(C(\gg, f, k), d_0)$. It has Virasoro element $L =  L_{\text{sug}} + \partial x + L_{\text{ch}} + L_{\text{ne}}$, where \begin{equation} \begin{split} L_{\text{sug}} &=  \frac{1}{2(k+h^\vee)} \sum_{\alpha \in S} (-1)^{|\alpha|} :X_\alpha X^\alpha:,\\ L_{\text{ch}} &= \sum_{\alpha \in S_+} \left(-m_\alpha :\varphi^\alpha \partial \varphi_\alpha: + (1-m_\alpha) :(\partial\varphi^\alpha )\varphi_\alpha:  \right),\\ L_{\text{ne}} &= \frac{1}{2} \sum_{\alpha \in S_{\frac{1}{2}}} :(\partial \Phi^\alpha) \Phi_\alpha : . \end{split} \end{equation}
Here $m_\alpha=j$ if $\alpha \in S_j$. The central charge of $L$ is computed to be

\begin{equation}\label{eq:c}
c(\gg, f, k) = \frac{k\, \text{sdim}\, \gg}{k+h^\vee} -12 k(x|x) -\sum_{\alpha \in S_+} (-1)^{|\alpha|} (12m_\alpha^2-12m_\alpha +2) -\frac{1}{2}\, \text{sdim}\, \gg_{\frac{1}{2}}.
\end{equation}

Denote by $\gg^f$ the centralizer of $f$ in $\gg$, and let $\ga = \gg^f \cap \gg_0$, which is a Lie subsuperalgebra of $\gg$. By \cite[Thm. 2.1]{KW3}, $\cW^k(\gg, f)$  contains an affine vertex superalgebra of type $\ga$. In particular, $\cW^k(\gg, f)$ contains elements $I^{\alpha}$ for $\alpha \in \gg_0$, 
\begin{equation}\label{eq:Ialpha}
I^\alpha :=  X^\alpha + \sum_{\beta, \gamma \in S_+} (-1)^{|\gamma|} {f^{\alpha \beta }}_\gamma :\varphi_\gamma\varphi^\beta:  + \frac{(-1)^{|\alpha|}}{2} \sum_{\beta \in S_{\frac{1}{2}}}  {f^{\beta\alpha}}_\gamma \Phi^\beta\Phi_\gamma
\end{equation}
and satisfying
\begin{equation} [I^\alpha {}_\lambda I^\beta] = {f^{\alpha \beta}}_\gamma I^\gamma + \lambda\left(k(q^\alpha | q^\beta) + \frac{1}{2}\left(\kappa_\gg(q^\alpha, q^\beta) -\kappa_{\gg_0}(q^\alpha, q^\beta)-\kappa_{\frac{1}{2}}(q^\alpha, q^\beta) \right) \right),\end{equation}
with $\kappa_{\frac{1}{2}}$ the supertrace of $\gg_0$ on $\gg_{\frac{1}{2}}$. The key structural theorem is the following.
\begin{thm} \cite[Thm 4.1]{KW3} \label{thm:kacwakimoto}
Let $\gg$ be a simple finite-dimensional Lie superalgebra with an invariant bilinear
form $( \ \ |  \ \ )$, and let $x, f$ be a pair of even elements of $\gg$ such that ${\rm ad}\ x$ is diagonalizable with
eigenvalues in $\frac{1}{2} \mathbb Z$ and $[x,f] = -f$. Suppose that all eigenvalues of ${\rm ad}\ x$ on $\gg^f$ are non-positive, so that
$\gg^f = \bigoplus_{j\leq 0} \gg^f_j$. Then
 \begin{enumerate}
\item For each $q^\alpha \in \gg^f_{-j}$, ($j\geq  0$) there exists a $d_0$-closed field $K^\alpha$ of conformal weight
$1 + j$, with respect to $L$.
\item The homology classes of the fields $K^\alpha$, where $\{ q^\alpha\}$ is a basis of $\gg^f$, strongly and freely generate the vertex algebra $\cW^k(\gg, f)$.
\item $H_0(C(\gg, f, k), d_0) = \cW^k(\gg, f)$ and $H_j(C(\gg, f, k), d_0) = 0$ if $j \neq 0$.
\end{enumerate}
\end{thm}
One can also consider the reduction of a module, i.e., for a $V^k(\gg)$-module $M$, the homology of the complex $H(M \otimes F(\gg_+) \otimes F(\gg_{\frac{1}{2}}), d_0)$ is a $\cW^k(\gg, f)$-module that we denote by $H_{k, f}(M)$, that is 
\begin{equation}\label{Wmodule}
H_{k, f}(M) := H(M \otimes F(\gg_+) \otimes F(\gg_{\frac{1}{2}}), d_0).
\end{equation}

\subsection{Translation of $\cW$-algebras}

We summarize a very useful result of Tomoyuki Arakawa, Boris Feigin and one of us \cite{ACF}. 
For this we consider two affine vertex algebras of type $\gg$, the first one, $V$, of level $k$ and the second one, $L$,  of level $\ell$. Denote the generators of the first one by $X^\alpha$ and the generators of the second one by $Y^\alpha$. As complexes we take the previous one $C(\gg, f, k) = V  \otimes F(\gg_+) \otimes F(\gg_{\frac{1}{2}})$ and also  $C(\gg, f, k, \ell) = V \otimes L \otimes F(\gg_+) \otimes F(\gg_{\frac{1}{2}})$. In addition to the differential $d_0$ of last subsection, we also define the differential $d'_0$ as the zero-mode of the field
\[
d'(z) = d(z) +  \sum_{\alpha \in S_+} (-1)^{|\alpha|} :Y^\alpha \varphi^\alpha: .
\]
Then the homology with respect to $d_0$ is just the quantum Hamiltonian reduction on the $V$ subalgebra of $V \otimes L$, while the homology with differential $d'_0$ is the reduction with respect to the diagonal action at level $k+\ell$. We now restrict to $\ell$ being a positive integer and consider  $L = L_\ell(\gg)$ the simple affine vertex algebra of $\gg$ at level $\ell$. The main result of \cite{ACF} is
\begin{thm}\label{Urod}\cite{ACF}
As vertex algebras
\[
 H_0(C(\gg, f, k), d_0) \otimes L \cong H_0(C(\gg, f, k, \ell), d'_0).
\]
\end{thm}
The conformal vector of the right-hand side is not the sum of the standard conformal vectors of the left-hand side. 
We will use this theorem for the case $L = L_m(\gs\gp_{2n})$ for $m \in \mathbb Z_{>0}$, $f$ a minimal nilpotent, and $k = -1/2$. In this case, $H_0(C(\gg, f, k), d_0) = \mathbb C$ is trivial, and we can use the theorem to get a nice decomposition of $L_m(\gs\gp_{2n})$, see \eqref{eq:decompC2}. This then allows us to prove rationality of $\text{Com}( L_m(\gs\gp_{2n-2}),  L_m(\gs\gp_{2n}))$.

\begin{remark}\label{remark:Urodembedding}
Let $ \ga = \gg^f \cap \gg_0$. Then the affine subalgebras of $\cW^{k}(\gg, f)$ and $\cW^{k+\ell}(\gg, f)$ are $V^s(\ga)$ and $V^t(\ga)$ for some levels $s, t$ depending on $k, k+\ell$. Moreover, $L$ has an action of $L_{t-s}(\ga)$. 
Recall that $I^\alpha$ is nontrivial in $H_0(C(\gg, f, k), d_0)$ and $I^\alpha +Y^\alpha$ is nontrivial in  $H_0(C(\gg, f, k, \ell), d'_0)$.
The homology classes $[I^\alpha+Y^\alpha]', [I^\alpha]$ of these fields  generate homomorphic images $N^s(\ga)$ of  $V^s(\ga)$ and $N^t(\ga)$ of $V^t(\ga)$ inside  $H_0(C(\gg, f, k, d_0) \otimes L$  and  $H_0(C(\gg, f, k, \ell), d'_0)$. Consider the cases:
\begin{itemize} 
\item If $V = V^k(\gg)$, then both homologies are subalgebras of $C(\gg, f, k, \ell)$ (see the discussion before \cite[Lem. 4.1]{ACF}), so in this instance the subalgebra $V^t(\ga)$ of  $\cW^{k+\ell}(\gg, f) \subseteq H_0(C(\gg, f, k, \ell), d'_0)$ is the diagonal subalgebra of  $V^s(\ga) \otimes L_{t-s}(\ga) \subseteq H_0(C(\gg, f, k), d_0) \otimes L$ generated by $I^\alpha +Y^\alpha$. 
\item 
The embedding $V^t(\ga) \subseteq V^s(\ga) \otimes L_{t-s}(\ga) \subseteq \cW^k(\gg, f) \otimes L$ induces the embedding $N^t(\ga) \subseteq N^s(\ga) \otimes L_{t-s}(\ga) \subseteq H_0(C(\gg, f, k), d_0) \otimes L$, mapping  $[I^\alpha +Y^\alpha]'$ to $[I^\alpha] +Y^\alpha$.
\end{itemize}
\end{remark}

\subsection{Singular Vectors} \label{sec:singularvectors} 

In this subsection, we compute some conformal weights of singular vectors in principal $\cW$-algebras at nondegenerate admissible levels. We use 
\begin{lemma}\label{lem:sing}\cite[Lem. 3.3]{CL3}
Let $\gg$ be a simple Lie algebra, $\bar\rho, \bar \rho^\vee$ its Weyl vector and Weyl covector, and set
$\bar\alpha =  -\theta$  if $(v, r^\vee) =1$  and $\bar \alpha = -\theta_s$ if $(v, r^\vee) =r^\vee$.
Here $r^\vee$ is the lacity of $\gg$ and $\theta, \theta_s$ are the highest root and highest short root. Set $\bar \lambda  =  n \frac{u}{v} \bar \alpha^\vee  - (\bar \rho, \bar \alpha^\vee)\bar\alpha$

Denote by Sing$(V)$ the weight of the singular vector of $V$ of lowest conformal weight. Then for $k = -h^\vee +\frac{u}{v}$ of (co)principal admissible weight, singular vectors of affine and principal $\cW$-algebra have weight
\begin{enumerate}
\item Sing$(V^k(\gg)) =  \frac{v}{2u}\bar\lambda(\bar\lambda+2\bar\rho)$,
\item Sing$(\cW^k(\gg)) =  \frac{v}{2u}\bar\lambda(\bar\lambda+2\bar\rho)- \bar\lambda\bar\rho^\vee$ for $k$ a nondegenerate admissible level.
\end{enumerate}
\end{lemma}

\begin{cor} \label{cor:singularvector}
Denote by Sing$(V)$ the weight of the singular vector of $V$ of lowest conformal weight. Then for $k = -h^\vee +\frac{u}{v}$ of (co)principal admissible weight, we have 
\begin{enumerate}
\item Sing$(V^k(\gs\gp_{2n})) = v(u-n)$ for $v$ odd, and  Sing$(V^k(\gs\gp_{2n})) = \frac{v}{2}(u-2n+1)$ for $v$ even,
\item Sing$(V^k(\gs\go_{2n+1})) = v(u-2n+2)$ for $v$ odd, and Sing$(V^k(\gs\go_{2n+1})) = \frac{v}{2}(u-2n+1)$ for $v$ even,
\item Sing$(W^k(\gs\gp_{2n})) = (v-2n+1)(u-n)$ for $v$ odd, and Sing$(W^k(\gs\gp_{2n})) = (\frac{v}{2}-2n+2)(u-2n+1)$ for $v$ even and for $k$ a nondegenerate admissible level,
\item Sing$(W^k(\gs\go_{2n+1})) = (v-2n+1)(u-2n+2)$ for $v$ odd, and Sing$(W^k(\gs\go_{2n+1})) = (\frac{v}{2}-n)(u-2n+1)$ for $v$ even and for $k$ a nondegenerate admissible level.
\end{enumerate}
\end{cor}
\begin{proof}  Consider the lattice $\mathbb Z^n$ with orthonormal basis $\epsilon_1, \dots, \epsilon_n$. We embed root and coroots in rescalings of this lattice in the standard way, e.g. the simple positive roots of $\gs\go_{2n+1}$ are $\epsilon_1 -\epsilon_2, \dots, \epsilon_{n-1}-\epsilon_n, \epsilon_n$, and for $\gs\gp_{2n}$ they are $\frac{\epsilon_1 -\epsilon_2}{\sqrt{2}}, \dots, \frac{\epsilon_{n-1}-\epsilon_n}{\sqrt{2}}, \sqrt{2}\epsilon_n$. Especially,
\begin{enumerate}
\item $\gg = \gs\go_{2n+1}$ and $v$ odd. Then $\theta = \theta^\vee = \epsilon_1 + \epsilon_2$, $\rho^\vee= n\epsilon_1 + (n-1)\epsilon_2 + \dots + \epsilon_n$, $\rho = \frac{1}{2}((2n-1)\epsilon_1+ (2n-3)\epsilon_2 + \dots + \epsilon_n)$ and so $\rho \theta^\vee = 2n-2$ and $\rho^\vee \theta = 2n-1$. 
It follows that $\bar \lambda = (u-2n+2)\theta$. 
\item  $\gg = \gs\go_{2n+1}$ and $v$ even. Then $\theta_s = \epsilon_1$ and $\theta_s^\vee = 2\epsilon_1$. Thus $\rho \theta_s^\vee = 2n-1$ and $\theta_s \rho^\vee = n$. It follows that $\bar \lambda = (u-2n+1)\theta_s$.
\item $\gg=\gs\gp_{2n}$ and $v$ odd. Then $\theta = \sqrt{2} \epsilon_1$ and $\rho^\vee  = \frac{1}{\sqrt{2}}((2n-1)\epsilon_1+ (2n-3)\epsilon_2 + \dots + \epsilon_n)$, $\rho= \frac{1}{\sqrt{2}}(n\epsilon_1 + (n-1)\epsilon_2 + \dots + \epsilon_n)$. 
Thus $\rho \theta^\vee = n$ and $\theta\rho^\vee = 2n-1$ and $\bar\lambda = (u-n)\theta$. 
\item $\gg=\gs\gp_{2n}$ and $v$ even. Then $\theta_s = \frac{1}{\sqrt{2}}(\epsilon_1+ \epsilon_2)$ and $ \theta_s^\vee = {\sqrt{2}}(\epsilon_1+ \epsilon_2)$. Thus $\rho\theta_s^\vee = 2n-1$ and $\rho^\vee \theta_s = 2n-2$ and $\bar \lambda = (u-2n+1)\theta_s$. 
\end{enumerate}
Inserting in the formula of Lemma \ref{lem:sing} gives the claim.
\end{proof}
\begin{remark}\label{remark: embedding}
Let $1\leq m <n$. Then there is an embedding $\iota :  \gs\gp_{2n-2m}  \rightarrow \gs\gp_{2n}$ of $\gs\gp_{2n-2m}$ in $\gs\gp_{2n}$ sending the simple roots of $\gs\gp_{2n-2m}$ to the  simple roots $\frac{\epsilon_{m+1} -\epsilon_{m+2}}{\sqrt{2}}, \dots, \frac{\epsilon_{n-1}-\epsilon_n}{\sqrt{2}}, \sqrt{2}\epsilon_n$ of $\gs\gp_{2n-2m}$. Denote by $X^x, Y^y$ the fields of $V^k(\gs\gp_{2n-2m})$ and $V^k(\gs\gp_{2n})$ corresponding to the elements $x \in \gs\gp_{2n-2m}, y \in \gs\gp_{2n}$. Then $\iota$ induces an embedding of $V^k(\gs\gp_{2n-2m})$ in $V^k(\gs\gp_{2n})$ sending $X^x$ to $Y^{\iota(x)}$ for $x \in \gs\gp_{2n-2m}$. This embedding can be characterized via nilpotent elements as follows:
Let $f_i = e_{-\sqrt{2}\epsilon_i}$ and $\gg_i$  be the subspace of $\gg_{i-1}^{f_i}$ of $\sqrt{2}\epsilon_i$ weight zero  for $i = 1, \dots, m$ and $\gg_0 := \gs\gp_{2n}$ so that $\gg_i \cong \gs\gp_{2n-2i}$ and especially $\gg_m = \iota(\gs\gp_{2n-2m})$. 
\end{remark}

\subsection{Universal two-parameter even spin $\cW_{\infty}$-algebra}
 We briefly recall the universal two-parameter vertex algebra $\cW^{\text{ev}}(c,\lambda)$ of type $\cW(2,4,\dots)$, which was conjectured to exist in the physics literature \cite{CGKV}, and was constructed in \cite{KL}. It is defined over the polynomial ring $\mathbb{C}[c,\lambda]$ and is generated by a Virasoro field $L$ of central charge $c$, and a weight $4$ primary field $W^4$. In addition, it is strongly and freely generated by the fields $\{L, W^{2i}|\ i \geq 2\}$ where $W^{2i} = W^4_{(1)} W^{2i-2}$ for $i\geq 3$. 

$\cW^{\mathrm{ev}}(c,\lambda)$ is simple as a vertex algebra over $\mathbb{C}[c,\lambda]$, but there is a discrete family of prime ideals $I = (p(c,\lambda)) \subseteq \mathbb{C}[c,\lambda]$ for which the quotient 
$$\cW^{{\rm ev}, I}(c,\lambda) = \cW^{\mathrm{ev}}(c,\lambda)/ I \cdot \cW^{\mathrm{ev}}(c,\lambda),$$ is not simple as a vertex algebra over the ring $\mathbb{C}[c,\lambda] / I$. We denote by $\cW^{\mathrm{ev}}_I(c,\lambda)$ the simple quotient of $\cW^{{\rm ev}, I}(c,\lambda)$ by its maximal proper graded ideal $\cI$. After a suitable localization, all one-parameter vertex algebras of type $\cW(2,4,\dots, 2N)$ for some $N$ satisfying some mild hypotheses, can be obtained as quotients of $\cW^{\mathrm{ev}}(c,\lambda)$ in this way; see \cite[Thm. 2.1]{CKoL}. The distinct generators of such ideals arise as irreducible factors of Shapovalov determinants, and are in bijection with such one-parameter vertex algebras. 

We also consider $\cW^{\mathrm{ev},I}(c,\lambda)$ for maximal ideals $$I = (c- c_0, \lambda- \lambda_0),\qquad c_0, \lambda_0\in \mathbb{C}.$$
Then $\cW^{\mathrm{ev},I}(c,\lambda)$ and its quotients are vertex algebras over $\mathbb{C}$. Given maximal ideals $I_0 = (c- c_0, \lambda- \lambda_0)$ and $I_1 = (c - c_1, \lambda - \lambda_1)$, let $\cW_0$ and $\cW_1$ be the simple quotients of $\cW^{\mathrm{ev},I_0}(c,\lambda)$ and $\cW^{\mathrm{ev},I_1}(c,\lambda)$. A criterion for $\cW_0$ and $\cW_1$ to be isomorphic is given by \cite[Thm. 8.1]{KL}; aside from a few degenerate cases, we must have $c_0 = c_1$ and $\lambda_0 = \lambda_1$. This implies that aside from the degenerate cases, all other coincidences among the simple quotients of one-parameter vertex algebras $\cW^{\mathrm{ev},I}(c,\lambda)$ and $\cW^{\mathrm{ev},J}(c,\lambda)$, correspond to intersection points of their truncation curves $V(I)$ and $V(J)$.

We record a slight improvement to the results of \cite{KL} for later use. By \cite[Thm. 8.1, case (4)]{KL}, for $c \neq 1, 25$, the Virasoro algebra $\text{Vir}^c$ is realized as a quotient of $\cW^{\text{ev}}(c,\lambda)$ by setting 
\begin{equation} \label{virasoroasquotient} \lambda = \pm \frac{1}{7 \sqrt{(c-25) (c-1)}},\end{equation} and taking the simple quotient. This occurs because the weight $4$ field becomes singular, and hence all higher descendants $W^{2k}$ for $k\geq 2$ also vanish in the simple quotient. It is straightforward to check that the truncation curve for $\cW^k(\gs\gp_{2n})$ given by  \cite[Eq. A.1]{KL}, in the case $n=1$ coincides with \eqref{virasoroasquotient}. It follows that \cite[Thm. 9.3]{KL}, which is stated in the range $2\leq n < m$, actually holds for $1\leq n < m$. Similarly,  \cite[Thm. 9.4]{KL} which is stated for $n\geq 2$, also holds for $n\geq 1$.

Next, the truncation curve for $\cW^k(\gs\go_{2n})^{\mathbb{Z}_2}$ given in \cite[Thm. 6.3]{KL} for $n\geq 3$, in fact holds for $n=2$ as well. In this case, 
$$\cW^k(\gs\go_{4})^{\mathbb{Z}_2} \cong \big(\text{Vir}^c \otimes \text{Vir}^c\big)^{\mathbb{Z}_2},\qquad c = -\frac{(1 + 2 k) (4 + 3 k)}{2 + k} ,$$ where $\mathbb{Z}_2$ acts by permutation. This truncation curve was computed in \cite{MPS} and is easily seen to coincide with the specialization of the curve for $\cW^k(\gs\go_{2n})^{\mathbb{Z}_2}$ to the case $n=2$. Therefore \cite[Thm. 9.1]{KL}, which was stated in the range $3\leq n < m$, actually holds for $2\leq n < m$. Similarly, \cite[Thm. 9.4]{KL} holds for all $n\geq 1$ and $m\geq 2$.

\subsection{Extensions of rational vertex operator algebras}
An extension of a lisse ($C_2$-cofinite) vertex algebra is also lisse \cite[Prop. 5.2]{ABD}. We also need a general result that says that extensions of a rational vertex algebra are rational as well. 
One such statement is \cite[Thm. 3.5]{HKL}. One assumption, however, is a positivity assumption on conformal weights of modules, and this assumption is not satisfied in most of our cases of interest. Another such statement is \cite[Cor. 1.1]{CKM2}, which however only applies to $\mathbb Z$-graded vertex algebra extensions. We need to consider $\frac{1}{2}\mathbb Z$-graded vertex superalgebra extensions, and so we now ensure that the rationality result generalizes to this setting.

We first recall the main basic theorems of \cite{KO, HKL, CKL, CKM1, CKM2} using \cite[Section 2]{CY}. 
Let $V$ be a vertex operator algebra and $\cC = (\cC, \btimes, \one, \cA_{\bullet,\bullet,\bullet}, l_{\bullet}, r_{\bullet}, \cR_{\bullet,\bullet})$ be a category of $V$-modules with a natural vertex and braided tensor category structure in the sense of \cite{HLZ0}-\cite{HLZ8}. The tensor bifunctor is denoted by $\btimes$, the tensor identity is just the vertex operator algebra $V$ itself and will be denoted by $\one$. The associativity constraint, the left and right unit constraint and the braiding are denoted by $\cA, l, r, \cR$.

\begin{defn}\cite{KO}
An algebra is a triple $(A, \mu_A, \iota_A)$ with $A$ an object in $\cC$ and multiplication $\mu_A: A\btimes A \rightarrow A$ and an embedding of tensor unit $\iota_A: \one_{\cC} \rightarrow A$  that satisfy
\begin{itemize}
\item[(1)]Multiplication is associative: $\mu_A \circ (\id_A \btimes \mu_A) = \mu_A \circ (\mu_A \btimes \id_A)\circ \cA_{A,A,A}: A\btimes (A \btimes A) \rightarrow A$
\item[(2)]Multiplication is unital: $\mu_A \circ (\iota_A \btimes \id_A) = l_A: \one_{\cC}\btimes A \rightarrow A$ and $\mu_A \circ (\id_A \btimes \iota_A) = r_A: A \btimes \one_{\cC} \rightarrow A$
\end{itemize}
$(A, \mu_A, \iota_A)$ is a  commutative algebra if additionally
\begin{itemize}
\item[(3)] Multiplication is commutative: $\mu_A \circ \cR_{A,A} = \mu_A: A \btimes A \rightarrow A$.
\end{itemize}
\end{defn}
We will use the short-hand notation $A$ for an algebra $(A, \mu_A, \iota_A)$.
\begin{defn}\cite{KO}
Let $A$ be an algebra, and define $\cC_A$ to be the category of pairs $(X, \mu_X)$, where $X$ is an object in $\cC$ and $\mu_X: A\btimes X \rightarrow X$ is a morphism in $\cC$ subject to
\begin{itemize}
\item[(1)]Unit property: $l_X = \mu_X \circ (\iota_A \btimes \id_X): \one_{\cC}\btimes X \rightarrow X$
\item[(2)]Associativity: $\mu_X \circ (\id_A \btimes \mu_X) = \mu_X \circ (\mu_A \btimes \id_X)\circ \cA_{A, A, X}: A \btimes (A \btimes X) \rightarrow X$.
\end{itemize}
A morphism $f \in \Hom_{\cC_A}((X_1, \mu_{X_1}), (X_2, \mu_{X_2}))$ is a morphism $f \in \Hom_{\cC}(X_1, X_2)$ such that $\mu_{X_2}\circ (\id_A \btimes f) = f\circ \mu_{X_1}$.

When $A$ is commutative, define $\cC_A^{loc}$ to be the full subcategory of $\cC_A$ containing local objects: those $(X, \mu_X)$ such that $\mu_X \circ \cR_{X,A} \circ \cR_{A, X} = \mu_X$.
\end{defn}
There is an induction functor $\cF_A: \cC \rightarrow \cC_A$ that maps an object $X \in \cC$ to $(A \boxtimes X, \mu \boxtimes \text{Id}_X)$ and a morphism $\varphi : X \rightarrow Y$ to $\text{Id}_A \boxtimes \varphi : \cF(X) \rightarrow \cF(Y)$, see \cite{CKM1} for more details. 

Super commutative algebras are defined similarly in \cite{CKL}.
The structural categorical results are summarized in the following theorem.
\begin{thm}\label{sum1}
Let $\cC$ be a braided tensor category and let $A$ be a commutative algebra in $\cC$. Then the following results hold:
\begin{itemize}
\item[(1)] The category $\cC_A$  is a tensor category (\cite[Thm.~1.5]{KO}, \cite[Thm.~2.53]{CKM1}).
\item[(2)]  The subcategory $\cC_A^{loc}$  is a braided tensor category (\cite[Thm.~1.10]{KO}, \cite[Thm.~2.55]{CKM1}]).
\item[(3)]The induction functor $\cF_A: \cC \rightarrow \cC_A$ is monoidal  (\cite[Thm.~1.6]{KO}, \cite[Thm.~2.59]{CKM1}).
\item[(4)]The induction functor satisfies Frobenius reciprocity, that is, it is left adjoint to the forgetful functor $\cG_A$ from $\cC_A$ to $\cC$:
\begin{equation}\label{eqn:fro-rec}
\Hom_{\cC_A}(\cF_A(W), X) = \Hom_{\cC}(W, \cG_A(X))
\end{equation}
for objects $W$ in $\cC$ and $X$ in $\cC_A$ (see for example \cite[Thm.~1.6(2)]{KO},  \cite[Lem.~2.61]{CKM1}).
\item[(5)]Let $W$ be an object in $\cC$. Then $\cF_A(W)$ is in $\cC_A^{loc}$ if and only if the monodromy is trivial, that is $\cM_{A, W} : = \cR_{W,A}\circ \cR_{A,W} = \id_{A\btimes W}$ (\cite[Prop.~2.65]{CKM1}).
\item[(6)]Let $\cC^0$ be the full subcategory of objects in $\cC$ that induce to $\cC_A^{loc}$. Then the restriction of the induction functor $\cF_A: \cC^0 \rightarrow \cC_A^{loc}$ is a braided tensor functor (\cite[Thm.~2.67]{CKM1}).
\end{itemize}
\end{thm}

The next theorem tells us that we can use the categorical results to study extensions of vertex algebras. 
\begin{thm}\label{sum2}
Let $V$ be a vertex operator algebra, and let $\cC$ be a category of $V$-modules with a vertex tensor category structure in the sense of \cite{HLZ0}-\cite{HLZ8}. Then the following results hold:
\begin{itemize}
\item[(1)] A vertex operator algebra extension $V \subseteq A$ in $\cC$ is equivalent to a commutative associative algebra in the braided tensor category $\cC$ with trivial twist and injective unit (\cite[Thm.~3.2]{HKL}).
\item[(2)] The category of modules in $\cC$ for the extended vertex operator algebra $A$ is isomorphic to the category of local $\cC$-algebra modules $\cC_A^{loc}$ (\cite[Thm.~3.4]{HKL}).
\item[(3)] The results in (1) and (2) hold for a vertex operator superalgebra extension: The vertex operator superalgebra extension $V \subseteq A$ in $\cC$ such that $V$ is in the even subalgebra $A^0$ is equivalent to a commutative associative superalgebra in $\cC$ whose twist $\theta$ satisfies $\theta^2 = \id_A$. The category of generalized modules for the vertex operator superalgebra $A$ is isomorphic to the category of local $\cC$-superalgebra modules $\cC_A^{loc}$ (\cite[Thm.~3.13, 3.14]{CKL}).
\item[(4)] The isomorphism given in \cite[Thm.~3.4]{HKL} and \cite[Thm.~3.14]{CKL} between the category of modules in $\cC$ for the extended vertex operator (super)algebra $A$ and the category of local $\cC$-algebra modules $\cC_A^{loc}$ is an isomorphism of vertex tensor (super)categories (\cite[Thm.~3.65]{CKM1}).
\end{itemize}
\end{thm}

A tensor category is called a fusion category if it is semisimple with finitely many inequivalent simple objects and every object is rigid. In particular in that case there is a trace and thus a notion of dimension of objects. The following theorem gives conditions under which a vertex algebra extension has a semisimple representation category. 
\begin{thm}\cite[Theorem 5.12]{CKM2} \label{thm:dim}
 Suppose $\cU$ and $\cW$ are braided fusion categories of modules for simple self-contragredient vertex operator algebras $U$ and $W$, respectively, and
 \begin{equation*}
  A=\bigoplus_{i\in I} U_i\otimes W_i
 \end{equation*}
is a simple $\mathbb Z$-graded vertex operator algebra extension of $ U\otimes W$ in $\cC=\cU\boxtimes \cW$ where the $U_i$ are distinct simple modules in $\cU$ including $U_0= U$ and the $W_i$ are modules in $\cW$ such that
\begin{equation*}
 \dim\text{Hom}_{\cW}( W, W_i)=\delta_{i,0}.
\end{equation*}
Then $\dim_\cC\, A>0$ and the category of (grading-restricted, generalized) $A$-modules in $\cC$ is a braided fusion category.
\end{thm}

We need to generalize this theorem to the case where $A$ is possibly a vertex operator superalgebra and possibly $\frac{1}{2} \mathbb Z$-graded. 
Let $V$ be a vertex operator algebra and $\mathcal C$ a category of $V$-modules. We assume this category to be braided fusion. 
Let $A\supseteq V$ be an object in $\mathcal C$ that itself carries the structure of a vertex algebra or vertex superalgebra. We assume $V$ to be $\mathbb Z$-graded, but $A$ does not need to be. However, we assume $A$ to be $\mathbb{Z}_2$-graded, $A = A_{\bar 0} \oplus A_{\bar 1}$, such that $A_{\bar  0}$ is a $\mathbb Z$-graded vertex algebra containing $V$, $V\subseteq A_{\bar 0}$.
In other words, $A_{\bar 0}$ is a $\mathbb{Z}_2$-orbifold of $A$ and so $A_{\bar 1}$ is a self-dual simple current \cite{McRae1}\footnote{A simpler proof of this statement is given in Appendix A of \cite{CKLR}, observing that the argument there is the same for vertex superalgebras.}.
 Let $\mathcal C_{A}^{\text{loc}}$ and $\mathcal C_{A_{\bar 0}}^{\text{loc}}$ be the categories of local $A$ and $A_{\bar 0}$-modules that lie in $\mathcal C$. Note that $\mathcal C_A = (\mathcal C_{A_{\bar{0}}})_A$ by \cite[Section 3.6]{DMNO} and since a local $A$-module is necessarily local as a module for the subalgebra $A_{\bar{0}}$, also 
$\mathcal C^{\text{loc}}_A = (\mathcal C^{\text{loc}}_{A_{\bar{0}}})^{\text{loc}}_A$.

Let $M$ be a simple object in $\mathcal C_{A_{\bar 0}}^{\text{loc}}$ and let $\cF : \mathcal C_{A_{\bar 0}}^{\text{loc}} \rightarrow \mathcal C_{A}$ be the induction functor and $\cG$ the restriction functor. We have $\cG(\cF(M)) \cong M  \oplus N$ with $N = A_{\bar 1} \boxtimes M$. The modules $M$ and $N$ are graded by conformal weight, i.e. $$M = \bigoplus_{n \in  \mathbb Z + h_M} M_n, \qquad N = \bigoplus_{n \in  \mathbb Z + h_N} N_n,$$ for some complex numbers $h_M, h_N$.  The monodromy $\cR_{A_{\bar 1}, M} \circ \cR_{M, A_{\bar 1}}$ is either the identity on $N$ or minus the identity on $N$, see \cite{CKL} for details. In the latter case, we call any submodule of $\cF(M)$ a twisted module and the subcategory of $\cC_A$ whose objects are direct sums of twisted modules is denoted by $\cC_A^{\text{tw}}$. 
The module $\cF(M)$ is either simple or a direct sum of two simple modules \cite[Prop. 4.18 and Cor. 4.22]{CKM1}. In the latter case one has $M \cong N$ and by Frobenius reciprocity $\text{Hom}_{\cC_{A_{\bar 0}}}(M, M \oplus M) = \text{Hom}_{\cC_{A}}(\cF(M), \cF(M))$, i.e. these two simple summands of $\cF(M)$ need to be inequivalent. 
We are interested in three cases and their properties are given in \cite[Section 4.2]{CKM1}.
\begin{enumerate}
\item $A$ is a $\frac{1}{2}\mathbb Z$-graded vertex algebra, especially conformal weights of $A_{\bar 1}$ are in $\mathbb Z +\frac{1}{2}$. 
In this case $\cF(M)$ is always simple. Moreover it is local if and only if $h_M = h_N +\frac{1}{2} \mod 1$.  Otherwise it is twisted. See Section 4.2.2 and especially Lemma 4.29 of \cite{CKM1}.
\item $A$ is a $\frac{1}{2}\mathbb Z$-graded vertex superalgebra, especially conformal weights of $A_{\bar 1}$ are in $\mathbb Z +\frac{1}{2}$. 
$\cF(M)$ is local if and only if $h_M = h_N +\frac{1}{2} \mod 1$ and otherwise it is twisted. $\cF(M)$ is simple if it is local. if $\cF(M)$ is twisted then either it is simple or $M \cong N$ and $\cF(M)$ is a direct sum of two  simple objects.   This type of extension is Section 4.2.3 of \cite{CKM1}.
\item $A$ is a $\mathbb Z$-graded vertex superalgebra. In this case $\cF(M)$ is always simple. Moreover it is local if and only if $h_M = h_N  \mod 1$.  Otherwise it is twisted. See Section 4.2.1 and especially Lemma 4.26 of \cite{CKM1}.
\end{enumerate}
If we assume that $A$ is a simple vertex (super)algebra, then the action of $A$ on a non-zero $A$-module cannot have a kernel and so especially $M$ and $N$ cannot be zero. Assume  that $A$ is simple. 

Assume that $V = U \otimes W$ is the tensor product of two  vertex operator algebras and assume that 
\[
A_{\bar 0} \cong \bigoplus_{i \in I} U_i \otimes W_i
\]
for some index set $I$. Here the $U_i$ are distinct simple $U$-modules and we set $U = U_0$ and $W=W_0$. 
Also, assume that the $U_i$ and $W_i$ are objects of vertex tensor categories $\mathcal C_U$ of $U$-modules and $\mathcal C_W$ of $W$-modules, that are braided fusion categories  Then the Deligne product $\mathcal C = \mathcal C_U \boxtimes \mathcal C_W$ is a vertex tensor category as well \cite[Thm. 5.5]{CKM2}.
The $W_i$ are not necessarily distinct, but one requires that Hom$(W, W_i)=0$ for $i\neq 0$. Under these assumptions $\mathcal C_{A_{\bar 0}}^{\text{loc}}$
is a braided fusion category as well by Theorem \ref{thm:dim}. We now prove that $\cC_A^{\text{loc}}$ and $\cC_A^{\text{tw}}$ are also semisimple.

The following is similar to the proof of \cite[Thm. 5.13]{CGN}.
Let $\mathcal D$ be either $\mathcal C_A^{\text{loc}}$ or $\mathcal C_A^{\text{tw}}$. 
Let $X, Y$ be two simple modules in $\cD$ and consider an exact sequence $E: 0 \rightarrow X \rightarrow Z \rightarrow Y \rightarrow 0$. We show that it splits. 
Let $M$ be a direct summand of $\cG(Y)$. There are two cases 
\begin{enumerate} 
\item $Y \cong \cF(M)$ and $\cG(Y) = M \oplus N$ with $N = A_{\bar 1} \boxtimes M$ and $M \not\cong N$ or 
\item $\cF(M) \cong Y \oplus W$ with $W \not \cong Y$ and $\cG(\cF(M)) \cong M \oplus M$. 
\end{enumerate}
We use Frobenius reciprocity. In the first case
  \begin{equation}\nonumber
  \begin{split}
  \text{Hom}_{\mathcal C_A}(Y, Z) &=  
  \text{Hom}_{\mathcal C_A}(\mathcal F(M), Z)  = \text{Hom}_{\mathcal C_{A_{\bar 0}}^{\text{loc}}}(M,  \mathcal G(Z)) 
  =  \text{Hom}_{\mathcal C_{A_{\bar 0}}^{\text{loc}}}(M, \mathcal G(X) \oplus \mathcal G(Y)) \\ &=  \text{Hom}_{\mathcal C_{A_{\bar 0}}^{\text{loc}}}(M, \mathcal G(X \oplus Y)) 
  = \text{Hom}_{\mathcal C_A}(\mathcal F(M), X \oplus Y) =  \text{Hom}_{\mathcal C_A}(Y, X \oplus Y)
  \end{split}
  \end{equation}
and hence $E$ splits.
In the second case
\begin{equation}\nonumber
  \begin{split}
  \text{Hom}_{\mathcal C_A}(Y \oplus& W, Z) =  
  \text{Hom}_{\mathcal C_A}(\mathcal F(M), Z)  = \text{Hom}_{\mathcal C_{A_{\bar 0}}^{\text{loc}}}(M, \mathcal G(Z)) 
  =  \text{Hom}_{\mathcal C_{A_{\bar 0}}^{\text{loc}}}(M, \mathcal G(X) \oplus \mathcal G(Y)) \\ &=  \text{Hom}_{\mathcal C_{A_{\bar 0}}^{\text{loc}}}(M, \mathcal G(X \oplus Y)) 
  = \text{Hom}_{\mathcal C_A}(\mathcal F(M), X \oplus Y) =  \text{Hom}_{\mathcal C_A}(Y \oplus W, X \oplus Y)
  \end{split}
  \end{equation}
  and hence $E$ splits as well.
Summarizing:
\begin{prop}\label{prop:rat} \textup{($\frac{1}{2} \mathbb Z$-graded vertex superalgebra generalization of \cite[Thm. 5.12]{CKM2})}
\newline
Let $A= A_{\bar 0} \oplus A_{\bar 1}$ be a simple vertex (super)algebra of one of the three types listed above extending the $\mathbb Z$-graded self-contragredient simple vertex algebra $U \otimes W$. Assume that
\[
A_{\bar 0} \cong \bigoplus_{i \in I} U_i \otimes W_i
\]
for some index set $I$. Here the $U_i$ are distinct simple $U$-modules and we set $U = U_0$ and $W=W_0$. 
Also assume that the $U_i$ and $W_i$ are objects of vertex tensor categories that are braided fusion categories $\mathcal C_U, \mathcal C_W$ of $U$ respectively $W$-modules. Assume that Hom$(W, W_i)=0$ for $i\neq 0$. Let $\mathcal C = \mathcal C_U \boxtimes \mathcal C_W$. Then both $\mathcal C_A^{\text{loc}}$ and $\mathcal C_A^{\text{tw}}$  are  semisimple. 
\end{prop}

\begin{prop}\label{prop:semisimple}
Let $V$ be a simple vertex superalgebra of CFT-type  and let $\pi$ be a simple rank $n$ Heisenberg subalgebra of $V$  with 
\[
V = \bigoplus_{\lambda \in L} V_\lambda
\]
  the decomposition of $V$ into generalized weight spaces for $\pi$. Here $L$ is the set of $\lambda \in \mathbb C^n$ with $V_\lambda \neq 0$. If $C = \text{Com}( \pi, V)$ acts semisimply on $V$,  then $L \subseteq\mathbb C$ is a subgroup of $\mathbb C^n$ and there are simple $C$-modules $C_\lambda$, such that $V_\lambda = \pi_\lambda \otimes C_\lambda$ as $\pi \otimes C$-modules. In particular, $\pi$ acts semisimply on $V$.
\end{prop}
\begin{proof}
$V_0$ is a vertex subalgebra of $V$. 
Since $V$ is simple, $V$ is spanned by $\{ u_nv \, | \,  u \in V, n \in \mathbb Z\}$ for any non-zero $v\in V$ \cite[Cor. 4.2]{DM2}. 
Thus $\lambda, \mu  \in L$ implies $\mu -\lambda \in L$ as well. For $u, v$ in $V$, there exists $m \in \mathbb Z$ such that $u_mv \neq  0$ \cite[Prop. 11.9]{DL} and so $\lambda, \mu  \in L$ implies $\mu + \lambda \in L$ as well.  It follows that each $V_\lambda$ is a simple $V_0$-module and that $L$ is a subgroup of $\mathbb C^n$. 

If $V_0$ is not completely reducible as a $\pi$-module, then there exists a length two self-extension of $\pi$. It is generated by $v \in V_0$, s.t. there is a Heisenberg field $X(z)$ whose zero-mode $X_0$ acts nilpotently, i.e.
$X_0v = w \neq  0$ and $X_nv=0$ for $n>0$ and such that $w$ is a vacuum vector for $\pi$, i.e. $w \in C$.  $C$ is simple, since it acts semisimply on itself and since vertex algebras can't be decomposable as modules for themselves. Hence there exists $y \in C$ and $m \in \mathbb Z$ such that $y_mw= | 0 \rangle$. But this means that $X_0 y_m v = y_m X_0 v= y_m w = | 0 \rangle$, i.e. $y_mv$ is a vector at the top level of $V$ and not in the kernel of $X_0$, contradicting that $V$ is of CFT-type. It follows that $V_0$ is completely reducible as a $\pi$-module and hence the only possibility is that $V_0 = \pi \otimes C$. Each $V_\lambda$ is a simple $V_0$-module and so it must be of the form $V_\lambda = \pi_\lambda \otimes C_\lambda$ for simple $C$-modules $C_\lambda$.
\end{proof}

\begin{prop} \label{rationalextension}
Let $V$ be a simple lisse vertex superalgebra of CFT-type and $U$ be the affine subalgebra generated by the weight one subspace of $V$. 
Let $W = \text{Com}(U, V)$ and assume that $W$ is self-contragredient. If $W$ is rational, then so is $V$. 
\end{prop}
\begin{proof}
By \cite{DM} applied to the even subalgebra of $V$, $U$ is necessarily the tensor product of a Heisenberg vertex algebra of some rank and an integrable affine vertex algebra $L$. The bilinear form on the weight one subspace is non-degenerate since $V$ is simple and so especially it is non-degenerate on the Heisenberg subalgebra, so that the Proposition \ref{prop:semisimple} applies. The commutant $\tilde U = \text{Com}(W, V)$ of $W$ in $V$ is an extension of $L \otimes V_L$ for some lattice vertex algebra $V_L$ (see \cite[Lem. 5.8]{CGN} or \cite{DM}). $L$ needs to be positive definite for $V$ being of CFT-type.  Since the (categorical) dimension of any simple lattice vertex algebra module is one and the one of any integrable representation is positive, it follows that $\tilde U$ is rational by \cite[Thm. 3.3]{KO}. Hence Proposition \ref{prop:rat} (respectively already \cite[Thm. 5.12]{CKM2} if $V$ is an integer graded vertex algebra)  applies to $V$ as an extension of $\tilde U \otimes W$, and so $V$ is rational as well. 
\end{proof}
By \cite[Cor. 3.2]{Li}, a simple integer graded vertex operator algebra is self-dual if its conformal weight one space is in the kernel of the Virasoro mode $L_1$. This holds especially if the conformal weight one space vanishes.
\begin{cor}\textup{(Corollary of \cite[Cor. 3.2]{Li})}
Let $W$ be a simple integer graded vertex operator algebra of CFT-type with no fields of conformal weight one. Then $W$ is its own contragredient dual. 
\end{cor}
Especially Proposition \ref{rationalextension} applies if $W$ is a simple rational principal $\cW$-algebra associated to a simple Lie algebra or an order two orbifold of a simple rational principal $\cW$-algebra.

We need another corollary of Proposition \ref{prop:rat}. For this let $V$ be a simple vertex (super)algebra and $W_1, W_2$ be simple vertex (super)subalgebras. Let $L, W_3$ be simple vertex (super)subalgebras of $W_2$, such that $\text{Com}(W_2, V)=W_1$, $\text{Com}(L, W_2) =W_3$ and such that $W_2, W_3$ are actually integer graded self-contragredient vertex operator algebras. Assume that there are braided fusion categories $\cC_L, \cC_{1}, \cC_{3}$ of modules for the vertex algebras $L, W_1, W_3$, such that $V$ is an object in $\cC := \cC_L \boxtimes \cC_{1} \boxtimes \cC_3$. 
Then $W_2$ corresponds to a commutative (super)algebra object in $\cD :=   \cC_L \boxtimes \cC_3$, that we also denote by $W_2$. We thus have an induction functor $\cF : \cD \rightarrow \cD_{W_2}$ with right adjoint denoted by $\cG$. For an object $M$ in $\cD_{W_2}$ one has by Frobenius reciprocity
\[
\text{Hom}_{\cD_{W_2}}(W_2, M) \cong \text{Hom}_\cD(W_3 \otimes L, \cG(M)). 
\]
Hence a simple object $M$ in $\cD_{W_2}$ with the property that $\text{Hom}_\cD(W_3 \otimes L, \cG(M))$ is non-zero necessarily is isomorphic to $W_2$. 
This implies 
\[
\text{Com}(W_2, V) = \text{Com}(W_3 \otimes L, V)
\]
and thus 
\[
W_1 = \text{Com}(W_2, V) =  \text{Com}(W_3 \otimes L, V) =  \text{Com}\left(W_3, \text{Com}(L, V)\right).
\]
The setting of Proposition \ref{prop:rat} thus holds for $W_1 = U, W_3 = W,  \text{Com}(L, V) = A$, i.e. 
\begin{cor}\label{cor:triplecoset}
With the above setting, the categories of local and twisted $\text{Com}(L, V)$-modules that lie in $\cC_1 \boxtimes \cC_3$ are semisimple. Especially if $W_1$ and $W_3$ are rational and lisse, then so is $\text{Com}(L, V)$.
\end{cor}

\section{Hook-type $\cW$-algebras in types $B$, $C$, and $D$} \label{sec:hooktype}
In this section, we define the eight families of $\cW$-(super)algebras that we need in a unified framework. First, let $\mathfrak{g}$ be a simple Lie (super)algebra of type $B$, $C$, or $D$; in particular, $\gg$ is either $\mathfrak{so}_{2n+1}$, $\mathfrak{sp}_{2n}$, $\mathfrak{so}_{2n}$, or $\mathfrak{osp}_{n|2r}$. We further assume that $\gg$ admits a decomposition 
\begin{equation} \label{decompg} \mathfrak{g} = \mathfrak{a} \oplus \mathfrak{b} \oplus \rho_{\mathfrak{a}} \otimes \rho_{\mathfrak{b}},\end{equation} with the following properties.
\begin{enumerate}
\item $\mathfrak{a}$ and $\mathfrak{b}$ are Lie sub(super)algebras of $\mathfrak{g}$. Here $\mathfrak{b}$ is either $\mathfrak{so}_{2m+1}$ or $\mathfrak{sp}_{2m}$, and $\mathfrak{a}$ can be $\mathfrak{so}_{2n+1}$, $\mathfrak{sp}_{2n}$, $\mathfrak{so}_{2n}$, or $\mathfrak{osp}_{1|2n}$.
\item $\rho_{\mathfrak{a}}$ and $\rho_{\mathfrak{b}}$ transform as the standard representations of $\mathfrak{a}$ and $\mathfrak{b}$, respectively.
\item $\rho_{\mathfrak{a}}$ and $\rho_{\mathfrak{b}}$ have the same parity, which can be even or odd.
\end{enumerate}
Note that if $\ga = \go\gs\gp_{1|2n}$, $\rho_{\ga}$ even means that $\rho_{\ga} \cong \mathbb C^{2n|1}$ as a vector superspace, whereas $\rho_{\ga}$ odd means that $\rho_{\ga} \cong \mathbb C^{1|2n}$. If $\gg = \go\gs\gp_{m|2n}$, we use the following convention for its dual Coxeter number $h^{\vee}$.
\begin{equation} \label{convention:osp}
h^\vee = \begin{cases} m-2n-2 & \qquad \text{type $B$} \\ \frac{2n+2-m}{2} & \qquad \text{type $C$} 
\end{cases} , \qquad \text{sdim} (\go\gs\gp_{m|2n})= \frac{(m-2n)(m-2n-1)}{2}.\end{equation}
In this notation, type $B$ (respectively $C$) means that the subalgebra $\gb \subseteq \gg$ is of type $B$ (respectively $C$), and the bilinear form on $\go\gs\gp_{m|2n}$ is normalized so that it coincides with the usual bilinear form on $\gb$.  The cases we need are the following. 
\begin{itemize}

\item{Case 1B}:  $\ \ \displaystyle \mathfrak{g} =\mathfrak{so}_{2n+2m+2},\quad \mathfrak{b} = \mathfrak{so}_{2m+1},\quad \mathfrak{a} = \mathfrak{so}_{2n+1},\qquad \rho_{\mathfrak{a}} \otimes \rho_{\mathfrak{b}}\ \text{even}$.

\smallskip

\item Case 1C: $\ \ \displaystyle\mathfrak{g} =\mathfrak{osp}_{2m+1|2n}, \quad \mathfrak{b} = \mathfrak{so}_{2m+1}, \quad \mathfrak{a} = \mathfrak{sp}_{2n},\qquad \rho_{\mathfrak{a}} \otimes \rho_{\mathfrak{b}}\ \text{odd}$.  

\smallskip

\item Case 1D: $\ \ \displaystyle\mathfrak{g}  =\mathfrak{so}_{2n+2m+1}, \quad \mathfrak{b} = \mathfrak{so}_{2m+1}, \quad \mathfrak{a} = \mathfrak{so}_{2n},\qquad \rho_{\mathfrak{a}} \otimes \rho_{\mathfrak{b}}\ \text{even}$.  
 
 \smallskip
 
 \item Case 1O:  $\ \ \displaystyle\mathfrak{g} =\mathfrak{osp}_{2m+2|2n}, \quad \mathfrak{b} = \mathfrak{so}_{2m+1}, \quad \mathfrak{a} = \mathfrak{osp}_{1|2n},\qquad \rho_{\mathfrak{a}} \otimes \rho_{\mathfrak{b}}\ \text{odd}$.  
 
 \smallskip
 
\item{Case 2B}:  $\ \ \displaystyle\mathfrak{g} =\mathfrak{osp}_{2n+1|2m}, \quad \mathfrak{b} = \mathfrak{sp}_{2m}, \quad \mathfrak{a} = \mathfrak{so}_{2n+1}, \qquad \rho_{\mathfrak{a}} \otimes \rho_{\mathfrak{b}}\ \text{odd}$.

\smallskip

\item Case 2C: $\ \ \displaystyle\mathfrak{g} =\mathfrak{sp}_{2n+2m}, \quad \mathfrak{b} = \mathfrak{sp}_{2m}, \quad \mathfrak{a} = \mathfrak{sp}_{2n}, \qquad \rho_{\mathfrak{a}} \otimes \rho_{\mathfrak{b}}\ \text{even}$.  

\smallskip

\item Case 2D:  $\ \ \displaystyle\mathfrak{g} =\mathfrak{osp}_{2n|2m}, \quad \mathfrak{b} = \mathfrak{sp}_{2m}, \quad \mathfrak{a} = \mathfrak{so}_{2n}, \qquad \rho_{\mathfrak{a}} \otimes \rho_{\mathfrak{b}}\ \text{odd}$.  

\smallskip

\item Case 2O:  $\ \ \displaystyle\mathfrak{g} =\mathfrak{osp}_{1|2n+2m}, \quad \mathfrak{b} = \mathfrak{sp}_{2m}, \quad \mathfrak{a} = \mathfrak{osp}_{1 | 2 n}, \qquad \rho_{\mathfrak{a}} \otimes \rho_{\mathfrak{b}}\ \text{even}$.  

\end{itemize}

Corresponding to \eqref{decompg}, we have an embedding $V^k(\gb) \otimes V^{\ell}(\ga) \hookrightarrow V^k(\gg)$, where the level $\ell$ is given as follows. 
 \begin{enumerate}
 \item In cases 1B, 1D, 2C, and 2O, $\ell = k$,
 \item In cases 1C and 1O,  $\ell=-\frac{k}{2}$,
 \item In cases 2B and 2D,  $\ell = -2k$.
 \end{enumerate}

Let $f_{\gb} \in \gg$ be the nilpotent element which is principal in $\mathfrak{b}$ and trivial in $\mathfrak{a}$. The corresponding $\cW$-algebras $\cW^k(\gg, f_{\gb})$ will be called {\it hook-type $\cW$-algebras} since they are analogous to the hook-type $\cW$-algebras of type $A$ introduced in \cite{CL3}. Let $d_{\mathfrak{a}} = \text{dim} \ \rho_{\mathfrak{a}}$ and $d_{\mathfrak{b}} = \text{dim} \ \rho_{\mathfrak{b}}$. In particular, $d_{\mathfrak{b}} = 2m+1$ if $\mathfrak{b} = \mathfrak{so}_{2m+1}$, and $d_{\mathfrak{b}} = 2m$ if $\mathfrak{b} = \mathfrak{sp}_{2m}$.

It follows from the decomposition \eqref{decompg} that in all cases, $\cW^{k}(\mathfrak{g}, f_{\gb})$ is of type
$$\cW \bigg(1^{\text{dim}\ \mathfrak{a}}, 2,4,\dots, 2m, \bigg(\frac{{d_{\mathfrak{b}}} + 1}{2} \bigg)^{d_{\mathfrak{a}}}\bigg).$$
The affine subalgebra is $V^t(\mathfrak{a})$ for some level $t$, which we describe below. The fields in weights $2,4,\dots, 2m$ are even and are invariant under $\mathfrak{a}$. The $d_{\mathfrak{a}}$ fields in weight $\frac{d_{\mathfrak{b}} + 1}{2}$ can be even or odd, and they transform as the standard $\mathfrak{a}$-module. By \cite[Cor. 3.5]{CL3}, we may assume without loss of generality that the fields in weights $2,4,\dots, 2m$ lie in the affine coset $\text{Com}(V^t(\ga), \cW^k(\gg,f_{\gb}))$, and that the $d_{\ga}$ fields in weight $\frac{d_{\mathfrak{b}} + 1}{2}$ are primary for the action of $V^t(\ga)$. 

Write $\gg = \bigoplus_d \rho_d$, where $\rho_d$ denotes the $d$-dimensional representation of the $\gs\gl_2$-triple $\{f,x,e\}$. Then each $\rho_d$ gives rise to a field of conformal weight $\frac{d+1}{2}$ in $\cW^k(\gg,f_{\gb})$, and the corresponding ghosts give rise to a central charge contribution 
\begin{equation} \label{ghostcontrib}
c_d = -\frac{(d-1)(d^2-2d-1)}{2}.\end{equation}
The central charge $c$ of $\cW^k(\gg, f_{\gb})$ is then computed to be
\begin{equation}
\begin{split}
c &= c_{\gg} + c_{\text{dilaton}} + c_{\text{ghost}}, \\
c_{\gg} &=     \frac{k\ \text{sdim} \, \gg}{ k+ h^\vee_{\gg}},   \\
c_{\text{dilaton}} &=  -k  \times \begin{cases}   2m(m+1)(2m+1) & \qquad  \gb = \gs\go_{2m+1}
\\ 2m(4m^2-1) &\qquad \gb = \gs\gp_{2m}, \\
\end{cases}  
\\ c_{\text{prin}}  & =   6m^2 -8m^4,
\\ c_{\text{ghost}} &=  c_{\text{prin}}  + sd_{\ga} c_{d_{\gb}}. 
\end{split}
\end{equation}
Here the formula for $c_{d_{\gb}}$ is obtained by specializing \eqref{ghostcontrib}.
Finally, the level $t$ of the affine subalgebra $V^t(\ga) \subseteq \cW^k(\gg,f_{\gb})$ is given by 
\begin{equation}
t = \ell \pm (d_{\gb} -1) \times \begin{cases} 1   &\qquad  \ga = \gs\go_{2n}, \ \gs\go_{2n+1}, \\  \frac{1}{2} &\qquad \ga = \gs\gp_{2n},  \ \go\gs\gp_{1|2n}. \end{cases}
\end{equation}
where we have $+$ if $\rho_{\ga}\otimes \rho_{\gb}$ is even and $-$ if it is odd. For $\ga =\go\gs\gp_{1|2n}$, we must replace $d_{\ga}$ by $d_{\ga} -2=sd_{\ga}$, where $sd_{\ga}$ means superdimension. Recall that in this case, $\rho_{\ga}$ is called even (respectively odd) if the $2n$-dimensional standard module for $\gs\gp_{2n}$ is even (respectively odd). We will always replace $k$ with the {\it critically shifted level $\psi = k+h^{\vee}$}, where $h^{\vee}$ denotes the dual Coxeter number of $\gg$. 
We now describe the examples we need in more detail.

\subsection{Case 1B}  For $\gg = \mathfrak{so}_{2n+2m+2}$, we have $\psi = k + 2n+2m$. We define
$$\cW^{\psi}_{1B}(n,m): = \cW^k(\mathfrak{so}_{2n+2m+2}, f_{\mathfrak{so}_{2m+1}}),$$ which has affine subalgebra $V^{\psi-2n}(\mathfrak{so}_{2n+1})$. We consider the following extreme cases.

\begin{enumerate} 
\item If $m\geq 1$ and $n=0$, $f_{\gs\go_{2m+1}}$ is also the principal nilpotent in $\gs\go_{2n+2}$, so 
$\cW^{\psi}_{1B}(0,m) = \cW^{\psi-2m}(\mathfrak{so}_{2m+2})$.

\item For $m=0$ and $n\geq 1$, $f_{\mathfrak{so}_{1}}\in \gs\go_{2n+2}$ is the zero nilpotent, so $\cW^{\psi}_{1B}(n,0) = V^{\psi-2n}(\gs\go_{2n+2})$.

\item If $m = n = 0$, $\cW^{\psi}_{1B}(0,0)= V^{\psi}(\gs\go_{2})$, which is just the rank one Heisenberg algebra $\cH(1)$.

\end{enumerate}

\subsection{Case 1C} For $\gg = \mathfrak{osp}_{2m+1|2n}$, we have $\psi =  k + 2 m - 2 n-1$. We define
$$\cW_{1C}^{\psi}(n,m) :=  \cW^{k}(\mathfrak{osp}_{2m+1|2n}, f_{\mathfrak{so}_{2m+1}}),$$ which has affine subalgebra $V^{-\psi/2 - n - 1/2}(\gs\gp_{2n})$. Here we are using the convention \eqref{convention:osp} that $\go\gs\gp_{2m+1|2n}$ has dual Coxeter number $2m-2n-1$. 
\begin{enumerate}

\item If $m \geq 1$ and $n=0$, $\mathfrak{g} =\mathfrak{so}_{2m+1}$ and $f_{\mathfrak{so}_{2m+1}}$ is the principal nilpotent, so 
$\cW_{1C}^{\psi}(0,m)  = \cW^{\psi -2m +1}(\gs\go_{2m+1})$.

\item If $m=0$ and $n\geq 1$, $\gg = \mathfrak{osp}_{1|2n}$ and $f_{\mathfrak{so}_{1}} = 0$, so 
$\cW_{1C}^{\psi}(n,0) = V^{\psi +2n+1}(\mathfrak{osp}_{1|2n})$. Note that even for $m=0$, we use the convention \eqref{convention:osp} that $\go\gs\gp_{1|2n}$ has dual Coxeter number $-2n-1$. With this choice, we have the embedding $V^{-\psi-n-1/2}(\gs\gp_{2n}) \rightarrow V^{\psi +2n+1}(\mathfrak{osp}_{1|2n})$. 

\item If $m = n = 0$, $\cW_{1C}^{\psi}(0,0) = \mathbb{C}$.

\end{enumerate}

\subsection{Case 1D} For $\gg =\mathfrak{so}_{2n+2m+1}$, we  have $\psi = k + 2n+2m-1$. We define 
$$\cW_{1D}^{\psi}(n,m) := \cW^{k}(\mathfrak{so}_{2n+2m+1}, f_{\mathfrak{so}_{2m+1}}),$$ which has affine subalgebra $V^{\psi-2n+1}(\mathfrak{so}_{2n})$. 

\begin{enumerate}

\item If $m\geq 1$ and $n=0$, $\gg = \gs\go_{2m+1}$ and $f_{\mathfrak{so}_{2m+1}}$ is principal, so
$\cW_{1D}^{\psi}(0,m) = \cW^{\psi -2m +1}(\gs\go_{2m+1})$.

\item  If $m=0$ and $n\geq 1$, $\gg = \gs\go_{2n+1}$ and $f_{\gs\go_1}=0$, so $\cW_{1D}^{\psi}(n,0) = V^{\psi -2n +1}(\gs\go_{2n+1})$.
 
\item If $m\geq 1$ and $n=1$, $\gg = \gs\go_{2m+3}$ and $f_{\mathfrak{so}_{2m+1}} \in  \gs\go_{2m+3}$ is the subregular nilpotent, so
$\cW_{1D}^{\psi}(1,m) = \cW^{\psi -2m-1}(\gs\go_{2m+3}, f_{\text{subreg}})$. In this case, the affine subalgebra $V^{\psi-1}(\mathfrak{so}_{2})$ is just $\cH(1)$.

\item If $m = n = 0$, $\cW_{1D}^{\psi}(0,0) = \mathbb{C}$.
\end{enumerate}

\subsection{Case 1O}  For $\gg = \mathfrak{osp}_{2m+2|2n}$, we have $\psi = k + 2 m - 2 n$. We define 
$$\cW_{1O}^{\psi}(n,m) = \cW^{k}(\mathfrak{osp}_{2m+2|2n}, f_{\mathfrak{so}_{2m+1}}),$$ which has affine subalgebra
$V^{- \psi/2 -n}(\go\gs\gp_{1|2n})$. We are using the convention \eqref{convention:osp} that $\go\gs\gp_{2m+2|2n}$ has dual Coxeter number $2m-2n$, whereas the dual Coxeter number of $\go\gs\gp_{1|2n}$ is taken to be $\frac{2n+1}{2}$.
\begin{enumerate}

\item If $m \geq 1$ and $n = 0$, $\gg = \gs\go_{2m+2}$ and $f_{\gs\go_{2m+1}}$ is the principal nilpotent, so $\cW_{1O}^{\psi}(0,m) = \cW^{\psi -2m}(\gs\go_{2m+2})$.

\item If $m=0$ and $n\geq 1$, we have $\gg = \mathfrak{osp}_{2|2n}$ and $f_{\gs\go_{1}} = 0$, so 
$\cW_{1O}^{\psi}(n,0) = V^{\psi +2n}(\mathfrak{osp}_{2|2n})$. As above, even for $m=0$, we use the convention \eqref{convention:osp} that $\go\gs\gp_{2|2n}$ has dual Coxeter number $-2n$. We then have an embedding
$$V^{-\psi/2-n}(\go\gs\gp_{1|2n}) \rightarrow V^{\psi+2n}(\go\gs\gp_{2|2n}),$$ where the dual Coxeter number of $\go\gs\gp_{1|2n}$ is chosen to be $\frac{2n+1}{2}$.

\item If $m = n = 0$, $\cW_{1O}^{\psi}(0,0) =  V^{\psi}(\mathfrak{osp}_{2|0}) = \cH(1)$.

\end{enumerate}

\subsection{Case 2B} For $\gg = \mathfrak{osp}_{2n+1|2m}$, we have $\psi =  k + m - n+1/2$. For $m\geq 1$ and $n\geq 0$, we define 
$$\cW_{2B}^{\psi}(n,m) := \cW^{k}(\mathfrak{osp}_{2n+1|2m}, f_{\mathfrak{sp}_{2m}}),$$ which has affine subalgebra $V^{-2 \psi - 2 n + 2}(\gs\go_{2n+1})$. 
We are using the convention \eqref{convention:osp} that $\mathfrak{osp}_{2n+1|2m}$ has dual Coxeter number $\frac{2m-2n+1}{2}$.
\begin{enumerate}
\item If $m \geq 1$ and $n=0$, $\gg = \mathfrak{osp}_{1|2m}$ and $f_{\gs\gp_{2m}}$ is the principal nilpotent, so $\cW_{2B}^{\psi}(0,m)$ is the principal $\cW$-superalgebra $\cW^{\psi -m -1/2}(\go\gs\gp_{1|2m})$. 

\item If $m = 1$ and $n\geq 1$, $\gg =  \mathfrak{osp}_{2n+1|2}$ and $f_{\mathfrak{sp}_{2}}$ is the minimal nilpotent, so 
$$\cW_{2B}^{\psi}(n,1) = \cW^{\psi +n -3/2}(\mathfrak{osp}_{2n+1|2}, f_{\text{min}}).$$

\item If $m = 0$ and $n\geq 1$, we need a different definition. We set $$\cW_{2B}^{\psi}(n,0) := V^{-2 \psi - 2 n + 1}(\gs\go_{2n+1}) \otimes \cF(2n+1).$$ Here $\cF(2n+1)$ is the rank $2n+1$ free fermion algebra, which has an action of $L_1(\gs\go_{2n+1})$. Therefore $\cW_{2B}^{\psi}(n,0)$ has a diagonal action of $V^{-2 \psi - 2 n + 2}(\gs\go_{2n+1})$.

\item If $m = n = 0$, we define $\cW_{2B}^{\psi}(0,0) = \cF(1)$.

\end{enumerate}

\subsection{Case 2C} For $\gg = \mathfrak{sp}_{2n+2m}$, we have $\psi = k + n+m+1$. For $m\geq 1$ and $n\geq 0$, we define 
$$\cW_{2C}^{\psi}(n,m): = \cW^{k}(\mathfrak{sp}_{2n+2m}, f_{\mathfrak{sp}_{2m}}),$$ which has affine subalgebra 
$V^{\psi-n-3/2}(\gs\gp_{2n})$.

\begin{enumerate}
\item If $m \geq 1$ and $n = 0$, $\gg = \gs\gp_{2m}$ and $f_{\gs\gp_{2m}}$ is the principal nilpotent, so $\cW_{2C}^{\psi}(0,m) = \cW^{\psi-m-1}(\gs\gp_{2m})$.

\item If $m=1$ and $n\geq 1$, $\gg = \mathfrak{sp}_{2n+2}$ and $f_{\gs\gp_{2m}}$ is the minimal nilpotent. Then 
$$\cW_{2C}^{\psi}(n,1)  = \cW^{\psi - n-2}(\gs\gp_{2n+2}, f_{\text{min}}).$$

\item If $m = 0$ and $n\geq 1$, we define
$$\cW_{2C}^{\psi}(n,0) := V^{\psi-n-1} (\gs\gp_{2n}) \otimes \cS(n).$$ Here $\cS(n)$ is the rank $n$ $\beta\gamma$ system, which has an action of $L_{-1/2}(\gs\gp_{2n})$. Therefore $\cW_{2C}^{\psi}(n,0)$ has a diagonal action of $V^{\psi-n-3/2}(\gs\gp_{2n})$.

\item If $m = n = 0$, we define $\cW_{2C}^{\psi}(0,0) = \mathbb{C}$.

\end{enumerate}

\subsection{Case 2D} For $\gg =\mathfrak{osp}_{2n|2m}$, we have $\psi = k + m - n+1$. For $m\geq 1$ and $n\geq 0$, we define 
$$\cW_{2D}^{\psi}(n,m) := \cW^{k}(\mathfrak{osp}_{2n|2m}, f_{ \mathfrak{sp}_{2m}}),$$ which has affine subalgebra
$V^{- 2 \psi - 2 n +3}(\gs\go_{2n})$. We are using the convention \eqref{convention:osp} that  $\mathfrak{osp}_{2n|2m}$ has dual Coxeter number $m-n+1$.

\begin{enumerate}

\item If $m\geq 1$ and $n=0$, $\gg = \gs\gp_{2m}$ and $f_{\gs\gp_{2m}}$ is the principal nilpotent. Then $\cW_{2D}^{\psi}(0,m)= \cW^{\psi-m-1}(\gs\gp_{2m})$.

\item If $m\geq 1$ and $n = 1$, $\gg =\mathfrak{osp}_{2|2m}$ and $f_{\gs\gp_{2m}}$ is the principal nilpotent, so $\cW_{2D}^{\psi}(1,m)$ is the principal $\cW$-superalgebra $\cW^{\psi - m}(\mathfrak{osp}_{2|2m})$. Note that in this case, the affine subalgebra $V^{- 2 \psi +1 }(\gs\go_{2})$ is just the Heisenberg algebra $\cH(1)$.

\item If $m = 1$ and $n\geq 1$, $\gg =\mathfrak{osp}_{2n|2}$, and $f_{\gs\gp_2}$ is the minimal nilpotent so $\cW_{2D}^{\psi}(n,1) = \cW^{\psi +n-2}(\mathfrak{osp}_{2n|2}, f_{\text{min}})$.

\item If $m = 0$ and $n\geq 1$, we define 
$$ \cW_{2D}^{\psi}(n,0): = V^{- 2 \psi - 2 n +2}(\gs\go_{2n}) \otimes \cF(2n).$$ Here $\cF(2n)$ is the rank $2n$ free fermion algebra, which admits an action of $L_1(\gs\go_{2n})$. Then $\cW_{2D}^{\psi}(n,0)$ admits a diagonal action of $V^{- 2 \psi - 2 n +3}(\gs\go_{2n})$. 

\item If $m = n = 0$, we define $\cW_{2D}^{\psi}(0,0) = \mathbb{C}$.

\end{enumerate}

\subsection{Case 2O} For $\gg = \mathfrak{osp}_{1|2n+2m}$, we have $\psi = k + m + n+1/2$. For $m\geq 1$ and $n\geq 0$, we define  
$$\cW_{2O}^{\psi}(n,m) := \cW^{k}(\mathfrak{osp}_{1|2n+2m}, f_{\mathfrak{sp}_{2m}}),$$ which has affine subalgebra $V^{\psi -n-1}(\go\gs\gp_{1|2n})$.
We are using the convention \eqref{convention:osp} that  $\mathfrak{osp}_{1|2n+2m}$ has dual Coxeter number $\frac{2m+2n+1}{2}$, and the dual Coxeter number of $\go\gs\gp_{1|2n}$ is taken to be $\frac{2n+1}{2}$.
\begin{enumerate}

\item If $m\geq 1$ and $n=0$, $\gg = \mathfrak{osp}_{1|2m}$ and $f_{\gs\gp_{2m}}$ is principal. Then $\cW_{2O}^{\psi}(0,m)= \cW^{\psi -m -1/2}(\go\gs\gp_{1|2m})$.

\item If $m = 1$ and $n\geq 1$, $\gg = \mathfrak{osp}_{1|2n+2}$, and $f_{\gs\gp_2}$ is the minimal nilpotent, so $\cW_{2O}^{\psi}(n,1) = \cW^{\psi -n -3/2}(\mathfrak{osp}_{1|2n+2}, f_{\text{min}})$.

\item If $m = 0$ and $n\geq 1$, we define 
$$ \cW_{2O}^{\psi}(n,0) : = V^{\psi-n-1/2}(\mathfrak{osp}_{1|2n}) \otimes \cS(n) \otimes \cF(1).$$ Recall that $ \cS(n) \otimes \cF(1)$ admits an action of $L_{-1/2}(\go\gs\gp_{1|2n})$, so $\cW_{2O}^{\psi}(n,0)$ admits a diagonal action of $V^{\psi -n-1}(\go\gs\gp_{1|2n})$.

\item If $m = n = 0$, we define $\cW_{2O}^{\psi}(0,0) = \cF(1)$.

\end{enumerate}

\subsection{Affine cosets} \label{subsec:affinecoset} Our main objects of study are the {\it affine cosets} of these $\cW$-algebras. In fact, in the case where $\ga$ is either $\gs\go_{2n}$, $\gs\go_{2n+1}$, or $\go\gs\gp_{1|2n}$, the action of $\ga$ integrates to an action of the corresponding connected (super)group $\text{SO}_{2n}$, $\text{SO}_{2n+1}$, or $\text{SOsp}_{1|2n}$, and this action further extends to the double cover $\text{O}_{2n}$, $\text{O}_{2n+1}$, or $\text{Osp}_{1|2n}$. In these cases, we need to take the $\mathbb{Z}_2$-orbifold of the corresponding affine coset. Here is the list of these algebras.

\smallskip
\noindent {\bf Case 1B}:
\begin{equation}  \cC^{\psi}_{1B}(n,m) = \left\{
\begin{array}{ll}
\text{Com}(V^{\psi-2n}(\mathfrak{so}_{2n+1}), \cW^{\psi}_{1B}(n,m))^{\mathbb{Z}_2} & m \geq 1,\  n \geq 1,
\smallskip
\\  \text{Com}(V^{\psi-2n}(\mathfrak{so}_{2n+1}), V^{\psi-2n}(\gs\go_{2n+2}))^{\mathbb{Z}_2} & m =0, \ n\geq 1,
\smallskip
\\ \cW^{\psi -2m}(\gs\go_{2m+2})^{\mathbb{Z}_2} & m \geq 1, \ n = 0,
\smallskip
\\ \cH(1)^{\mathbb{Z}_2} & m = n = 0.
\\ \end{array} 
\right.
\end{equation}
In all cases, $\cC^{\psi}_{1B}(n,m)$ has central charge
$$c = -\frac{( \psi + m \psi -m -n -1) (2 m \psi - 2 m - 2 n -1) ( \psi + 2 m \psi -2 m - 2 n)}{(\psi -1) \psi}.$$

\smallskip
\noindent {\bf Case 1C}:
\begin{equation}   \cC^{\psi}_{1C}(n,m) = \left\{
\begin{array}{ll}
 \text{Com}(V^{-\psi/2 - n - 1/2}(\gs\gp_{2n}), \cW_{1C}^{\psi}(n,m))
& m \geq 1,\  n \geq 1,
\smallskip
\\  \text{Com}(V^{-\psi/2 - n - 1/2}(\gs\gp_{2n}), V^{\psi +2n+1}(\mathfrak{osp}_{1|2n}))& m =0, \ n\geq 1,
 \smallskip
\\  \cW^{\psi-2m+1}(\gs\go_{2m+1})& m \geq 1, \ n = 0,
\smallskip
\\ \mathbb{C} & m = n = 0.
\\ \end{array} 
\right.
\end{equation}
If we define the central charge of $\mathbb{C}$ to be zero, then in all cases, $\cC^{\psi}_{1C}(n,m)$ has central charge 
$$c = -\frac{(-m + n + m \psi) (1 - 2 m + 2 n + \psi + 2 m \psi) (-1 - 2 m + 2 n + 2 \psi + 2 m \psi)}{( \psi -1) \psi}.$$

\smallskip
\noindent {\bf Case 1D}:
\begin{equation}  \cC^{\psi}_{1D}(n,m) = \left\{
\begin{array}{ll}
\text{Com}(V^{\psi-2n+1}(\mathfrak{so}_{2n}), \cW_{1D}^{\psi}(n,m) )^{\mathbb{Z}_2}
 & m \geq 1,\  n > 1,
 \smallskip
 \\ \text{Com}(\cH(1),  \cW^{\psi -2m-1}(\gs\go_{2m+3}, f_{\text{subreg}}))^{\mathbb{Z}_2}
 & m \geq 1,\  n = 1,
 \smallskip
 \\ \text{Com}(V^{\psi-2n+1}(\mathfrak{so}_{2n}), V^{\psi-2n+1}(\gs\go_{2n+1}))^{\mathbb{Z}_2} & m =0, \ n\geq 1,
 \smallskip
\\ \cW^{\psi-2m+1}(\gs\go_{2m+1}) & m \geq 1, \ n = 0,
\smallskip
\\ \mathbb{C} & m = n = 0.
\\ \end{array} 
\right.
\end{equation}
 In all cases, $\cC^{\psi}_{1D}(n,m)$ has central charge 
$$c =-\frac{(-m - n + m \psi) (1 - 2 m - 2 n + \psi + 2 m \psi) (-1 - 2 m - 2 n + 2 \psi + 2 m \psi)}{( \psi -1) \psi}.$$

\smallskip
\noindent {\bf Case 1O}:
\begin{equation}   \cC^{\psi}_{1O}(n,m) = \left\{
\begin{array}{ll}
\text{Com}(V^{- \psi/2 -n}(\go\gs\gp_{1|2n}), \cW_{1O}^{\psi}(n,m))^{\mathbb{Z}_2}
 & m \geq 1,\  n \geq 1,
 \smallskip
\\  \text{Com}(V^{- \psi/2 -n}(\go\gs\gp_{1|2n}), V^{\psi+2n}(\mathfrak{osp}_{2|2n})) & m =0, \ n\geq 1,
 \smallskip
\\ \cW^{\psi-2m}(\gs\go_{2m+2})^{\mathbb{Z}_2} & m \geq 0, \ n = 0,
\smallskip
\\ \cH(1)^{\mathbb{Z}_2} & m = n = 0.
\\ \end{array} 
\right.
\end{equation}
 In all cases, $\cC^{\psi}_{1O}(n,m)$ has central charge 
$$c = -\frac{(-1 - m + n + \psi + m \psi) (-1 - 2 m + 2 n + 2 m \psi) (-2 m + 2 n + \psi + 2 m \psi)}{(\psi -1) \psi}.$$

\smallskip
\noindent {\bf Case 2B}:
\begin{equation}   \cC^{\psi}_{2B}(n,m) = \left\{
\begin{array}{ll}
 \text{Com}(V^{-2 \psi - 2 n + 2}(\gs\go_{2n+1}), \cW_{2B}^{\psi}(n,m) )^{\mathbb{Z}_2}
 & m \geq 1,\  n \geq 1,
 \smallskip
\\ \text{Com}(V^{-2 \psi - 2 n + 2}(\gs\go_{2n+1}), V^{-2 \psi - 2 n + 1}(\gs\go_{2n+1}) \otimes \cF(2n+1))^{\mathbb{Z}_2} & m =0, \ n\geq 1,
 \smallskip
\\ \cW^{\psi-m-1/2}(\go\gs\gp_{1|2m})^{\mathbb{Z}_2} & m \geq 1, \ n = 0,
\smallskip
\\ \cF(1)^{\mathbb{Z}_2} & m = n = 0.
\\ \end{array} 
\right.
\end{equation}
 In all cases, $\cC^{\psi}_{2B}(n,m)$ has central charge 
$$ c = -\frac{(-m + n - \psi + 2 m \psi) (1 - 2 m + 2 n + 4 m \psi) (-1 - 2 m + 2 n + 2 \psi + 4 m \psi)}{2 \psi (2 \psi -1)}.$$

\smallskip
\noindent {\bf Case 2C}:
\begin{equation}   \cC^{\psi}_{2C}(n,m) = \left\{
\begin{array}{ll}
\text{Com}(V^{\psi-n-3/2}(\gs\gp_{2n}), \cW_{2C}^{\psi}(n,m))
 & m \geq 1,\  n \geq 1,
 \smallskip
\\ \text{Com}(V^{\psi-n-3/2}(\gs\gp_{2n}), V^{\psi-n-1} (\gs\gp_{2n}) \otimes \cS(n)) & m =0, \ n\geq 1,
 \smallskip
\\ \cW^{\psi -m-1}(\gs\gp_{2m})  & m \geq 1, \ n = 0,
\smallskip
\\ \mathbb{C} & m = n = 0.
\\ \end{array} 
\right.
\end{equation}
 In all cases, $\cC^{\psi}_{2C}(n,m)$ has central charge 
$$c = -\frac{(-m - n + 2 m \psi) (-1 - m - n + \psi + 2 m \psi) (-1 - 2 m - 2 n - 2 \psi + 4 m \psi)}{\psi (2 \psi -1)}.$$

\smallskip
\noindent {\bf Case 2D}:
\begin{equation}   \cC^{\psi}_{2D}(n,m) = \left\{
\begin{array}{ll}
 \text{Com}(V^{- 2 \psi - 2 n +3}(\gs\go_{2n}), \cW_{2D}^{\psi}(n,m) )^{\mathbb{Z}_2}
 & m \geq 1,\  n > 1,
 \smallskip
\\  \text{Com}(\cH(1), \cW^{\psi - m}(\mathfrak{osp}_{2|2m}) )^{\mathbb{Z}_2}
 & m \geq 1,\  n = 1,
 \smallskip
\\ \text{Com}(V^{- 2 \psi - 2 n +3}(\gs\go_{2n}), V^{- 2 \psi - 2 n +2}(\gs\go_{2n}) \otimes \cF(2n))^{\mathbb{Z}_2} & m =0, \ n\geq 1,
 \smallskip
\\ \cW^{\psi -m-1}(\gs\gp_{2m})  & m \geq 1, \ n = 0,
\smallskip
\\ \mathbb{C} & m = n = 0.
\\ \end{array} 
\right.
\end{equation}
 In all cases, $\cC^{\psi}_{2D}(n,m)$ has central charge
$$ c = -\frac{(-m + n + 2 m \psi) (-1 - m + n + \psi + 2 m \psi) (-1 - 2 m + 2 n - 2 \psi + 4 m \psi)}{\psi (2 \psi -1)}.$$

\smallskip
\noindent {\bf Case 2O}:
\begin{equation}  \cC^{\psi}_{2O}(n,m) = \left\{
\begin{array}{ll}
\text{Com}(V^{\psi -n-1}(\go\gs\gp_{1|2n}), \cW_{2O}^{\psi}(n,m) )^{\mathbb{Z}_2} & m \geq 1,\  n \geq 1,
 \smallskip
\\ \text{Com}(V^{\psi -n-1}(\go\gs\gp_{1|2n}), V^{\psi-n-1/2}(\mathfrak{osp}_{1|2n}) \otimes \cS(n) \otimes \cF(1))^{\mathbb{Z}_2} & m =0, \ n\geq 1,
 \smallskip
\\ \cW^{\psi -m -1/2}(\go\gs\gp_{1|2m})^{\mathbb{Z}_2}  & m \geq 1, \ n = 0,
\smallskip
\\ \cF(1)^{\mathbb{Z}_2} & m = n = 0.
\\ \end{array} 
\right.
\end{equation}
 In all cases, $\cC^{\psi}_{2O}(n,m)$ has central charge
$$c = -\frac{(-m - n - \psi + 2 m \psi) (1 - 2 m - 2 n + 4 m \psi) (-1 - 2 m - 2 n + 2 \psi + 4 m \psi)}{2 \psi (2 \psi - 1)}.$$

We shall regard $\psi$ as a formal variable and the algebras $\cC^{\psi}_{iX}(n,m)$ for $i = 1,2$ and $X = B,C,D,O$, as one-parameter vertex algebras with parameter $\psi$. If $\psi_0 \in \mathbb{C}$ is a complex number, $\cC^{\psi_0}_{iX}(n,m)$ will always denote the specialization of $\cC^{\psi}_{iX}(n,m)$ to the value $\psi = \psi_0$. For generic values of $\psi_0$, this coincides with the actual coset, although it can be a proper subalgebra of the coset if $\psi_0$ is a negative rational number.

\begin{thm} For $i=1,2$ and $X = B,C,D,O$, $\cC^{\psi}_{iX}(n,m)$ is simple as a one-parameter vertex algebra; equivalently this holds for generic values of $\psi$.
\end{thm}

\begin{proof} In all cases where $\cW^{\psi}_{iX}(n,m)$ is a quantum Hamiltonian reduction, namely the cases where $n+m \geq 1$, and $m\geq 1$ when $i=2$, the simplicity of $\cW^{\psi}_{iX}(n,m)$ and of its affine coset follows from parts (1) and (2) of \cite[Thm. 3.6]{CL3}. In the cases $X = B,D,O$ where $\cC^{\psi}_{iX}(n,m)$ is the $\mathbb{Z}_2$-orbifold of the affine coset, the simplicity follows from \cite{DLM}.

In the cases $\cW^{\psi}_{2X}(n,0)$, the simplicity of the affine coset follows from \cite[Prop. 5.4]{CGN}, and in the cases $X = B,D,O$, the simplicity of $\cC^{\psi}_{2X}(n,0)$ again follows from \cite{DLM}. Finally, the claim is obvious in the all cases when $n = m = 0$.
\end{proof}

\section{Main result} \label{sec:mainresult} 

The main result in this paper is analogous to \cite[Thm. 1.1]{CL3}.

\begin{thm} \label{main} For all integers $m\geq n \geq 0$, we have the following isomorphisms of one-parameter vertex algebras.
\begin{equation} \label{2b2b2o}  \cC^{\psi}_{2B}(n,m) \cong  \cC^{\psi'}_{2O}(n,m-n) \cong \cC^{\psi''}_{2B}(m,n), \qquad \psi' = \frac{1}{4\psi},  \qquad \frac{1}{\psi}  + \frac{1}{\psi''} = 2,
\end{equation}
\begin{equation} \label{1c1c2c} \cC^{\psi}_{1C}(n,m) \cong \cC^{\psi'}_{2C}(n,m-n) \cong  \cC^{\psi''}_{1C}(m,n), \qquad \psi'= \frac{1}{2\psi}, \qquad \frac{1}{\psi} + \frac{1}{\psi''} = 1,\end{equation}
\begin{equation}  \label{2d1d1o} \cC^{\psi}_{2D}(n,m) \cong  \cC^{\psi'}_{1D}(n,m-n ) \cong \cC^{\psi''}_{1O}(m,n-1),\qquad \psi' = \frac{1}{2\psi}, \qquad \frac{1}{2\psi} + \frac{1}{\psi''} = 1,
\end{equation}
\begin{equation} \label{1o1b2d} \cC^{\psi}_{1O}(n,m) \cong \cC^{\psi'}_{1B}(n,m-n) \cong  \cC^{\psi''}_{2D}(m+1,n) ,\qquad \psi' = \frac{1}{\psi},  \qquad  \frac{1}{\psi}  + \frac{1}{2\psi''} = 1.
\end{equation}
\end{thm}
Note that $\cC^{\psi}_{2D}(n,m)$ for $m\geq n$ and $\cC^{\psi}_{2D}(n,m)$ for $m<n$ belong in different families, and similarly for $\cC^{\psi}_{1O}(n,m)$. For the rest of this section we discuss some special cases of this result.

\subsection{Special cases of  \eqref{2b2b2o}}
In the case $n=0$ of \eqref{2b2b2o}, we have
\begin{equation*} \begin{split}  & \cC^{\psi}_{2B}(0,m) = \cW^{\psi-m-1/2}(\go\gs\gp_{1|2m})^{\mathbb{Z}_2} ,
\\ & \cC^{\psi'}_{2O}(0,m) = \cW^{\psi' -m -1/2}(\go\gs\gp_{1|2m})^{\mathbb{Z}_2}.\end{split} \end{equation*} The isomorphism $\cC^{\psi}_{2B}(0,m) \cong  \cC^{\psi'}_{2O}(0,m)$ is the $\mathbb{Z}_2$-invariant part of Feigin-Frenkel duality for principal $\cW$-algebras of $\go\gs\gp_{1|2m}$, which was proven in a different way in \cite{CGe}.

\begin{remark} \label{rem:cosetZ2} A special case of Theorem \ref{thm:uniquenesswnm} is that the OPE algebra of $\cW^{\psi' -m -1/2}(\go\gs\gp_{1|2m})$, which is a simple current extension of $\cW^{\psi' -m -1/2}(\go\gs\gp_{1|2m})^{\mathbb{Z}_2}$
by an odd field of weight $\frac{2m+1}{2}$, is uniquely determined by $\cW^{\psi' -m -1/2}(\go\gs\gp_{1|2m})^{\mathbb{Z}_2}$. Therefore our result implies the full Feigin-Frenkel duality $\cW^{\psi-m-1/2}(\go\gs\gp_{1|2m})  \cong \cW^{\psi' -m -1/2}(\go\gs\gp_{1|2m})$.
\end{remark}

In the case $m=0$ of \eqref{2b2b2o}, we have
\begin{equation*} \begin{split}  & \cC^{\psi}_{2B}(n,0) = \text{Com}(V^{-2 \psi - 2 n + 2}(\mathfrak{so}_{2n+1}), V^{-2 \psi - 2 n + 1}(\mathfrak{so}_{2n+1}) \otimes \cF(2n+1))^{\mathbb{Z}_2} ,
\\ & \cC^{\psi''}_{2B}(0,n) =  \cW^{\psi'' -n -1/2}(\mathfrak{osp}_{1|2n})^{\mathbb{Z}_2}.
\end{split} \end{equation*}

Therefore the isomorphism $\cC^{\psi}_{2B}(n,0)  \cong \cC^{\psi''}_{2B}(0,n)$ implies that both $\cW^{\psi'' -n -1/2}(\mathfrak{osp}_{1|2n})$ and $\text{Com}(V^{-2 \psi - 2 n + 2}(\mathfrak{so}_{2n+1}), V^{-2 \psi - 2 n + 1}(\mathfrak{so}_{2n+1}) \otimes \cF(2n+1))$ are simple current extensions of $\cW^{\psi'' -n -1/2}(\mathfrak{osp}_{1|2n})^{\mathbb{Z}_2}$ by an odd field in weight $\frac{2n+1}{2}$. As above, Theorem \ref{thm:uniquenesswnm} then implies
\begin{equation} \begin{split} \cW^{\psi'' -n -1/2}(\mathfrak{osp}_{1|2n}) &  \cong \text{Com}(V^{-2 \psi - 2 n + 2}(\mathfrak{so}_{2n+1}), V^{-2 \psi - 2 n + 1}(\mathfrak{so}_{2n+1}) \otimes \cF(2n+1))
\\ & \cong \text{Com}(V^{-2 \psi - 2 n + 2}(\mathfrak{so}_{2n+1}), V^{-2 \psi - 2 n + 1}(\mathfrak{so}_{2n+1}) \otimes L_1(\gs\go_{2n+1})).\end{split} \end{equation}
We recover the coset realization of principal $\cW$-superalgebras of $\go\gs\gp_{1|2n}$, which was proven in a different way in \cite{CGe}.

\subsection{Special cases of  \eqref{1c1c2c}}

In the case $n=0$, the isomorphism $\cC^{\psi}_{1C}(0,m) \cong \cC^{\psi'}_{2C}(0,m)$ for $\psi'= \frac{1}{2\psi}$, is just Feigin-Frenkel duality in types $B$ and $C$, since $\cC^{\psi}_{1C}(0,m) = \cW^{\psi-2m+1}(\mathfrak{so}_{2m+1})$ and $\cC^{\psi'}_{2C}(0,m) \cong \cW^{\psi' - m - 1}(\gs\gp_{2m})$.

In the case $m=0$, we have
\begin{equation*} \begin{split}  & \cC^{\psi}_{1C}(n,0) =  \text{Com}(V^{-\psi/2 - n - 1/2}(\mathfrak{sp}_{2n}), V^{\psi +2n+1}(\mathfrak{osp}_{1|2n})),
\\ & \cC^{\psi''}_{1C}(0,n) = \cW^{\psi''-2n+1}(\mathfrak{so}_{2n+1}),
\end{split}\end{equation*}
so the isomorphism $\cC^{\psi}_{1C}(n,0) \cong  \cC^{\psi''}_{1C}(0,n)$ yields a new coset realization of type $B$ and $C$ principal $\cW$-algebras. Recall that we are using the convention \eqref{convention:osp} that $\go\gs\gp_{1|2n}$ has dual Coxeter number $-2n-1$. If we instead use the dual Coxeter number $\frac{2n+1}{2}$, so that $V^k(\gs\gp_{2n})$ embeds in $V^k(\go\gs\gp_{1|2n})$, then we have
$\cC^{\psi}_{1C}(n,0) = \text{Com}(V^k(\gs\gp_{2n}), V^k(\go\gs\gp_{1|2n}))$ for $k = - \frac{1}{2} (\psi +2n +1)$. We obtain

\begin{cor} \label{cosetrealizationBCk} For all $n\geq 1$, we have the following isomorphism of one-parameter vertex algebras
\begin{equation} \text{Com}(V^k(\gs\gp_{2n}), V^k(\go\gs\gp_{1|2n})) \cong \cW^r(\gs\go_{2n+1}), \qquad r = -(2n-1) + \frac{1 + 2 k + 2 n}{2 (1 + k + n)}. 
\end{equation}  
\end{cor}
This realization of $\cW^r(\gs\go_{2n+1})$ is very different from the coset realization of $\cW^{\ell}(\gg)$ for $\gg$ simply-laced given in \cite{ACL} since it involves affine vertex superalgebras. Although we are not aware of this statement being previously conjectured in the literature, both algebras were known to have the same strong generating type and graded character; see \cite[Cor. 7.4]{CL1} and \cite[Cor 5.7]{CL2}.

\subsection{Special cases of  \eqref{2d1d1o}}
In the case $n=0$, the isomorphism $\cC^{\psi}_{2D}(0,m) \cong  \cC^{\psi'}_{1D}(0,m)$ is again Feigin-Frenkel duality in types $B$ and $C$, since
\begin{equation*} \begin{split} & \cC^{\psi}_{2D}(0,m) = \cW^{\psi-m-1}(\mathfrak{sp}_{2m}),
\\ & \cC^{\psi'}_{1D}(0,m) = \cW^{\psi'-2m+1}( \mathfrak{so}_{2m+1}).\end{split} \end{equation*}

In the case $n=1$ and $m\geq 1$, we have
\begin{equation*} \begin{split} 
& \cC^{\psi}_{2D}(1,m) = \text{Com}(\cH(1), \cW^{\psi-m}(\mathfrak{osp}_{2|2m}))^{\mathbb{Z}_2},
\\ & \cC^{\psi'}_{1D}(1,m-1) = \text{Com}(\cH(1), \cW^{\psi' -2m +1}(\mathfrak{so}_{2m+1}, f_{\text{subreg}}))^{\mathbb{Z}_2}.
\end{split} \end{equation*}
Therefore the isomorphism $\cC^{\psi}_{2D}(1,m) \cong  \cC^{\psi'}_{1D}(1,m-1)$ recovers the $\mathbb{Z}_2$-invariant part of the duality 
\begin{equation} \label{eq:CGN} \text{Com}(\cH(1), \cW^{\psi' -2m +1}(\mathfrak{so}_{2m+1}, f_{\text{subreg}})) \cong \text{Com}(\cH(1), \cW^{\psi-m}(\mathfrak{osp}_{2|2m})),\end{equation} of Genra, Nakatsuka and one of us \cite{CGN}.

\begin{remark} The isomorphism $\cC^{\psi}_{2D}(1,m)  \cong \cC^{\psi'}_{1D}(1,m-1)$ can be used to give a new proof of \eqref{eq:CGN} as follows. Both $ \text{Com}(\cH(1), \cW^{\psi' -2m +1}(\mathfrak{so}_{2m+1}, f_{\text{subreg}}))$ and $\text{Com}(\cH(1), \cW^{\psi-m}(\mathfrak{osp}_{2|2m}))$ are simple current extensions of $\cC^{\psi}_{2D}(1,m)$, where the extension is generated by an even field $\omega$ in weight $2m+1$. The generators of $\cC^{\psi}_{2D}(1,m)$ and $\omega$ do not close under OPE, and new strong generators are needed in weights $2m+3, 2m+5, \dots$. It can be shown that there is a unique simple current extension of $\cC^{\psi}_{2D}(1,m)$ with these properties, such that $\omega_{(4m+1)} \omega \neq 0$. The proof is similar to, but more involved than the proof of Theorem \ref{thm:uniquenesswnm}, and is omitted. \end{remark}

In the case $n=1$ and $m = 0$ of  \eqref{2d1d1o}, we have $\cW_{2D}^{\psi}(1,0): = \cH(1) \otimes \cF(2)$ and $\cC^{\psi}_{2D}(1,0) = \text{Com}(\cH(1), \cH(1) \otimes \cF(2))^{\mathbb{Z}_2} \cong \cF(2)^{\text{O}_2}$. Also, $\cC^{\psi}_{1O}(0,0) \cong \cH(1)^{\mathbb{Z}_2}$. We recover the isomorphism $\cH(1)^{\mathbb{Z}_2} \cong \cF(2)^{\text{O}_2}$.

\subsection{Special cases of  \eqref{1o1b2d}}

In the case $n=0$, the isomorphism $\cC^{\psi}_{1O}(0,m) \cong \cC^{\psi'}_{1B}(0,m)$ is just the $\mathbb{Z}_2$-invariant part of Feigin-Frenkel duality in type $D$, since 
\begin{equation*}\begin{split} & \cC^{\psi}_{1O}(0,m) = \cW^{\psi-2m}(\mathfrak{so}_{2m+2})^{\mathbb{Z}_2},
\\ & \cC^{\psi'}_{1B}(0,m) =  \cW^{\psi'-2m}(\mathfrak{so}_{2m+2})^{\mathbb{Z}_2}.
\end{split} \end{equation*} As in Remark \ref{rem:cosetZ2}, this statement together with Theorem \ref{thm:uniquenesswnm}, gives a new proof of the full duality.

In the case $m=0$ and $n\geq 2$, we have 
\begin{equation} \begin{split} \cC^{\psi''}_{2D}(n,0) & \cong \text{Com}(V^{- 2 \psi'' - 2 n +3}(\gs\go_{2n}), V^{- 2 \psi'' - 2 n +2}(\gs\go_{2n}) \otimes \cF(2n))^{\mathbb{Z}_2}
\\ & \cong \text{Com}(V^{- 2 \psi'' - 2 n +3}(\gs\go_{2n}), V^{- 2 \psi'' - 2 n +2}(\gs\go_{2n}) \otimes L_1(\gs\go_{2n}))^{\mathbb{Z}_2}. \end{split}\end{equation} Since $\cC^{\psi}_{1O}(0,n-1) \cong  \cW^{\psi - 2n+2}(\gs\go_{2n})^{\mathbb{Z}_2}$, the isomorphism $\cC^{\psi}_{1O}(0,n-1) \cong   \cC^{\psi''}_{2D}(n,0)$ recovers the $\mathbb{Z}_2$-invariant part of the coset realization of principal $\cW$-algebras of type $D$ proven in \cite{ACL}. Finally, this statement together with Theorem \ref{thm:uniquenesswnm}, gives a new proof of the coset realization of $\cW^{\psi' - 2n+2}(\gs\go_{2n})$.

\subsection{Sketch of proof} The proof of Theorem \ref{main} involves the following steps.

\begin{enumerate}

\item Using the free field limits of $\cW^{\psi}_{iX}(n,m)$, together with some classical invariant theory, we find strong generating sets for $\cC^{\psi}_{iX}(n,m)$ for $i=1,2$ and $X = B,C,D,O$. If $\ga = \gs\gp_{2n}$, and if $\ga = \gs\go_{2n}$ or $\ga = \gs\go_{2n+1}$ and $\rho_{\ga} \otimes \rho_{\gb}$ is odd, we will find {\it minimal} strong generating sets. In the remaining cases, namely, $\ga = \go\gs\gp_{1|2n}$, and $\ga = \gs\go_{2n}$ or $\ga = \gs\go_{2n+1}$ and $\rho_{\ga} \otimes \rho_{\gb}$ is even, we will not find minimal strong generating sets at this stage, but we will deduce them later as a consequence of Theorem \ref{main}.

\item We show that in all cases, $\cC^{\psi}_{iX}(n,m)$ has a subalgebra $\tilde{\cC}^{\psi}_{iX}(n,m)$ generated by the weights $2$ and $4$ field, which is isomorphic to a quotient $\cW$ of $\cW^{{\rm ev},I_{iX,n,m}}(c,\lambda)$, for some ideal $I_{iX,n,m}\subseteq \mathbb{C}[c,\lambda]$. In particular, $\cW^{\psi}_{iX}(n,m)$ is an extension of $V^t(\ga) \otimes \cW$ by $d_{\ga}$ fields of appropriate parity in weight $\frac{d_{\gb}+1}{2}$ which transform as the standard $\ga$-module.

\item We show that the existence of such an extension of $V^t(\ga) \otimes \cW$ uniquely and explicitly specifies the ideal $I_{iX,n,m}$.

\item We compute coincidences between the simple quotient $\tilde{\cC}_{\psi, iX}(n,m)$ and principal $\cW$-algebras of type $C$, by finding intersection points on their truncation curves. Using Corollary \ref{cor:singularvector}, we prove that $\tilde{\cC}^{\psi}_{iX}(n,m)= \cC^{\psi}_{iX}(n,m)$ as one-parameter vertex algebras. In particular, $\cC^{\psi}_{iX}(n,m)$ is isomorphic to the simple quotient  $\cW^{{\rm ev}}_{I_{iX,n,m}}(c,\lambda)$ in all cases.

\item The isomorphisms in Theorem \ref{main} all follow from the explicit formulas for $I_{iX,n,m}$.
\end{enumerate}

\section{Large level limits} \label{sec:largelevel}
In this section, we describe the large level limits of $\cW^{\psi}_{iX}(n,m)$ and the strong generating types of $\cC^{\psi}_{iX}(n,m)$.

\subsection{Case 1C} Recall that $\cW^{\psi}_{1C}(n,m)$ has affine subalgebra $V^{-\psi/2 - n - 1/2}( \gs\gp_{2n})$, even fields in weights $2,4,\dots, 2m$ which are invariant under $\gs\gp_{2n}$, and $2n$ odd fields of weight $m+1$, which transform as the standard representation of $\gs\gp_{2n}$. By \cite[Cor. 3.5]{CL3}, we may assume without loss of generality that the fields in weights $2,4,\dots, 2m$ commute with $V^{-\psi/2 - n - 1/2}( \gs\gp_{2n})$, and the weight $m+1$ fields are primary with respect to $V^{-\psi/2 - n - 1/2}( \gs\gp_{2n})$. By \cite[Cor. 3.4]{CL3}, the free field limit of $\cW^{\psi}_{1C}(n,m)$ is
$$\cO_{\text{ev}}(2n^2+n,2) \otimes \big(\bigotimes_{i=1}^m  \cO_{\text{ev}}(1,4i)\big) \otimes \cS_{\text{odd}}(n, 2m+2).$$

\begin{lemma}  \label{gff1C} For $n+m \geq 1$, $\cC^{\psi}_{1C}(n,m)$ is of type $$\cW(2,4,\dots, 2 (1 + m) (1 + n) - 2)$$ as a one-parameter vertex algebra. Equivalently, this holds for generic values of $\psi$.
\end{lemma}

\begin{proof} First, it follows from \cite[Lem. 4.2]{CL3} that $\cC^{\psi}_{1C}(n,m)$ has limit $$\big(\bigotimes_{i=1}^m  \cO_{\text{ev}}(1,4i)\big) \otimes \big(\cS_{\text{odd}}(n, 2m+2)\big)^{\text{Sp}_{2n}}.$$
By \cite[Thm. 4.3]{CL3}, $ \big(\cS_{\text{odd}}(n, 2m+2)\big)^{\text{Sp}_{2n}}$ is of type $$\cW(2m+2, 2m+4,\dots, 2 (1 + m) (1 + n) - 2).$$ Since $\bigotimes_{i=1}^m  \cO_{\text{ev}}(1,4i)$ is of type $\cW(2,4,\dots, 2m)$, it follows from \cite[Lem. 4.2]{CL3} that $\cC^{\psi}_{1C}(n,m)$ is of type $\cW(2,4,\dots, 2 (1 + m) (1 + n) - 2)$.
\end{proof}

\subsection{Case 2B} Recall that for $m\geq 1$, $\cW^{\psi}_{1B}(n,m)$ has affine subalgebra $V^{-2 \psi - 2 n + 2}(\gs\go_{2n+1})$, even fields in weights $2,4,\dots, 2m$ which commute with $V^{-2 \psi - 2 n + 2}(\gs\go_{2n+1})$, and $2n+1$ odd fields of weight $\frac{2m+1}{2}$, which are primary with respect to $V^{-2 \psi - 2 n + 2}(\gs\go_{2n+1})$ and transform as the standard representation of $\gs\go_{2n+1}$. The free field limit of $\cW^{\psi}_{1B}(n,m)$ is therefore
$$\cO_{\text{ev}}(2n^2+n,2) \otimes \big(\bigotimes_{i=1}^m  \cO_{\text{ev}}(1,4i)\big) \otimes \cO_{\text{odd}}(2n+1, 2m+1).$$

\begin{lemma}  \label{gff2B} For $n+m \geq 1$, $\cC^{\psi}_{2B}(n,m)$ is of type $$\cW(2,4,\dots, 4( m + 1) ( n + 1) - 2)$$
as a one-parameter vertex algebra. Equivalently, this holds for generic values of $\psi$.
\end{lemma}

\begin{proof} 
By \cite[Lem. 4.2]{CL3} as above, $\cC^{\psi}_{2B}(n,m)$ has limit $$\big(\bigotimes_{i=1}^m  \cO_{\text{ev}}(1,4i)\big) \otimes \big(\cO_{\text{odd}}(2n+1, 2m+1)\big)^{\text{O}_{2n+1}}.$$
By \cite[Thm. 4.4]{CL3}, $\big(\cO_{\text{odd}}(2n+1, 2m+1)\big)^{\text{O}_{2n+1}}$ is of type $$\cW(2m+2, 2m+4,\dots, 4( m + 1) ( n + 1) - 2),$$
so the claim follows as above. \end{proof}

\subsection{Case 2C} Recall that for $m\geq 1$, $\cW^{\psi}_{2C}(n,m)$ has affine subalgebra 
$V^{\psi-n-3/2}(\gs\gp_{2n})$, even fields in weights $2,4,\dots, 2m$ which commute with $V^{\psi-n-3/2}(\gs\gp_{2n})$, and $2n$ even fields of weight $\frac{2m+1}{2}$, which are primary with respect to $V^{\psi-n-3/2}(\gs\gp_{2n})$ and transform as the standard representation of $\gs\gp_{2n}$. The free field limit of $\cW^{\psi}_{2C}(n,m)$ is therefore
$$\cO_{\text{ev}}(2n^2+n,2) \otimes \big(\bigotimes_{i=1}^m  \cO_{\text{ev}}(1,4i)\big) \otimes \cS_{\text{ev}}(n, 2m+1).$$

\begin{lemma}  \label{gff2C} For $n+m \geq 1$, $\cC^{\psi}_{2C}(n,m)$ is of type $$\cW(2,4,\dots, 2 (1 + n) (1 + m + n) -2)$$
as a one-parameter vertex algebra. Equivalently, this holds for generic values of $\psi$.
\end{lemma}

\begin{proof} 
By \cite[Lem. 4.2]{CL3}, $\cC^{\psi}_{2C}(n,m)$ has limit $$\big(\bigotimes_{i=1}^m  \cO_{\text{ev}}(1,4i)\big) \otimes \big(\cS_{\text{ev}}(n, 2m+1)\big)^{\text{Sp}_{2n}}.$$
By \cite[Thm. 4.2]{CL3}, $\big(\cS_{\text{ev}}(n, 2m+1)\big)^{\text{Sp}_{2n}}$ is of type $$\cW(2m+2, 2m+4,\dots, 2 (1 + n) (1 + m + n) -2),$$
so the claim follows. \end{proof}

\subsection{Case 2D} Recall that for $m\geq 1$, $\cW^{\psi}_{2D}(n,m)$ has affine subalgebra
$V^{- 2 \psi - 2 n +3}(\gs\go_{2n})$, even fields in weights $2,4,\dots, 2m$ which commute with $V^{- 2 \psi - 2 n +3}(\gs\go_{2n})$, and $2n$ odd fields of weight $\frac{2m+1}{2}$, which are primary with respect to $V^{- 2 \psi - 2 n +3}(\gs\go_{2n})$ and transform as the standard representation of $\gs\go_{2n}$. The free field limit of $\cW^{\psi}_{2D}(n,m)$ is therefore
$$\cO_{\text{ev}}(2n^2-n,2) \otimes \big(\bigotimes_{i=1}^m  \cO_{\text{ev}}(1,4i)\big) \otimes \cO_{\text{odd}}(2n, 2m+1).$$

\begin{lemma}  \label{gff2D} For $n+m \geq 1$, $\cC^{\psi}_{2D}(n,m)$ is of type $$\cW(2,4,\dots, 2( m + 1) (2 n + 1) - 2)$$
as a one-parameter vertex algebra. Equivalently, this holds for generic values of $\psi$.
\end{lemma}

\begin{proof} 
By \cite[Lem. 4.2]{CL3}, $\cC^{\psi}_{2D}(n,m)$ has limit $$\big(\bigotimes_{i=1}^m  \cO_{\text{ev}}(1,4i)\big) \otimes \big(\cO_{\text{odd}}(2n, 2m+1)\big)^{\text{O}_{2n}}.$$
By \cite[Thm. 4.4]{CL3}, $\big(\cO_{\text{odd}}(2n, 2m+1)\big)^{\text{O}_{2n}}$ is of type $$\cW(2m+2, 2m+4,\dots, 2( m + 1) (2 n + 1) - 2),$$
so the claim follows. \end{proof}

Next, we consider the cases $\cC^{\psi}_{1B}(n,m)$ and $\cC^{\psi}_{1D}(n,m)$ where we are not able to find a minimal strong generating set at this stage.
\subsection{Case 1B} Recall that $\cW^{\psi}_{1B}(n,m)$ has affine subalgebra $V^{\psi-2n}(\gs\go_{2n+1})$, even fields in weights $2,4,\dots, 2m$ which commute with $V^{\psi-2n}(\gs\go_{2n+1})$, and $2n+1$ even fields in weight $m+1$ which are primary with respect to $V^{\psi-2n}(\gs\go_{2n+1})$ and transform as the standard representation of $\gs\go_{2n+1}$. The free field limit of $\cW^{\psi}_{1B}(n,m)$ is then
$$\cO_{\text{ev}}(2n^2+n,2) \otimes \big(\bigotimes_{i=1}^m  \cO_{\text{ev}}(1,4i)\big) \otimes \cO_{\text{ev}}(2n+1, 2m+2).$$

\begin{lemma} \label{gff1B} 
As a one-parameter vertex algebra, $\cC^{\psi}_{1B}(n,m)$ is of type $\cW(2,4,\dots,  2N)$ for some $N$ satisfying $2N \geq 2 (1 + n) (3 + 2 m + 2 n) - 2$.
\end{lemma}

\begin{proof} By \cite[Lem. 4.2]{CL3}, $\cC^{\psi}_{1B}(n,m)$ has limit 
$$\big(\bigotimes_{i=1}^m  \cO_{\text{ev}}(1,4i)\big) \otimes \big(\cO_{\text{ev}}(2n+1, 2m+2)\big)^{\text{O}_{2n+1}}.$$ By \cite[Thm. 4.5]{CL3}, 
$\big(\cO_{\text{ev}}(2n+1, 2m+2)\big)^{\text{O}_{2n+1}}$ is of type $\cW(2m+2, 2m+4,\dots, 2N)$ for some $2N \geq 2 (1 + n) (3 + 2 m + 2 n) - 2$, so the claim follows. \end{proof}
\begin{remark} The lower bound on $N$ is a consequence of Weyl's second fundamental theorem of invariant theory for $\text{O}_{2n+1}$ \cite{W}. We have an isomorphism of differential graded rings $$\text{gr}\big(\cO_{\text{ev}}(2n+1, 2m+2)\big)^{\text{O}_{2n+1}} \cong \mathbb{C}[\bigoplus_{i\geq 0} V_i]^{\text{O}_{2n+1}},$$ where each $V_i \cong \mathbb{C}^{2n+1}$ as an $\text{O}_{2n+1}$-module. The generators of $\mathbb{C}[\bigoplus_{i\geq 0} V_i]^{\text{O}_{2n+1}}$ are all quadratics, and the ideal of relations is generated by determinants of degree $2n+2$ in these quadratics. The relation of minimal weight has weight $2 (1 + n) (3 + 2 m + 2 n)$, and the statement that $\big(\cO_{\text{ev}}(2n+1, 2m+2)\big)^{\text{O}_{2n+1}}$ is of type $\cW(2m+2, 2m+4,\dots, 2 (1 + n) (3 + 2 m + 2 n) - 2)$ is equivalent to this relation being a decoupling relation for the generator in weight $2 (1 + n) (3 + 2 m + 2 n)$. We will see later (Corollary \ref{cor:newstrongtypes}) that in fact $2N = 2 (1 + n) (3 + 2 m + 2 n) - 2$. \end{remark}

\subsection{Case 1D} Recall that $\cW^{\psi}_{1D}(n,m)$ has affine subalgebra $V^{\psi-2n+1}(\gs\go_{2n})$, even fields in weights $2,4,\dots, 2m$ which commute with $V^{\psi-2n+1}(\gs\go_{2n})$, and $2n$ additional even fields of weight $m+1$ which are primary with respect to $V^{\psi-2n+1}(\gs\go_{2n})$ and transform as the standard representation of $\gs\go_{2n}$. The free field limit of $\cW^{\psi}_{1D}(n,m)$ is therefore
$$\cO_{\text{ev}}(2n^2-n,2) \otimes \big(\bigotimes_{i=1}^m  \cO_{\text{ev}}(1,4i)\big) \otimes \cO_{\text{ev}}(2n, 2m+2).$$

\begin{lemma} \label{gff1D} As a one-parameter vertex algebra, $\cC^{\psi}_{1D}(n,m)$ is of type $\cW(2,4,\dots,  2N)$ for some $N$ satisfying $2N \geq 2 (1 + m + n) (1 + 2 n) - 2$.
\end{lemma}

\begin{proof} First, \cite[Lem. 4.2]{CL3} shows that $\cC^{\psi}_{1D}(n,m)$ has limit $$\big(\bigotimes_{i=1}^m  \cO_{\text{ev}}(1,4i)\big) \otimes \big(\cO_{\text{ev}}(2n, 2m+2)\big)^{\text{O}_{2n}}.$$ Again by \cite[Thm. 4.5]{CL3}, 
$\big(\cO_{\text{ev}}(2n, 2m+2)\big)^{\text{O}_{2n}}$ is of type $\cW(2m+2, 2m+4, \dots, 2N)$ for some $2N \geq  2 (1 + m + n) (1 + 2 n) - 2$, so the claim follows. \end{proof}
As above, the relation of minimal weight among the generators has weight $2 (1 + m + n) (1 + 2 n)$, and again we will see later (Corollary \ref{cor:newstrongtypes}) that $2N = 2 (1 + m + n) (1 + 2 n) - 2$.

Finally, we consider the cases $\cC^{\psi}_{1O}(n,m)$ and $\cC^{\psi}_{2O}(n,m)$. We need two new ingredients: the description of orbifolds of certain free field algebras under $\text{Osp}_{1|2n}$, and the adaptation of the method of studying affine cosets by passing to their orbifold limits developed in \cite{CL1}, to cosets of $V^k(\go\gs\gp_{1|2n})$. We begin by restating the versions of Sergeev's first and second fundamental theorems of invariant theory for $\text{Osp}_{1|2n}$ that we need; these are specializations of \cite[Thm. 1.3]{S1} and \cite[Thm. 4.5]{S2}. First, we consider the invariants in the ring of functions on a sum of copies of the standard module with odd parity.

\begin{thm} \label{sergeevodd} For $k\geq 0$, let $U_k$ be a copy of the standard $\text{Osp}_{1|2n}$-module $\mathbb{C}^{1|2n}$, which has odd subspace spanned by $\{x_{k,i}, y_{k,i}|\ i=1,\dots, n\}$, and even subspace spanned by $z_k$. Then the ring of invariant polynomial functions 
$$R =  \mathbb{C}[ \bigoplus_{k\geq 0}  U_k]^{\text{Osp}_{1|2n}}$$ is generated by the quadratics
$$q_{a,b} = \frac{1}{2} \sum_{i=1}^n (x_{i,a} y_{i,b} + x_{i,b} y_{i,a})  + \frac{1}{2} z_{a} z_{b}\qquad a,b\geq 0.$$ Let $Q_{a,b}$ be commuting indeterminates satisfying $Q_{a,b} = Q_{b,a}$. The kernel of the map $$\mathbb{C}[Q_{a,b}] \rightarrow R,\qquad Q_{a,b} \mapsto q_{a,b}$$ is generated by polynomials $p_I$ of degree $2n+2$ in the variables $Q_{a,b}$ corresponding to a rectangular Young tableau of size $2\times (2n+2)$, filled by entries from a standard sequence $I$ of length $4n+4$ from the set of indices $\{0,1,2,\dots\}$. The entries must weakly increase along rows and strictly increase along columns. \end{thm}

For the invariants in the ring of functions on a sum of copies of the standard module with even parity, a few modifications are needed.

\begin{thm} \label{sergeeveven} For $k\geq 0$, let $U_k$ be a copy of the standard $\text{Osp}_{1|2n}$-module $\mathbb{C}^{2n|1}$, with even subspace spanned by $\{x_{k,i}, y_{k,i}|\ i=1,\dots, n\}$ and odd subspace spanned by $z_k$. Then the ring of invariant polynomial functions 
$$R =  \mathbb{C}[ \bigoplus_{k\geq 0}  U_k]^{\text{Osp}_{1|2n}}$$ is generated by the quadratics
$$q_{a,b} = \frac{1}{2} \sum_{i=1}^n (x_{i,a} y_{i,b}- x_{i,b} y_{i,a})  -\frac{1}{2} z_{a} z_{b}\qquad a,b\geq 0.$$ Let $Q_{a,b}$ be commuting indeterminates satisfying $Q_{a,b} = -Q_{b,a}$. The kernel of the map $$\mathbb{C}[Q_{a,b}] \ra R,\qquad Q_{a,b} \mapsto q_{a,b}$$ is generated by polynomials $p_I$ of degree $2n+2$ in the variables $Q_{a,b}$ corresponding to a rectangular Young tableau of size $2\times (2n+2)$, filled by entries from a standard sequence $I$ of length $4n+4$ from the set of indices $\{0,1,2,\dots\}$. The entries must strictly increase along rows and weakly increase along columns. \end{thm}

In both cases, the precise form of the relations can be found in \cite{S2}, but is not needed for our purposes. We only need the conformal weight of the relations which can be read off from the entries in the corresponding Young tableau.

Next, we recall that $V^k(\go\gs\gp_{1|2n})$ comes from a deformable family in the sense of \cite{CL1} as follows. Let $\kappa$ be a formal variable satisfying $\kappa^2 = k$, and let $F$ be the ring of complex-valued rational functions of $\kappa$ of degree at most zero, with possible poles only at $\kappa = 0$. In other words, $F$ consists of functions of the form $\frac{p(\kappa)}{\kappa^d}$, where $d \geq 0$ and $p$ is a polynomial of degree at most $d$. There is a vertex algebra $\cV$ over $F$ such that $\cV/ (\kappa - \sqrt{k}) \cV \cong V^k(\go\gs\gp_{1|2n})$ for all $k \neq 0$, Here $(\kappa - \sqrt{k}) \cV$ denotes the ideal generated by $\kappa - \sqrt{k}$. In the notation of \cite{CL1}, 
$$\cV^{\infty} = \lim_{\kappa \rightarrow \infty} \cV \cong \cH(2n^2+n) \otimes \cA(n),$$ where $\cH(2n^2+n)$ denotes the Heisenberg algebra of rank $2n^2 + n = \text{dim} \ \gs\gp_{2n}$ and $\cA(n)$ denotes the rank $n$ symplectic fermion algebra.

We now consider vertex algebras $\cW^k$ which admit a homomorphism $V^k(\go\gs\gp_{1|2n})\rightarrow \cW^k$ with the following properties:

\begin{enumerate}
\item There exists a deformable family $\cW$ defined over the ring $F_K$ of rational functions of degree at most zero in $\kappa$, with poles in some at most countable set $K$, such that $$\cW/ (\kappa - \sqrt{k})\cW \cong \cW^k,\  \text{for all}\ \sqrt{k} \notin K.$$

\item The map $V^k(\go\gs\gp_{1|2n}) \rightarrow \cW^k$ is induced by a map of deformable families $\cV \rightarrow \cW$.
 
\item $\cW^{\infty} = \lim_{\kappa \rightarrow \infty} \cW$ decomposes as 
$$\cW^{\infty} \cong \cV^{\infty} \otimes \tilde{\cW} \cong \cH(2n^2 + n) \otimes \cA(n) \otimes \tilde{\cW},$$ for some vertex subalgebra $\tilde{\cW}\subseteq \cW^{\infty}$.

\item The action of $\go\gs\gp_{1|2n}$ on $\cW$ infinitesimally generates an action of the Lie supergroup $\text{SOsp}_{1|2n}$, and $\cW$ decomposes into finite-dimensional $\text{SOsp}_{1|2n}$-modules.

\end{enumerate}

Under these circumstances, we obtain

\begin{thm} \label{thm:ospcosetlimit}

\begin{enumerate}
\item $\cC = \text{Com}(\cV, \cW)$ is a deformable family, and $$\cC/ (\kappa - \sqrt{k})\cC \cong \cC^k = \text{Com}(V^k(\mathfrak{osp}_{1|2n}), \cW^k),$$
for generic $k$.
\item $\text{SOsp}_{1|2n}$ acts on $\tilde{\cW}$, and we have an isomorphism 
\begin{equation} \begin{split} \cC^{\infty} & \cong \text{Com}(\cV^{\infty}, \cV^{\infty} \otimes \tilde{\cW})^{\text{SOsp}_{1|2n}}
\\ & \cong \text{Com}(\cH(2n^2+n) \otimes \cA(n), \cH(2n^2+n) \otimes \cA(n) \otimes \tilde{\cW})^{\text{SOsp}_{1|2n}}
\\ & \cong \tilde{\cW}^{\text{SOsp}_{1|2n}}.
\end{split} \end{equation}

\item For generic $k$, $\cW^k$ admits a decomposition
\begin{equation} \cW^k \cong \bigoplus_{\lambda \in P^+} V^{k}(\lambda) \otimes \cC^{k}(\lambda),\end{equation} where $P^+$ denotes the set of dominant weights of $\go\gs\gp_{1|2n}$, $V^{k}(\lambda)$ are the corresponding Weyl modules, and the multiplicity spaces $\cC^{k}(\lambda)$ are irreducible $\cC^k$-modules.

\end{enumerate}
\end{thm}

The proof of the first two statements is the same as the proof of \cite[Thm. 6.10]{CL1}, and only uses the fact that finite-dimensional $\text{SOsp}_{1|2n}$-modules are completely reducible. Similarly, the proof of the third statement is the same as the proof of \cite[Thm. 4.12]{CL3}. It is apparent than in our main examples, namely $\cW^k = \cW^{\psi}_{1O}(n,m)$ and $\cW^{\psi}_{2O}(n,m)$ these hypotheses are satisfied. Moreover, these algebras are in fact modules over the double cover $\text{Osp}_{1|2n}$ of $\text{SOsp}_{1|2n}$, and it is the $\mathbb{Z}_2$-orbifold of the coset that we actually need to study.

\subsection{Case 1O} Recall that $\cW^{\psi}_{1O}(n,m)$ has affine subalgebra
$V^{- \psi/2 -n}(\go\gs\gp_{1|2n})$, even fields in weights $2,4,\dots, 2m$ which commute with $V^{- \psi/2 -n}(\go\gs\gp_{1|2n})$, and $2n$ odd fields and one even field of weight $m+1$, which are primary with respect to $V^{- \psi/2 -n}(\go\gs\gp_{1|2n})$ and transform as the standard representation of $\go\gs\gp_{1|2n}$. The free field limit of $\cW^{\psi}_{1O}(n,m)$ is therefore
$$\cO_{\text{ev}}(2n^2+n,2) \otimes \cS_{\text{odd}}(n,2) \otimes \big(\bigotimes_{i=1}^m  \cO_{\text{ev}}(1,4i)\big) \otimes \cS_{\text{odd}}(n, 2m+2) \otimes \cO_{\text{ev}}(1, 2m+2).$$

\begin{lemma} \label{gff1O} As a one-parameter vertex algebra, $\cC^{\psi}_{1O}(n,m)$ has a strong generating set of type $\cW(2,4,\dots)$, which need not be minimal. If it admits a finite strong generating set of type $\cW(2,4,\dots ,2N)$ for some $N$, we must have $2N \geq 2 (3 + 2 m) (1 + n) - 2$.
\end{lemma}

\begin{proof} First, it follows from Theorem \ref{thm:ospcosetlimit} that $\cC^{\psi}_{1O}(n,m)$ has limit $$\big(\bigotimes_{i=1}^m  \cO_{\text{ev}}(1,4i)\big) \otimes \big(\cS_{\text{odd}}(n, 2m+2) \otimes \cO_{\text{ev}}(1, 2m+2)\big)^{\text{Osp}_{1|2n}}.$$
We assign $\cS_{\text{odd}}(n, 2m+2) \otimes \cO_{\text{ev}}(1, 2m+2)$ the good increasing filtration where the weight $m+1$ generators 
$\{a^i, b^i |\ i = 1,\dots,n\}$ of $\cS_{\text{odd}}(n, 2m+2)$, and the weight $m+1$ generator $a$ of $\cO_{\text{ev}}(1, 2m+2)$ all have degree $1$. Then $$\text{gr}\bigg(\big(\cS_{\text{odd}}(n, 2m+2) \otimes \cO_{\text{ev}}(1, 2m+2)\big)^{\text{Osp}_{1|2n}}\bigg) \cong \text{gr}\bigg(\cS_{\text{odd}}(n, 2m+2) \otimes \cO_{\text{ev}}(1, 2m+2)\bigg)^{\text{Osp}_{1|2n}} \cong R,$$ where $R$ is the ring of invariants in Theorem \ref{sergeevodd}. Then $\big(\cS_{\text{odd}}(n, 2m+2) \otimes \cO_{\text{ev}}(1, 2m+2)\big)^{\text{Osp}_{1|2n}}$ is strongly generated by the corresponding fields $$\omega_{a,b} = \frac{1}{2} \sum_{i=1}^n \big( :(\partial^a a^i)(  \partial^b b^i ):+ :(\partial^b a^i)( \partial ^a b^i): \big)  + \frac{1}{2} :(\partial^a a)(\partial^b a):, \qquad a,b\geq 0,$$ which have weight $2m+2+a+b$. As usual, there are linear relations among these fields and their derivatives, and the subsets
$$\{\partial^k \omega_{2a,0}|\ a \geq 0\},\qquad \{\omega_{a,b}|\ a \geq b \geq 0\}$$ span the same vector space. Therefore the fields 
$\{\omega_{2a,0}|\ a \geq 0\}$, which have weight $2m +2+ 2a$, are a strong generating set. This shows that $\big(\cS_{\text{odd}}(n, 2m+2) \otimes \cO_{\text{ev}}(1, 2m+2)\big)^{\text{Osp}_{1|2n}}$ has a strong generating set of type $\cW(2m+2, 2m+4,\dots)$, which proves this first statement since $\cO_{\text{ev}}(1,4) \otimes \cO_{\text{ev}}(1,8) \otimes \cdots \otimes \cO_{\text{ev}}(1,4m)$ is of type $\cW(2,4,\dots, 2m)$.

Next, the relation of minimal weight given by Theorem \ref{sergeevodd} corresponds to the $2\times (2n+2)$, Young tableau with bottom row consisting of $0$'s and top row consisting of $1$'s. This relation therefore has weight $2 (3 + 2 m) (1 + n)$. If there exists a decoupling relation
\begin{equation} \label{2O:decoup} \omega_{2a,0} = P(\omega_{0,0},\omega_{2,0},\dots, \omega_{2a-2,0}),\end{equation} where $P$ is a normally ordered polynomial in $\omega_{0,0},\omega_{2,0},\dots, \omega_{2a-2,0}$ and their derivatives, the weight $2a+2m+2$ of this relation must therefore be at least  $2 (3 + 2 m) (1 + n)$. Starting with this relation, we claim that there exist similar decoupling relations expressing $\omega_{2b,0}$ for all $b>a$ as normally ordered polynomials in $\omega_{0,0},\omega_{2,0},\dots, \omega_{2a-2,0}$ and their derivatives. Therefore if \eqref{2O:decoup} is such a relation of minimal weight, then $\cC^{\psi}_{1O}(n,m)$ would be of type $\cW(2,4,\dots, 2N)$ for $2N = 2a+2m$. 

To construct these decoupling relations, we regard $\cS_{\text{odd}}(n, 2m+2) \otimes \cO_{\text{ev}}(1, 2m+2)$ as a subalgebra of $\cA(n) \otimes \cH(1)$, where $\cA(n) = \cS_{\text{odd}}(n, 2)$ is the rank $n$ symplectic fermion algebra and $\cH(1) = \cO_{\text{ev}}(1, 2)$ is the rank one Heisenberg algebra. As in \cite{CL3}, let $e^i, f^i$ denote the generators of $\cA(n)$, which satisfy \begin{equation} \begin{split}
e^{i} (z) f^{j}(w)&\sim \delta_{i,j} (z-w)^{-2},\quad f^{j}(z) e^{i}(w)\sim - \delta_{i,j} (z-w)^{-2},\\
e^{i} (z) e^{j} (w)&\sim 0,\qquad\qquad\qquad f^{i} (z) f^{j} (w)\sim 0,
\end{split} \end{equation} 
and let $\alpha$ be the generator of $\cH(1)$ satisfying $\alpha(z) \alpha(w) \sim (z-w)^{-2}$. Then $\cS_{\text{odd}}(n, 2m+2) \otimes \cO_{\text{ev}}(1, 2m+2)$ is realized inside $\cA(n) \otimes \cH(1)$ via
$$a^i =  \frac{\epsilon}{\sqrt{(k-1)!}} \partial^{k/2-1} e^i, \qquad b^i =  \frac{\epsilon}{\sqrt{(k-1)!}} \partial^{k/2-1} f^i,\qquad a = \frac{\epsilon}{\sqrt{(k-1)!}} \partial^{k/2-1}\alpha.$$ 
 Next, let $$r = \frac{1}{2} \sum_{i=1}^n \big(:(\partial e^i) f^i :+ : e^i \partial f^i: \big) + \frac{1}{2} :(\partial \alpha) \alpha: \ \in \big(\cA(n) \otimes \cH(1)\big)^{\text{Osp}_{1|2n}}.$$ Note that $r$ does not lie in the subalgebra $\big(\cS_{\text{odd}}(n, 2m+2) \otimes \cO_{\text{ev}}(1, 2m+2)\big)^{\text{Osp}_{1|2n}}$; however, the mode $r_{(1)}$ preserves this subalgebra. A calculation shows that for all $a \geq 0$,
\begin{equation} \label{outsideoperator} r_{(1)} \omega_{2a,0} = (-1)^{m+1}(2a+2m+4)\omega _{2a+2,0} + \cdots,\end{equation} where the remaining terms are of the form $\partial^{2i+2} \omega_{2a-2i,0}$ for $0\leq i \leq a$.
Now suppose we have a decoupling relation of the form \eqref{2O:decoup}. Applying $r_{(1)}$, we obtain a relation 
$$(-1)^{m+1}(2a+2m+4) \omega _{2a+2,0} = r_{(1)} P(\omega_{0,0},\omega_{2,0},\dots, \omega_{2a-2,0}).$$ It follows from \eqref{outsideoperator} together with the fact that $(\omega_{2i,0})_{(0)} \omega_{2j,0}$ is a total derivative for all $i,j\geq 0$, that the right hand side is a normally ordered polynomial in $\omega_{0,0}, \omega_{2,0},\dots, \omega_{2a,0}$ and their derivatives. All appearances of $\omega_{2a,0}$ and its derivatives can be eliminated by substituting \eqref{2O:decoup} and its derivatives. Therefore we can express $\omega_{2a+2,0}$ as a normally ordered polynomial in $\omega_{0,0}, \omega_{2,0},\dots, \omega_{2a-2,0}$ and their derivatives.

Inductively, assume that we have constructed such relations \begin{equation} \label{previousdecouplings} \omega_{2a+2i,0} = P_{2a+2i}(\omega_{0,0}, \omega_{2,0},\dots, \omega_{2a-2,0}),\qquad 0 \leq i \leq t.\end{equation} As above, applying $r_{(1)}$ to both sides of $\omega_{2a+2t,0} = P_{2a+2t}(\omega_{0,0}, \omega_{2,0},\dots, \omega_{2a-2,0})$ yields $$ (-1)^{m+1}(2a+2t+2m+4)\omega _{2a+2t+2,0} = r_{(1)} P_{2a+2t+2}(\omega_{0,0},\omega_{2,0},\dots, \omega_{2a-2,0}).$$ Again, the right hand side can be written as a normally ordered polynomial in  $\omega_{0,0}, \omega_{2,0},\dots, \omega_{2a+2t,0}$ and their derivatives. All appearances of $\omega_{2a,0}, \omega_{2a+2,0}, \dots, \omega_{2a+2t,0}$ and their derivatives can be eliminated using the previous decoupling relations \eqref{previousdecouplings} for $0\leq i \leq t$. 
\end{proof}

We will see later (Corollary \ref{cor:newstrongtypes}) that $\cC^{\psi}_{1O}(n,m)$ is of type $\cW(2,4,\dots, 2 (3 + 2 m) (1 + n) - 2)$, so the relation of weight $2 (3 + 2 m) (1 + n)$ must in fact be a decoupling relation.

\subsection{Case 2O} Recall that $\cW^{\psi}_{2O}(n,m)$ has affine subalgebra $V^{\psi -n-1}(\go\gs\gp_{1|2n})$, even fields in weights $2,4,\dots, 2m$ which commute with $V^{\psi -n-1}(\go\gs\gp_{1|2n})$, and $2n$ even fields and one odd field of weight $\frac{2m+1}{2}$, which are primary with respect to $V^{\psi -n-1}(\go\gs\gp_{1|2n})$ and transform as the standard representation of $\go\gs\gp_{1|2n}$. The free field limit of $\cW^{\psi}_{2O}(n,m)$ is therefore
$$\cO_{\text{ev}}(2n^2+n,2) \otimes \cS_{\text{odd}}(n,2) \otimes \big(\bigotimes_{i=1}^m  \cO_{\text{ev}}(1,4i)\big) \otimes \cS_{\text{ev}}(n, 2m+1) \otimes \cO_{\text{odd}}(1, 2m+1).$$

\begin{lemma} \label{gff2O} As a one-parameter vertex algebra, $\cC^{\psi}_{2O}(n,m)$ has a strong generating set of type $\cW(2,4,\dots)$, which need not be minimal. If it admits a finite strong generating set of type$\cW(2,4,\dots ,2N)$ for some $N$, we must have $2N \geq 4 (1 + n) (1 + m + n) - 2$.
\end{lemma}

\begin{proof} First, it follows from Theorem \ref{thm:ospcosetlimit} that $\cC^{\psi}_{2O}(n,m)$ has limit 
$$\big(\bigotimes_{i=1}^m  \cO_{\text{ev}}(1,4i)\big) \otimes \big( \cS_{\text{ev}}(n, 2m+1) \otimes \cO_{\text{odd}}(1, 2m+1) \big)^{\text{Osp}_{1|2n}}.$$ So to prove the first statement, it suffices to show that $\big( \cS_{\text{ev}}(n, 2m+1) \otimes \cO_{\text{odd}}(1, 2m+1) \big)^{\text{Osp}_{1|2n}}$ is of type $\cW(2m+2, 2m+4, \dots)$. The argument is similar to the proof of Lemma \ref{gff1O}, and is based on Theorem \ref{sergeeveven}. First, in terms of the weight $\frac{2m+1}{2}$ generators $\{a^i, b^i |\ i = 1,\dots,n\}$ of $\cS_{\text{ev}}(n, 2m+1)$, and the  weight $\frac{2m+1}{2}$ generator $\phi$ of $\cO_{\text{odd}}(1, 2m+1)$, we have strong generators  
$$\omega_{a,b} = \frac{1}{2} \sum_{i=1}^n \big( :(\partial^a a^i)(  \partial^b b^i ):- :(\partial^b a^i)( \partial ^a b^i):\big)   -\frac{1}{2} :(\partial^a \phi)(\partial^b \phi):, \qquad a,b\geq 0,$$ which have weight $2m+1+a+b$. Not all of these are necessary, and it is easy to see that the subset $\{\omega_{2a+1,0}|\ a\geq 0\}$ suffices to strongly generate. Since these fields have weights $2m+2, 2m+4,\dots$, this proves the first statement.

Next, the relation of minimal weight among these generators corresponds to the $2\times (2n+2)$ Young tableau with both rows consisting of $0,1,\dots, 2n+1$, so this relation has weight $4 (1 + n) (1 + m + n)$. If there exists a decoupling relation for any of the generating fields, the lowest possible weight where this could occur is therefore $4 (1 + n) (1 + m + n)$. As in the case of Lemma \ref{gff1O}, if there exists a decoupling relation for $\omega_{2a+1,0}$ for some $2a+1 \geq (1 + 2 n) (3 + 2 m + 2 n)$, it is easy to construct similar decoupling relations expressing $\omega_{2b+1,0}$ for all $b>a$ as normally ordered polynomials in $\omega_{1,0},\omega_{3,0},\dots, \omega_{2a-1,0}$ and their derivatives.
\end{proof}
Again, Corollary \ref{cor:newstrongtypes} implies that the relation of weight $4 (1 + n) (1 + m + n)$ is in fact a decoupling relation.

\subsection{On subalgebras of $\cC^{\psi}_{iX}(n,m)$}
Even though $\cC^{\psi}_{iX}(n,m)$ is of type $\cW(2,4,\dots)$, it is not yet obvious that it can be obtained as a quotient of $\cW^{{\rm ev}, I_{iX,n,m}}(c,\lambda)$ because it remains to show that it is generated by the weights $2$ and $4$ fields. This will be shown in the next section, and the following weaker statement will be needed.

\begin{lemma} \label{lem:genofcpsi} For $i=1,2$ and $X = B,C,D,O$, $\cC^{\psi}_{iX}(n,m)$ is generated by the fields in weights $2,4,\dots, 2m+4$.
\end{lemma}

\begin{proof} It suffices to show that the free field limit has this property. In all cases, this limit has the form 
$$\big(\bigotimes_{i=1}^m  \cO_{\text{ev}}(1,4i)\big) \otimes \cA^G,$$ where $\cA$ is a free field algebra and $G$ is either $\text{O}_{2n+1}$, $\text{Sp}_{2n}$, $\text{O}_{2n}$, or $\text{Osp}_{1|2n}$. In all cases, it is straightforward to check that the fields in weights $2m+2$ and $2m+4$ are sufficient to generate all the fields in higher weights $2m+6, 2m+8,\dots$ which strongly generate $\cA^G$. The proof is similar to the proof of \cite[Lem. 4.2]{L1}, and is omitted.
\end{proof}

For $i=1,2$ and $X = B,C,D,O$, let $$\tilde{\cC}^{\psi}_{iX}(n,m) \subseteq \cC^{\psi}_{iX}(n,m)$$ be the subalgebra generated by the weights $2$ and $4$ fields. 
 Let $\{\omega^{2r}| \ 1 \leq r \leq N \}$ be the strong generators of $\cC^{\psi}_{iX}(n,m)$ corresponding to the large level limits which are given by Lemmas \ref{gff1C}-\ref{gff2O}. (Here we are excluding the degenerate cases where $N=1$). Without loss of generality, we may assume that $\omega^2 = L$ and $W^4 = \omega^4$, that is, $\omega^4$ has been chosen to be primary with respect to $L$ and normalized as in \cite{KL}. Set $W^{2r} = W^4_{(1)} W^{2r-2}$, for $r \geq 3$. 

For $3 \leq r \leq N$, we can write 
\begin{equation} \label{lambdaterm} W^{2r} = \lambda_r \omega^{2r} + \cdots,\qquad \lambda_r \in \mathbb{C},\end{equation} where the remaining terms are normally ordered monomials in $\{L, \omega^{2s} |\ 2 \leq s < r\}$. If $\lambda_r \neq 0$ for all $r$, then $\tilde{\cC}^{\psi}(n,m) = \cC^{\psi}(n,m)$. Otherwise, let $M\geq 2$ be the first integer such that $\lambda_{M+1} = 0$.

\begin{lemma} \label{lem:tildecpsi} With $M$ as above, $\{L,W^4,\dots, W^{2M}\}$ close under OPE, so that $\tilde{C}^{\psi}_{iX}(n,m)$ is of type $\cW(2,4,\dots, 2M)$. 
\end{lemma}

\begin{proof} First, since $\{L, \omega^{2s} |\ 2 \leq s \leq N\}$ close under OPE and $\lambda_r \neq 0$ for $3 \leq r \leq M$, we can replace $\omega^{2s}$ with $W^{2s}$ for $3 \leq s \leq M$. It follows that $W^{2i}_{(k)} W^{2j}$ is a normally ordered polynomial in $\{L, W^{2s} |\ 2 \leq s \leq M\}$ and their derivatives whenever $2i + 2j - k -1 \leq 2M+1$. Since $\lambda_{M+1} = 0$, $W^4_{(1)} W^{2M}$ is also a normally ordered polynomial in $\{L, W^{2s} |\ 2 \leq s \leq M\}$. 

Next, we need to show that $W^4_{(0)} W^{2M}$ is a normally ordered polynomial in $\{L, W^{2s} |\ 2 \leq s \leq M\}$ and their derivatives, that is, $\partial \omega^{2M+2}$ does not appear. The argument is a slight modification of the proof of \cite[Lemma 3.3]{KL}, except that we replace $W^{2M+2}$ with $\omega^{2m+2}$.  Write
\begin{equation}\begin{split} W^4_{(1)} W^{2M} = & a_{4, 2M} \omega^{2M+2} + C_{4,2M},
\\  W^4_{(0)} W^{2M} =  & b_{4, 2M} \partial \omega^{2M+2} + D_{4,2M},
\end{split} \end{equation} where $C_{4,2M}, D_{4,2M}$ depend only on $\{L, W^{2s} |\ 2 \leq s \leq M\}$ and their derivatives. Note that \eqref{lambdaterm} implies that $a_{4,2M} = \lambda_{M+1} = 0$.

Recall the Jacobi relation
\begin{equation}\label{closure:jacobi} \begin{split} L_{(2)} (W^4_{(0)} W^{2M}) & = (L_{(2)} W^4)_{(0)} W^{2M} + W^4 _{(0)}( L_{(2)} W^{2M}) \\ & + 2(L_{(1)} W^4)_{(1)} W^{2M} + (L_{(0)} W^4)_{(2)} W^{2M}.\end{split} \end{equation} 
First, $L_{(2)} (W^4_{(0)} W^{2M}) =  b_{4,2M} L_{(2)}  \partial \omega^{2M+2} + L_{(2)} D_{4,2M}$, and since $D_{4,2M}$ only depends on $L,W^4,\dots, W^{2M}$ and their derivatives, $L_{(2)}D_{4,2M}$ does not contribute to the coefficient of $\omega^{2M+2}$. Next, we have 
$$L_{(2)} \partial \omega^{2M+2} = -(\partial L)_{(2)} \omega^{2M+2} + \partial (L_{(2)} \omega^{2M+2}).$$ By weight considerations, $L_{(2)} \omega^{2M+2}$ only depends on $L, W^4, \dots, W^{2M}$, and their derivatives. Modulo terms which depend on $L, W^4, \dots, W^{2M}$ and their derivatives, we have
$$L_{(2)} \partial \omega^{2M+2} \equiv - (\partial L)_{(2)} \omega^{2M+2} = 2 L_{(1)} \omega^{2M+2} = 2(2M+2) \omega^{2M+2}.$$  So the left hand side of \eqref{closure:jacobi} is $2(2M+2)  b_{4,2M} \omega^{2M+2}$, up to terms which do not depend on $\omega^{2M+2}$.

Next, the term $(L_{(2)} W^4)_{(0)} W^{2M}$ from \eqref{closure:jacobi} vanishes because $W^4$ is primary. The term $W^4 _{(0)}( L_{(2)} W^{2M})$ from \eqref{closure:jacobi} does not contribute to the coefficient of $\omega^{2M+2}$, since $L_{(2)} W^{2M}$ only depends on $L, W^4, \dots, W^{2M-2}$ and their derivatives. The term $2(L_{(1)} W^4)_{(1)} W^{2M}$ from \eqref{closure:jacobi} contributes $8 W^4_{(1)}W^{2M} = 8 a_{4,2M} \omega^{2M+2}$. The term $(L_{(0)} W^4)_{(2)} W^{2M}$ from \eqref{closure:jacobi} contributes $\partial W^4_{(2)} W^{2M} =  -2 W^4_{(1)} W^{2M} = -2 a_{4,2M} \omega^{2M+2}$. Equating the coefficients of $\omega^{2M+2}$, we obtain $2(2M+2) b_{4,2M} = 6 a_{4, 2M}$, so that $b_{4,2M} = \frac{3}{2M+2} a_{4,2M} = 0$. This proves the above claim that $W^4_{(0)} W^{2M}$ does not depend on $\partial \omega^{2M+2}$.

Similarly, for $4 < 2i \leq 2j \leq 2M$ and $2i+2j = 2M+4$, write
\begin{equation}\begin{split} W^{2i}_{(1)} W^{2j} = & a_{2i,2j} \omega^{2M+2} + C_{2i,2j},
\\  W^{2i}_{(0)} W^{2j} =  & b_{2i,2j} \partial \omega^{2M+2} + D_{2i,2j},
\end{split} \end{equation} where $C_{2i,2j},  D_{2i,2j}$ depend only on $\{L, W^{2s} |\ 2 \leq s \leq M\}$ and their derivatives. 
A similar modification of the proof of \cite[Lemmas 3.3, 3.4, and 3.5]{KL} show that the constants $a_{2i,2M+4-2i}, b_{2i,2M+4-2i}$ are scalar multiplies of $a_{4,2M} = \lambda_{M+1}$, hence they all vanish. 

Next, since $W^{2M+4} = W^4_{(1)} W^{2M+2}$ and $W^{2M+2}$ is a normally ordered polynomial in $L, W^4,$ $\dots, W^{2M}$ and their derivatives, we have $\lambda_{M+2} = 0$. As above, for $4 < 2i \leq 2j \leq 2M$ and $2i+2j = 2M+6$, write
\begin{equation}\begin{split} W^{2i}_{(1)} W^{2j} = & a_{2i,2j} \omega^{2M+4} + C_{2i,2j},
\\  W^{2i}_{(0)} W^{2j} =  & b_{2i,2j} \partial \omega^{2M+4} + D_{2i,2j},
\end{split} \end{equation} where $C_{2i,2j},  D_{2i,2j}$ depend only on $\{L, W^{2s} |\ 2 \leq s \leq M\}$ and their derivatives.
The same argument shows that for $4 < 2i$, $a_{2i,2M+6-2i}, b_{2i,2M+6-2i}$ are all scalar multiplies of $a_{6,2M} = \lambda_{M+2} = 0$, hence they all vanish. In particular, for 
 $4 < 2i \leq 2j \leq 2M$ and $2i+2j = 2M+6$, all terms in the OPE of $W^{2i}(z) W^{2j}(w)$ depend only on $L,W^4,\dots, W^{2M}$.

By induction on $r$, the same procedure shows that for $4 < 2i \leq 2j \leq 2M$, $2i+2j \leq 2M+2r$, and $2r \leq 2M$, all terms in the OPE of $W^{2i}(z) W^{2j}(w)$ depend only on $L,W^4,\dots, W^{2M}$.
\end{proof}

\begin{thm} \label{thm:tildecpsi} For $i=1,2$ and $X = B,C,D,O$, $\tilde{\cC}^{\psi}_{iX}(n,m)$ is a one-parameter quotient of $\cW^{\rm ev}(c,\lambda)$ for some ideal $I_{iX,n,m}$.
\end{thm} 

\begin{proof}
We will prove this only for $\tilde{\cC}^{\psi}_{1D}(n,m)$ since the proof in the other cases is similar. First, for $n\geq 1$ and $m=0$, $$\cC^{\psi}_{1D}(n,0) \cong \text{Com}(V^{\psi-2n+1}(\gs\go_{2n}), V^{\psi-2n+1}(\gs\go_{2n+1}))^{\mathbb{Z}_2},$$ which is generated by the weights $2$ and $4$ fields and arises as a quotient of $\cW^{\rm ev}(c,\lambda)$ \cite[Thm. 3.3]{CKoL}. In particular, it coincides with $\tilde{\cC}^{\psi}_{1D}(n,0)$. Similarly, for $n = 0$ and $m\geq 1$, $\cC^{\psi}_{1D}(0,m) \cong \cW^k(\gs\go_{2m+1})$ so the same holds by \cite[Cor. 5.2]{KL}.

We assume next that $n\geq 1$ and $m\geq 1$, and let $\{L, W^4,\dots, W^{2M}\}$ be the strong generating set for $\tilde{\cC}^{\psi}_{1D}(n,m)$ given by Lemma \ref{lem:tildecpsi}. We need a slightly different argument in the cases $M \geq 7$ and $M < 7$.

Suppose first that $M \geq 7$. By \cite[Thm. 3.10]{KL}, to show that $\tilde{C}^{\psi}_{1D}(n,m)$ is a quotient of $\cW^{\text{ev}}(c,\lambda)$ it suffices to prove that $\{L, W^{2r}|\ 2 \leq r \leq 7 \}$ satisfy the OPE relations of \cite{KL}; equivalently, all Jacobi identities of type $(W^{2a}, W^{2b}, W^{2c})$ for $a+b+c \leq 8$ hold as a consequence of \cite[Eq. (2.6)-(2.9)]{KL}. In this notation, $W^2 = L$, as in \cite{KL}.

By \cite[Thm. 2.1]{CKoL}, this condition is automatic if the graded character of $\tilde{\cC}_{1D}^{\psi}(n,m)$ coincides with that of $\cW^{\rm ev}(c,\lambda)$ up to weight $13$. By Lemma  \ref{gff1D}, the first relation among the generators $\{L, \omega^{2r}|\ r \geq 2\}$ of $\cC^{\psi}_{1D}(n,m)$ and their derivatives occurs in weight $2(m+n+1)(2n+1)$, and since $n,m \geq 1$, there are no normally ordered relations in $\cC^{\psi}_{1D}(n,m)$ among these fields in weight below $18$. Therefore the character of $\cC^{\psi}_{1D}(n,m)$ coincides with that of $\cW^{\rm ev}(c,\lambda)$ in weight up to $14$. Since $M \geq 7$, $\tilde{\cC}^{\psi}_{1D}(n,m)$ and $\cC^{\psi}_{1D}(n,m)$ have the same graded character up to weight $14$, so the conclusion holds.

Finally, suppose that $M < 7$. Since $\lambda_{M+1} = 0$ and $\lambda_r \neq 0$ for $2\leq r \leq M$, there can be no nontrivial normally ordered relations among the generators $\{L, W^4,\dots, W^{2M}\}$ of $\tilde{\cC}_{1D}^{\psi}(n,m)$ in weight up to $2M$, since this property holds for the corresponding fields $\{L, \omega^4,\dots, \omega^{2M}\}$. Equivalently, all Jacobi relations among $\{L, W^4,\dots, W^{2M}\}$ of type 
$$(W^{2a}, W^{2b}, W^{2c}),\qquad 2a + 2b + 2c \leq 2M+2,$$ must hold as a consequence of \cite[Eq. (2.6)-(2.9)]{KL} alone. Therefore the OPEs $W^{2i}(z) W^{2j}(w)$ for $2i+2j \leq 2M$ are the same as those of $\cW^{I, \text{ev}}(c,\lambda)$ for some ideal $I \subseteq \mathbb{C}[c,\lambda]$.

If we use the same procedure as the construction $\cW^{\text{ev}}(c,\lambda)$ given by \cite[Thm. 3.9]{KL}, beginning with the fields $L,W^4,\dots, W^{2M}$ and the OPEs $W^{2i}(z) W^{2j}(w)$ for $2i+2j \leq 2M$, we can formally define new fields $W^{2M+2r} = (W^4_{(1)})^r W^{2M}$ for all $r \geq 1$, and then define the OPE algebra of all fields $\{L, W^{4},\dots, W^{2M}, W^{2M+2r}|\ r\geq 1\}$ recursively so that they are the same as the OPEs in $\cW^{I, \text{ev}}(c,\lambda)$. In particular, this realizes $\tilde{\cC}_{1D}^{\psi}(n,m)$ as a one-parameter quotient of $\cW^{I, \text{ev}}(c,\lambda)$ by some vertex algebra ideal $\cI$ containing a field in weight $2M+2$ of the form $W^{2M+2} - P(L,W^4,\dots, W^{2M})$, where $P$ is a normally ordered polynomial in $L, W^2,\dots, W^{2M}$ and their derivatives. \end{proof}

\begin{cor} \label{cor:extension} For $n+m\geq 1$, $\cW^{\psi}_{iX}(n,m)$ is an extension of $V^t(\ga) \otimes \cW$, where $\cW$ is a quotient of $\cW^{{\rm ev}, I_{iX,n,m}}(c,\lambda)$, for some ideal $I_{iX,n,m} \subseteq \mathbb{C}[c,\lambda]$.
\end{cor}

\section{Proof of main result} \label{sec:proofmain}

\subsection{Step 1: Computation of truncation curves} In this subsection, we shall compute the ideals $I_{iX, n,m} \subseteq \mathbb{C}[c,\lambda]$ such that $\tilde{\cC}^{\psi}_{iX}(n,m)$ is realized as a quotient of $\cW^{{\rm ev}, I_{iX,n,m}}(c,\lambda)$. More precisely, we will parametrize the corresponding variety $V(I_{iX,n,m}) \subseteq \mathbb{C}^2$ by giving a rational map
$$\Phi_{iX,n,m}: \mathbb{C} \setminus P \rightarrow V(I_{iX,n,m}), \qquad \Phi_{iX,n,m}(\psi) = \big(c(\psi), \lambda(\psi)\big).$$ Here $P$ is the finite set consisting of poles $c(\psi)$ and $\lambda(\psi)$. Note first that in the cases $n = 0$, $m\geq 1$, and $X = C,D$, there is nothing to prove because
\begin{equation} \begin{split} &  \tilde{\cC}^{\psi}_{1C}(0,m) = \cC^{\psi}_{1C}(0,m) = \cW^{\psi-2m+1}(\gs\go_{2m+1}) = \tilde{\cC}^{\psi}_{1D}(0,m) = \cC^{\psi}_{1D}(0,m),
\\ & \tilde{\cC}^{\psi}_{2C}(0,m) = \cC^{\psi}_{2C}(0,m) = \cW^{\psi-m-1}(\gs\gp_{2m}) = \tilde{\cC}^{\psi}_{2D}(0,m) = \cC^{\psi}_{2D}(0,m).\end{split} \end{equation} 
The truncation curve for $\cW^{\psi-2m+1}(\gs\go_{2m+1})$ already appears in \cite{KL}, and coincides with the truncation curves for both
$\tilde{\cC}^{\psi}_{1C}(n,m)$ and $\tilde{\cC}^{\psi}_{1D}(n,m)$ specialized to $n = 0$. Similarly, the truncation curve for $\cW^{\psi-m-1}(\gs\gp_{2m})$ from \cite{KL} coincides with the truncation curves for both $\tilde{\cC}^{\psi}_{2C}(n,m)$ and $\tilde{\cC}^{\psi}_{2D}(n,m)$ when $n = 0$.

The following cases must also be treated separately, and will be discussed briefly at the end of this section.
\begin{enumerate}
\item $\tilde{\cC}^{\psi}_{1D}(1,m)$ and $\tilde{\cC}^{\psi}_{2D}(1,m)$, where $\ga = \gs\go_2$,
\item $\tilde{\cC}^{\psi}_{1B}(0,m)$, $\tilde{\cC}^{\psi}_{1O}(0,m)$, $\tilde{\cC}^{\psi}_{2B}(0,m)$, and $\tilde{\cC}^{\psi}_{2O}(0,m)$ where $\ga = 0$. In these cases, $\cW^{\psi}_{iX}(0,m)$ is a simple current extension of $\cC^{\psi}_{iX}(0,m)$ of order two.
\end{enumerate}
In all other cases, $\ga$ is simple and our approach will be uniform, and from now on we assume this to be the case. We postulate that $\cW$ is a one-parameter quotient of $\cW^{\rm ev}(c,\lambda)$ and that $V^t(\ga) \otimes \cW$ admits an extension which has $d_{\ga}$ additional strong generating fields of weight $\frac{d_{\gb} + 1}{2}$ and appropriate parity, which transform in the standard representation $\rho_{\ga}$ of $\ga$. We will show that these data uniquely determine the truncation curve for $\cW$, or equivalently, the formula $\lambda(\psi)$.

Let $p$ be a vector in this copy of $\rho_{\ga}$ which is primary with respect to the action of $V^t(\ga)$. Without loss of generality, we may take $p$ to be a highest-weight vector in this representation of $\ga$. This forces the following OPEs:
\begin{equation} \label{totalvirasoroaction} L(z) p(w) \sim \bigg(\mu - \frac{\text{Cas}}{t+h_{\ga}^{\vee}}\bigg) p(w)(z-w)^{-2} + \bigg(\partial p + \cdots\bigg)(w)(z-w)^{-1}.\end{equation}
Here $\mu = \frac{d_{\gb} + 1}{2}$, $\text{Cas}$ is the Casimir eigenvalue of the standard representation of $\ga$, and $h_{\ga}^{\vee}$ is the dual Coxeter number of $\ga$. Additionally, the OPEs of $W^4$, $W^6$, and $W^8$ with $p$ must have the following form: 
\begin{equation} \label{formofOPEW468p} \begin{split} W^4(z) p(w) & \sim k_0 p(w)(z-w)^{-4} +  \bigg(k_1 \partial p + \cdots\bigg)(w)(z-w)^{-3} 
\\ & + \bigg(k_2 \partial^2 p + k_3 :Lp: + \cdots \bigg)(w)(z-w)^{-2}
\\ & + \bigg(k_4 \partial^3 p + k_5 :(\partial L) p: + k_6 :L \partial p: + \cdots \bigg)(w)(z-w)^{-1},
\\ W^6(z) p(w) & \sim k_7 p(w)(z-w)^{-6} +  \bigg(k_8 \partial p + \cdots \bigg) (w)(z-w)^{-7} + \cdots , 
\\ W^8(z) p(w)  & \sim k_9 p(w) (z-w)^{-8} + \cdots. \end{split}\end{equation}
This is because our vertex algebra is strongly generated by $V^t(\ga) \otimes \cW$ together with the fields of weight $\frac{d_{\gb} + 1}{2}$ transforming in the standard representation of $\ga$. Since the fields $W^{2i}$ commute with $V^t(\ga)$, each term appearing in these OPEs has the same Cartan weight as $p$ relative to the Cartan subalgebra of $\ga$. 
Moreover, only those terms which depend on $p$, $L$, and their derivatives are needed in our calculations, so all other terms are omitted in \eqref{formofOPEW468p}.

Next, we impose the following Jacobi identities
\begin{equation} \label{jacobi1}  \begin{split} & L_{(2)} (W^4_{(1)} p) - W^4_{(1)} (L_{(2)} p) -  (L_{(0)} W^4)_{(3)} p - 2 (L_{(1)} W^4)_{(2)} p - (L_{(2)} W^4)_{(1)} p =0,
\end{split} \end{equation}
\begin{equation}  \label{jacobi2} \begin{split} &L_{(3)} (W^4_{(0)} p) - W^4_{(0)} (L_{(3)} p)- (L_{(0)} W^4)_{(3)} p - 3 (L_{(1)} W^4)_{(2)} p - 3 (L_{(2)} W^4)_{(1)} p - (L_{(3)} W^4)_{(0)} p =0,
\end{split} \end{equation}
\begin{equation}  \label{jacobi3} \begin{split} &L_{(4)} (W^4_{(0)} p) - W^4_{(0)} (L_{(4)} p)- (L_{(0)} W^4)_{(3)} p - 4 (L_{(1)} W^4)_{(3)} p - 6 (L_{(2)} W^4)_{(2)} p - 4 (L_{(3)} W^4)_{(1)} p 
\\ &  - (L_{(4)} W^4)_{(0)} p =0,
\end{split} \end{equation}
\begin{equation}  \label{jacobi4} \begin{split} & L_{(3)} (W^4_{(1)} p) - W^4_{(1)} (L_{(3)} p)- (L_{(0)} W^4)_{(4)} p - 3 (L_{(1)} W^4)_{(3)} p - 3 (L_{(2)} W^4)_{(2)} p - (L_{(3)} W^4)_{(1)} p =0,
\end{split} \end{equation}
\begin{equation}  \label{jacobi5} \begin{split} &L_{(2)} (W^4_{(2)} p) - W^4_{(2)} (L_{(2)} p) -  (L_{(0)} W^4)_{(4)} p - 2 (L_{(1)} W^4)_{(3)} p - (L_{(2)} W^4)_{(2)} p =0,
\end{split} \end{equation}
\begin{equation}  \label{jacobi6} \begin{split} &L_{(2)} (W^4_{(0)} p) - W^4_{(0)} (L_{(2)} p) -  (L_{(0)} W^4)_{(2)} p - 2 (L_{(1)} W^4)_{(1)} p - (L_{(2)} W^4)_{(0)} p =0,
\end{split} \end{equation}
\begin{equation} \label{jacobi7}  \begin{split} &W^4_{(0)} (W^4_{(6)} p)  - W^4_{(6)} (W^4_{(0)} p) - (W^4_{(0)} W^4)_{(6)} p =0,
\end{split} \end{equation}
\begin{equation}  \label{jacobi8} \begin{split} &W^4_{(4)} (W^4_{(2)} p) - W^4_{(2)} (W^4_{(4)} p) -  (W^4_{(0)} W^4)_{(6)} p  -  4(W^4_{(1)} W^4)_{(5)} p -  6(W^4_{(2)} W^4)_{(4)} p  - 4 (W^4_{(3)} W^4)_{(3)} p 
\\ &  -  (W^4_{(4)} W^4)_{(2)} p =0,
\end{split} \end{equation}
\begin{equation} \label{jacobi9}  \begin{split} &W^4_{(1)} (W^4_{(5)} p) - W^4_{(5)} (W^4_{(1)} p) -  (W^4_{(0)} W^4)_{(6)} p  -  (W^4_{(1)} W^4)_{(5)} p =0,
\end{split} \end{equation}
 \begin{equation}  \label{jacobi10} W^4_{(0)} (W^6_{(8)} p)  - W^6_{(8)} (W^4_{(0)} p) - (W^4_{(0)} W^6)_{(8)} p =0, \end{equation}
\begin{equation}  \label{jacobi11} W^4_{(1)} (W^6_{(7)} p)  - W^6_{(7)} (W^4_{(1)} p) - (W^4_{(0)} W^6)_{(8)} p- (W^4_{(1)} W^6)_{(7)} p =0. \end{equation}
 Note that \eqref{jacobi1} has weight $\mu + 1$, and a computation shows that the coefficient of $\partial p$ depends only on $k_1, k_2, k_3$ together with the level $t$ of $\ga$, and the parameters $n,m$. Similarly, \eqref{jacobi2} has weight $\mu + 1$, and the coefficient of $\partial p$ depends only on $k_1, k_4, k_5, k_6$ together with $t,n,m$. Next, \eqref{jacobi3}, \eqref{jacobi4}, and \eqref{jacobi5} all have weight $\mu$, and hence are scalar multiples of $p$; these equations depend only on $k_0,\dots, k_6$, together with $t,n,m$. Also, \eqref{jacobi6} has weight $\mu+2$, and the coefficient of $\partial^2 p$ depends only on $k_2, k_4, k_6$ together with $t, n,m$.

Next, \eqref{jacobi7}, \eqref{jacobi8}, and \eqref{jacobi9} all have weight $\mu$, and hence are scalar multiples of $p$; these equations depend only on $k_0,\dots, k_8$, together with $\lambda,t,n,m$. Finally, \eqref{jacobi10} and \eqref{jacobi11}, all have weight $\mu$, and hence are scalar multiples of $p$; these equations depend only on $k_0,\dots, k_9$, together with $\lambda, t,n,m$. 

Using the Mathematica package of Thielemans \cite{T}, we can solve these equations to obtain a unique solution for $k_0,\dots, k_9$ and $\lambda$ as functions of $t, n,m$. We then set $t$ to be the level of the affine subalgebra $V^t(\ga)$, which depends on $\psi$ and $n$. Solving for $\lambda$ in terms of $\psi,n,m$, and using the formulas for $c = c(\psi,n,m)$ appearing in Subsection \ref{subsec:affinecoset}, gives the explicit rational parametrizations $$\Phi_{iX,n,m}: \mathbb{C} \setminus P \rightarrow V(I_{iX,n,m}), \qquad \Phi_{iX,n,m}(\psi) = \big(c(\psi), \lambda(\psi)\big),$$ for $i=1,2$ and $X = B,C,D,O$. The explicit formula for $\Psi_{2B,n,m}(\psi)$ is given in Appendix \ref{appendixA}, but we do not give the others because as we shall see in the next section, all others can be obtained from this one together with various symmetries.

Finally, we comment on how this argument must be modified in the cases where $\ga = \gs\go_2$, or $\ga = 0$ and we take a $\mathbb{Z}_2$-orbifold. First, in the case $\tilde{\cC}^{\psi}_{1D}(1,m)$, the affine subalgebra is a Heisenberg algebra $\cH(1)$, and we normalize the generator $J$ so that
the two fields $p^{\pm}$ transforming as $\rho_{\ga}$ satisfy 
\begin{equation} \label{heis:normalization} J(z) p^{\pm}(w) \sim \pm (z-w)^{-1}.\end{equation} Then $J$ satisfies
\begin{equation} \label{heis:normalization1D} J(z) J(w) \sim (\psi - 1) (z-w)^{-2}.\end{equation} We replace the level $t$ in the above argument by $\psi - 1$, we replace \eqref{totalvirasoroaction} with 
\begin{equation} \label{totalvirasoroaction1D} L(z) p(w) \sim \bigg(\frac{2 \psi + 2 m \psi - 2 m -3}{2 (\psi -1)}\bigg) p(w)(z-w)^{-2} + \bigg(\partial p + \cdots\bigg)(w)(z-w)^{-1},\end{equation} and we solve the same system of equations to determine $\lambda$.

Next, in the case $\tilde{\cC}^{\psi}_{2D}(1,m)$, if we normalize the Heisenberg field $J$ so that \eqref{heis:normalization} holds, we have
\begin{equation} \label{heis:normalization2D} J(z) J(w) \sim (1-2\psi ) (z-w)^{-2}.\end{equation} Again, we replace the level $t$ by $1-2\psi$, 
we replace \eqref{totalvirasoroaction} with
\begin{equation} \label{totalvirasoroaction2D} L(z) p(w) \sim \bigg( \frac{\psi + 2 m \psi -m}{2 \psi -1} \bigg) p(w)(z-w)^{-2} + \bigg(\partial p + \cdots\bigg)(w)(z-w)^{-1},\end{equation} and we apply the same procedure.

Finally, in the remaining cases $\tilde{\cC}^{\psi}_{1B}(0,m)$, $\tilde{\cC}^{\psi}_{1O}(0,m)$, $\tilde{\cC}^{\psi}_{2B}(0,m)$, and $\tilde{\cC}^{\psi}_{2O}(0,m)$, we let $p$ be a generator of the simple current extension of $\cC^{\psi}_{iX}(0,m)$ of weight $\mu$, and we replace 
\eqref{totalvirasoroaction} with 
\begin{equation} \label{totalvirasoroactionothers} L(z) p(w) \sim \mu p(w)(z-w)^{-2} + \bigg(\partial p + \cdots\bigg)(w)(z-w)^{-1}.\end{equation}
The variable $t$ no longer appears, and rest of the argument is the same.

\subsection{Step 2: Symmetries of truncation curves}
\begin{thm}  \label{trialities:parametrization} For $m \geq n \geq 0$ and $m+n \geq 1$, we have the following identities
\begin{equation} \label{eq:trialities:parametrization}  
\begin{split}  & \Phi_{2B, n,m}(\psi) =  \Phi_{2O, n,m-n}\big( \frac{1}{4\psi}\big) = \Phi_{2B,m,n}\big(\frac{\psi}{2 \psi -1}\big),
\\ & \Phi_{1C,n,m}(\psi) = \Phi_{2C, n,m-n}\big( \frac{1}{2\psi}\big) = \Phi_{1C, m,n}\big(\frac{\psi}{\psi-1}\big),
\\ & \Phi_{2D, n,m}(\psi) =  \Phi_{1D, n,m-n}\big( \frac{1}{2\psi}\big)  =  \Phi_{1O, m,n-1}\big(\frac{2 \psi}{2\psi-1}\big), 
\\ &   \Phi_{1O,n,m}(\psi) =  \Phi_{1B, n,m-n}\big(\frac{1}{\psi}\big) =  \Phi_{2D, m+1,n}\big(\frac{\psi}{2 (\psi-1)}\big).
\end{split} \end{equation}
\end{thm}

\begin{proof} The explicit formulas for $\Phi_{iX, n,m}(\psi)$ in all cases can be computed using the approach in the previous subsection. These symmetries follow immediately from our formulas.
\end{proof}

It turns out that all eight functions $\Phi_{iX,n,m}(\psi)$ can be expressed uniformly in terms of one of them. From \eqref{eq:trialities:parametrization}, it it clear that within each of the four triality classes, there is a uniform expression, so what remains is to find an expression that relates the expressions from different triality classes. The explicit formula for $\Phi_{2B,n,m}(\psi)$ appears in Appendix \ref{appendixA} as \eqref{truncationcurve2B}. Here $n,m$ are nonnegative integers, but if we are allowed to replace them with half-integers, we obtain the following.
\begin{thm}
\begin{equation} \label{relationship1} \Phi_{1O,n,m}(\psi) = \Phi_{2B, n,m+\frac{1}{2}}\big(\frac{\psi}{2}\big),\end{equation}
\begin{equation}  \label{relationship2}  \Phi_{2D,n, m}(\psi) = \Phi_{2B, n-\frac{1}{2},m}(\psi),\end{equation}
\begin{equation}  \label{relationship3}  \Phi_{1C, n , m}(\psi) = \Phi_{2B, n+\frac{1}{2} ,m+\frac{1}{2}} \big(\frac{\psi}{2}\big).\end{equation}
\end{thm}

By \eqref{eq:trialities:parametrization}, we can recover $\Phi_{1B, n,m}(\psi)$, $\Phi_{1D, n,m}(\psi)$, and $\Phi_{2C, n,m}(\psi)$, from $\Phi_{1O, n,m+n}(\frac{1}{\psi})$, $\Phi_{2D, n,m+n}(\frac{1}{2\psi})$, and  $\Phi_{1C, n,n+m}(\frac{1}{2\psi})$, respectively. Together with \eqref{relationship1} - \eqref{relationship3}, this shows that all functions $\Phi_{iX, n,m}(\psi)$ for $i=1,2$ and $X = B,C,D,O$, can be recovered from these symmetries together with the explicit formula \eqref{truncationcurve2B} for $\Phi_{2B,n,m}(\psi)$.

\subsection{Step 3: Exhaustiveness} \label{subsec:exhaust}
The last step in the proof of Theorem \ref{main} is to show that $\tilde{\cC}^{\psi}_{iX}(n,m)= \cC^{\psi}_{iX}(n,m)$ for $i=1,2$ and $X = B,C,D,O$. The isomorphisms in Theorem \ref{main} then follow immediately from the symmetries in Theorem \ref{trialities:parametrization}. For a particular value of $\psi \in \mathbb{C}$, let $\tilde{\cC}_{\psi, iX}(n,m)$ and $\cC_{\psi, iX}(n,m)$ denote the simple quotients of these algebras. 

In view of Lemma \ref{lem:genofcpsi} and Theorem \ref{thm:tildecpsi}, it suffices to show that $\tilde{\cC}^{\psi}_{iX}(n,m)$ contains the strong generating fields of $\cC^{\psi}_{iX}(n,m)$ in weights $2,4,\dots, 2m+4$. We give the proof only for $\tilde{\cC}^{\psi}_{2B}(n,m)$ since the argument in the other cases is the same. 
The truncation curve \eqref{truncationcurve2B} for $\tilde{\cC}^{\psi}_{2B}(n,m)$ and the truncation curve for $\cW^s(\gs\gp_{2r})$, which appears in Appendix A of \cite{KL}, intersect at the point $(c,\lambda)$ given in \eqref{app:intersectionpoint}. This intersection gives rise to the following isomorphism:
\begin{equation} \tilde{\cC}_{\psi, 2B}(n,m) \cong \cW_s(\gs\gp_{2r}), \qquad  \psi = \frac{1 + 2 m - 2 n}{2 (1 + 2 m + 2 r)},  \qquad s = -(r+1) + \frac{1 + 2 m + 2 r}{4 (n + r)}.\end{equation} Note that $s$ is a nondegenerate admissible level for $\widehat{\gs\gp}_{2r}$ whenever $1 + 2 m + 2 r$ and $n+r$ are coprime. By Corollary \ref{cor:singularvector}, for $\psi$ and $r$ sufficiently large, the universal algebra $\cW^s(\gs\gp_{2r})$ has a singular vector in weight $4( m + 1) ( n + 1)$, and no singular vector in lower weight. Also, by \cite[Rem. 5.3]{KL}, $\cW_s(\gs\gp_{2r})$ is generated by the weights $2$ and $4$ fields for all non-critical values of $s$, hence this holds for the simple quotient $\cW_s(\gs\gp_{2r})$ as well. It follows that $\tilde{\cC}_{\psi, 2B}(n,m)$ contains all fields in weights $2,4,\dots, 4( m + 1) ( n + 1) -2$, so it must coincide with $\cC_{\psi, 2B}(n,m)$. Since this holds at infinitely many values of $\psi$ and $r$, it holds for the universal objects as well. This shows that $\tilde{\cC}^{\psi}_{2B}(n,m)= \cC^{\psi}_{2B}(n,m)$ as one-parameter vertex algebras. Repeating this argument in the other cases completes the proof of Theorem \ref{main}.

As a consequence of Theorem \ref{main} and the minimal strong generating types for $\cC^{\psi}_{2B}(n,m)$ and $\cC^{\psi}_{2D}(n,m)$ given earlier, we immediately obtain

\begin{cor} \label{cor:newstrongtypes} For $n+m\geq 1$, we have the following minimal strong generating types as one-parameter vertex algebras.
\begin{enumerate}
\item $\cC^{\psi}_{1B}(n,m)$ is of type $\cW(2,4,\dots,  2 (1 + n) (3 + 2 m + 2 n) - 2)$,

\item $\cC^{\psi}_{1D}(n,m)$ is of type $\cW(2,4,\dots,  2 (1 + m + n) (1 + 2 n) - 2)$,

\item $\cC^{\psi}_{1O}(n,m)$ is of type $\cW(2,4,\dots,  2 (3 + 2 m) (1 + n) - 2)$,

\item $\cC^{\psi}_{2O}(n,m)$ is of type $\cW(2,4,\dots,  4 (1 + n) (1 + m + n)  - 2)$.
\end{enumerate}
\end{cor}

A remarkable feature of the truncation curves is that their pairwise intersection points are all rational points. We expect, but do not prove, that these four families of curves account for all nontrivial truncations of $\cW^{\text{ev}}(c,\lambda)$; an equivalent conjecture is also due to Proch\'azka \cite{Pro3}. In Appendices \ref{appendixB}, \ref{appendixC}, and \ref{appendixD}, we will give the explicit classification of coincidences between the simple quotients $\cC_{\psi, iX}(n,m)$ and the algebras $\cW_s(\gs\gp_{2r})$, $\cW_s(\gs\go_{2r})^{\mathbb{Z}_2}$, and $\cW_s(\go\gs\gp_{1|2r})^{\mathbb{Z}_2}$; certain isomorphisms of this kind will be needed for our rationality results in Section \ref{sec:rationality}.

\subsection{Uniqueness and reconstruction} \label{section:uniqueness}

The algebras $\cW^{\psi}_{iX}(n,m)$ satisfy a uniqueness theorem which is analogous to \cite[Thm. 9.1 and Thm. 9.8]{CL3}.
\begin{thm} \label{thm:uniquenesswnm} For all $n,m$ with $n+m \geq 1$, $i=1,2$, and $X = B,C,D,O$, the full OPE algebra of $\cW^{\psi}_{iX}(n,m)$ is determined completely from the structure of $\cC^{\psi}_{iX}(n,m)$, the action of the Lie algebra $\ga$ on the fields which transform as the standard representation $\rho_{\ga}$, and the nondegeneracy condition on these fields given by \cite[Thm. 3.5]{CL3}. In particular,
\begin{enumerate}
\item If $\cA^{\psi}_{iX}(n,m)$ is a one-parameter vertex algebra which extends $V^t(\ga) \otimes \cC^{\psi}_{iX}(n,m)$ by $d_{\ga}$ fields in conformal weight $\frac{d_{\gb} + 1}{2}$ of correct parity, which are primary with respect to $V^t(\ga)$ as well as the total Virasoro field, and satisfy the nondegeneracy condition, then $\cA^{\psi}_{iX}(n,m) \cong \cW^{\psi}_{iX}(n,m)$, as one-parameter vertex algebras.
\item The same result holds if we specialize to a particular value of $\psi$, and replace $\cA^{\psi}_{iX}(n,m)$ and $\cW^{\psi}_{iX}(n,m)$ by their simple quotients $\cA_{\psi, iX}(n,m)$ and $\cW_{\psi, iX}(n,m)$.
\end{enumerate}
\end{thm}

In the cases where $\ga$ is simple, the proof is the same as the proof of \cite[Thm. 9.1]{CL3} in the case $m>1$, and is omitted. In the cases $\cW^{\psi}_{iD}(1,m)$ where $\ga = \gs\go_2$, the affine subalgebra is a Heisenberg algebra $\cH(1)$, and we normalize the generator $J$ such that \eqref{heis:normalization} holds. By the same argument as the proof of \cite[Thm. 9.1]{CL3} in the case $m=1$, all OPEs in $\cW^{\psi}_{iD}(1,m)$ are uniquely determined by the structure of $\cC^{\psi}_{iD}(1,m)$ and \eqref{heis:normalization}, \eqref{heis:normalization1D}, and \eqref{heis:normalization2D}.

Finally, in the cases $\cC^{\psi}_{1B}(0,m)$, $\cC^{\psi}_{1O}(0,m)$, $\cC^{\psi}_{2B}(0,m)$, and $\cC^{\psi}_{2O}(0,m)$, $\ga$ is zero and $\cC^{\psi}_{iX}(0,m)$ is just the $\mathbb{Z}_2$-orbifold of $\cC^{\psi}_{iX}(0,m)$. In these cases, the argument showing the uniqueness of order two simple current extensions of $\cC^{\psi}_{iX}(0,m)$ by one field in the appropriate weight and parity, is even easier and is left to the reader.

\section{Rationality results} \label{sec:rationality} 
By combining Theorem \ref{main} with the theory of extensions of rational vertex superalgebras, we prove many new rationality results in this section. 
 
\subsection{Affine vertex superalgebras of $\go\gs\gp_{1|2n}$}
Among the most fundamental examples of rational vertex algebras are the simple affine vertex algebras $L_k(\gg)$ at positive integer level $k$ \cite{FZ}. For Lie superalgebras, it is known that the only examples of lisse affine vertex superalgebras are $L_k(\go\gs\gp_{1|2n})$ for $k\geq 0$ \cite{GK}, but the rationality is only known for $n=1$ \cite{CFK}. In this case, it is a consequence of the fact that $L_k(\go\gs\gp_{1|2})$ is an extension of $L_k(\gs\gp_2)$ times a rational Virasoro algebra. This perspective generalizes naturally to the case $n>1$, where the Virasoro algebra is replaced by a principal $\cW$-algebra of type $C$.

\begin{thm}\label{thm:osp}
For all positive integers $n, k$, the vertex superalgebra $L_k(\go\gs\gp_{1|2n})$ is lisse and rational, and is an extension of $L_k(\gs\gp_{2n}) \otimes \cW_{\ell}(\gs\gp_{2n})$, for $\ell = -(n + 1) + \frac{1 + k + n}{1 + 2 k + 2 n}$.
\end{thm}
\begin{proof}
\cite[Prop. 8.1 and 8.2]{KW5} tells us that $L_k(\gs\gp_{2n})$ embeds into $L_k(\go\gs\gp_{1|2n})$ if $k$ is a positive integer. 
Since $L_k(\gs\gp_{2n})$ is rational, $L_k(\go\gs\gp_{1|2n})$ is completely reducible as a module for $L_k(\gs\gp_{2n})$. It follows that $\text{Com}(L_k(\gs\gp_{2n}), L_k(\go\gs\gp_{1|2n}))$ is simple by \cite[Prop. 5.4]{CGN}. 

Next, since $k > -(n+1)$, it follows from \cite[Thm. 8.1]{CL1} that $\text{Com}(L_k(\gs\gp_{2n}), L_k(\go\gs\gp_{1|2n}))$ is a homomorphic image of $\cC^{\psi}_{1C}(n,0) = \text{Com}(V^k(\gs\gp_{2n}), V^k(\go\gs\gp_{1|2n}))$, where $k = - \frac{1}{2} (\psi +2n +1)$. Since $\text{Com}(L_k(\gs\gp_{2n}), L_k(\go\gs\gp_{1|2n}))$ is simple, it must be the simple quotient $\cC_{\psi,1C}(n,0)$. Combining this with Corollary \ref{cosetrealizationBCk}, together with Feigin-Frenkel duality, we obtain
$$\text{Com}(L_k(\gs\gp_{2n}), L_k(\go\gs\gp_{1|2n})) \cong \cW_{\ell}(\gs\gp_{2n}),\qquad \ell = -(n + 1) + \frac{1 + k + n}{1 + 2 k + 2 n},$$
which is lisse and rational \cite{Ar2}. We thus have that both $L_k(\go\gs\gp_{1|2n})$ and its even subalgebra $L_k(\go\gs\gp_{1|2n})^{\mathbb{Z}_2}$ are extensions of a lisse vertex algebra. This extension must be of finite index; otherwise, at least one of the finitely many irreducible modules of the lisse vertex algebra must appear with infinite multiplicity. This is impossible since conformal weight spaces of $L_k(\go\gs\gp_{1|2n})$, and its even subalgebra  $L_k(\go\gs\gp_{1|2n})^{\mathbb{Z}_2}$, are finite dimensional. It follows that both these extensions are lisse. 
Rationality of $L_k(\go\gs\gp_{1|2n})^{\mathbb{Z}_2}$ follows from Proposition \ref{rationalextension}, and rationality of $L_k(\go\gs\gp_{1|2n})$ then follows from \cite[Thm. 5.13]{CGN}.
\end{proof}

\subsection{Rationality of $\cW_k(\go\gs\gp_{1|2n})$}
A celebrated result of Arakawa \cite{Ar2} says that for a simple Lie algebra $\gg$, $\cW_{\ell}(\gg)$ is lisse and rational when $\ell$ is a nondegenerate admissible level for $\widehat{\gg}$. When $\gg$ is simply-laced, recall from \cite{ACL} that 
\begin{equation} \label{eqn:aclmain} \text{Com}(L_{k+1}(\gg), L_k(\gg) \otimes L_1(\gg)) \cong \cW_{\ell}(\gg),\text{where}\ \ell = - h^{\vee} + \frac{k+h^{\vee}}{k+h^{\vee} + 1}.\end{equation} In particular, this realizes $\cW_{\ell}(\gg)$ for all nondegenerate admissible levels $\ell$.

We consider the analogous diagonal coset for type $B$. First, if $k$ is an admissible level for $\widehat{\gs\go}_{2n+1}$ we have an embedding $L_{k+1}(\mathfrak{so}_{2n+1}) \hookrightarrow L_{k}(\mathfrak{so}_{2n+1}) \otimes L_1(\gs\go_{2n_1})$ \cite{KW2}. Additionally, $L_1(\gs\go_{2n+1})$ acts on the free fermion algebra $\cF(2n+1)$, and 
 $$\text{Com}(L_{k+1}(\gs\go_{2n+1}), L_k(\gs\go_{2n+1}) \otimes L_1(\gs\go_{2n+1})) \cong \text{Com}(L_{k+1}(\gs\go_{2n+1}), L_k(\gs\go_{2n+1}) \otimes \cF(2n+1)).$$
 
 In the notation of Theorem \ref{main}, recall the isomorphism
\begin{equation} \begin{split} \cC^{\psi}_{2B}(n,0) & \cong \text{Com}(V^{-2 \psi - 2 n + 2}(\mathfrak{so}_{2n+1}), V^{-2 \psi - 2 n + 1}(\mathfrak{so}_{2n+1}) \otimes \cF(2n+1))^{\mathbb{Z}_2} 
\\ &\cong \cC^{\psi'}_{2B}(0,n) \cong  \cW^{\psi' -n -1/2}(\mathfrak{osp}_{1|2n})^{\mathbb{Z}_2}, \qquad \frac{1}{\psi} + \frac{1}{\psi'} = 2.\end{split} \end{equation} Suppose that the level $-2 \psi - 2 n + 1$ is admissible for $\widehat{\gs\go}_{2n+1}$, that is, 
$$-2 \psi - 2 n + 1 = -(2n-1) + \frac{p}{q},$$ where $p,q \in \mathbb{N}$ are coprime and $p\geq 2n-1$ if $q$ is odd, and $p\geq 2n$ is $q$ is even. In this case, by  \cite[Thm. 8.1 and Rem. 8.3]{CL1} the simple quotient $\cC_{\psi', 2B}(0,n)$ coincides with $\text{Com}(L_{k+1}(\gs\go_{2n+1}), L_k(\gs\go_{2n+1}) \otimes \cF(2n+1))^{\mathbb{Z}_2}$, which we expect to be lisse and rational by analogy with the simply-laced case. This motivates the following conjecture.

\begin{conj} \label{conj:ospcoset} The principal $\cW$-superalgebra $\cW_{\psi' -n -1/2}(\mathfrak{osp}_{1|2n})$ where $\psi'= \frac{p}{2 (p+q)}$, is lisse and rational if 
\begin{enumerate}
\item $p,q \in \mathbb{N}$ are coprime,
\item $p\geq 2n-1$ if $q$ is odd,
\item $p\geq 2n$ if $q$ is even.
\end{enumerate}
By \eqref{2b2b2o}, this conjecture implies that $\cW_{\psi' -m-1/2}(\mathfrak{osp}_{1|2m})$ is also lisse and rational at the Feigin-Frenkel dual level, where $\psi' = \frac{1}{4\psi}  = \frac{p + q}{2 p}$. \end{conj}

As in the case of $\cW_k(\gg)$ for a Lie algebra $\gg$, we expect that rational vertex superalgebras $\cW_{k}(\mathfrak{osp}_{1|2n})$ will serve as building blocks for many non-principal rational $\cW$-superalgebras. In the next subsection, we will give examples of subregular $\cW$-algebras of $\gs\go_{2m+3}$ and principal $\cW$-superalgebras of $\go\gs\gp_{2|2n+2}$ with this property. 

Using the realization $\cC^{\psi}_{2B}(0,m) \cong \cW^{\psi-m-1/2}(\go\gs\gp_{1|2m})^{\mathbb{Z}_2}$, we are able to prove some cases of Conjecture \ref{conj:ospcoset} using the coincidences appearing in Appendices \ref{appendixB} and \ref{appendixC}.

\begin{thm} \label{thm:Wosp1}${}$ \begin{enumerate}
\item For $ k = -(m+\frac{1}{2}) + \frac{2 m-1}{4 (m + r)}$ and $r\in \mathbb{N}$, $\cW_k(\go\gs\gp_{1|2m})$ is lisse and rational when $m+r$ and $1+2r$ are coprime.

\item For $ k =  -(m+\frac{1}{2}) + \frac{1 + 2 m}{2 (1 + 2 m + 2 r)}$ and $r\in \mathbb{N}$, $\cW_k(\go\gs\gp_{1|2m})$ is lisse and rational when $r$ and $1+2m$ are coprime.
\end{enumerate}
\end{thm}

\begin{proof} By Theorem \ref{coinc:typeC-2B}, for $ k = -(m+\frac{1}{2}) + \frac{2 m-1}{4 (m + r)}$, we have
$$\cC_{\psi, 2B}(0,m) =  \cW_{\psi-m-1/2}(\go\gs\gp_{1|2m})^{\mathbb{Z}_2} \cong \cW_s(\gs\gp_{2r}),\qquad s = -(r+1) + \frac{1 + 2 r}{4 (m + r)}.$$
Since $s$ is a nondegenerate admissible level for $\widehat{\gs\gp}_{2r}$, the first statement follows.

Similarly, Theorem \ref{coinc:typeC-2B} also shows that for $ k =  -(m+\frac{1}{2}) + \frac{1 + 2 m}{2 (1 + 2 m + 2 r)}$, we have
$$\cC_{\psi, 2B}(0,m) =  \cW_{\psi-m-1/2}(\go\gs\gp_{1|2m})^{\mathbb{Z}_2} \cong \cW_s(\gs\gp_{2r}),\qquad  s = -(r+1) + \frac{1 + 2 m + 2 r}{4 r}.$$
Again, $s$ is a nondegenerate admissible level for $\widehat{\gs\gp}_{2r}$, so the second statement follows.
\end{proof}

We have a similar result coming from coincidences with algebras of the form $\cW_r(\gs\go_{2n})^{\mathbb{Z}_2}$.
\begin{thm} \label{thm:Wosp2} For $ k = -(m+\frac{1}{2}) + \frac{m}{2 m + 2 r-1}$ and $r\in \mathbb{N}$, $\cW_k(\go\gs\gp_{1|2m})$ is lisse and rational when $2r-1$ and $2m$ are coprime. \end{thm}

\begin{proof} By Theorem \ref{coinc:typeD-2B}, for $ k = -(m+\frac{1}{2}) + \frac{m}{2 m + 2 r-1}$ we have
$$\cC_{\psi, 2B}(0,m) =  \cW_{\psi-m-1/2}(\go\gs\gp_{1|2m})^{\mathbb{Z}_2} \cong \cW_s(\gs\go_{2r})^{\mathbb{Z}_2},\qquad s = -(2r-2) + \frac{ 2 r -1}{2 m + 2 r -1}.$$
Since $s$ is a nondegenerate admissible level for $\widehat{\gs\go}_{2r}$, the claim follows.
\end{proof}

In the case of $\cW^k(\go\gs\gp_{1|2})$, it was shown in \cite{CFL} that the diagonal coset 
$$\cC^{r}_2 = \text{Com}(V^{r+2}(\gs\gl_2), V^r(\gs\gl_2) \otimes L_2(\gs\gl_2))$$ is a quotient of $\cW^{\text{ev}}(c,\lambda)$ with parametrization
$$c = \frac{3 r (6 + r)}{2 (2 + r) (4 + r)},\qquad \lambda = -\frac{2 (r +2) (r + 4) (-5248 - 4488 r - 352 r^2 + 132 r^3 + 11 r^4)}{7 (r -2) (r +8) (68 + 42 r + 7 r^2) (352 + 354 r + 59 r^2)}.$$
A calculation shows that $\cC^{r}_2 \cong \cC^{\psi}_{2B}(0,1) = \cW^{\psi-3/2}(\go\gs\gp_{1|2})^{\mathbb{Z}_2}$, where $\psi$ and $r$ are related by $\psi = \frac{2 + r}{2 (4 + r)}$ or $\psi = \frac{4 + r}{2 (2 + r)}$. Since the simple quotient $\cC_{r,2}$ is lisse and rational whenever $r$ is admissible for $\widehat{\gs\gl}_2$, we obtain 
\begin{thm} \label{thm:osp12} $\cW_{\psi-3/2}(\go\gs\gp_{1|2})$ is lisse and rational when $\psi = \frac{2 + r}{2 (4 + r)}$ or $\psi = \frac{4 + r}{2 (2 + r)}$, and $r$ is admissible for $\widehat{\gs\gl}_2$.
\end{thm}

\subsection{Subregular $\cW$-algebras of type $B$}
Recall that $\cW_k(\gs\go_{2m+3}, f_{\text{subreg}})$ for $m\geq 1$ is exceptional in the sense of \cite{AvE} when $k = -(2m+1) + \frac{p}{q}$ is admissible and $q = 2m+2$ or $2m+1$; see Table 1 of \cite{AvE}. It is therefore lisse \cite{Ar1}, but the rationality is only known in the case $m = 1$ \cite{F}. We will prove the rationality in all cases where $q = 2m+2$, and in all cases where $q = 2m+1$ and $p$ is odd. In the missing cases where $q = 2m+1$ and $p$ is even, we will see that rationality would follow from Conjecture \ref{conj:ospcoset}. 

Recall that $\text{Com}(\cH(1), \cW_{\psi-2m-1}(\gs\go_{2m+3}, f_{\text{subreg}}))^{\mathbb{Z}_2}$ can be identified with $\cC_{\psi, 1D}(1,m)$. We set $\psi = \frac{p}{q}$ as above, and we begin with the case $q = 2m+2$.
\begin{thm} \label{Brational1} For all $\psi = \frac{3 + 2 m + 2 r}{2m+2}$ such that $r\in \mathbb{N}$ and $m+1$ and $2r+1$ are coprime, $\cW_{\psi-2m-1}(\gs\go_{2m+3}, f_{\text{subreg}})$ is lisse and rational.
\end{thm}

\begin{proof} By Theorem \ref{coinc:typeC-1D}, we have 
\begin{equation*} \begin{split} \cC_{\psi, 1D}(1,m) & \cong \text{Com}(\cH(1), \cW_{\psi-2m-1}(\gs\go_{2m+3}, f_{\text{subreg}}))^{\mathbb{Z}_2}
\\ & \cong \cW_s(\gs\gp_{2r}),\quad s = -(r+1) + \frac{2 m   + 2 r +3}{2 (2 r+1)}.\end{split} \end{equation*}
Note that the first isomorphism holds by \cite[Thm. 8.1 and Rem. 8.3]{CL1}.
Under the above arithmetic condition, $s$ is a nondegenerate admissible level for $\widehat{\gs\gp}_{2r}$, so $\cC_{\psi, 1D}(1,m) $ is lisse and rational. Therefore $\text{Com}(\cH(1), \cW_{\psi-2m-1}(\gs\go_{2m+3}, f_{\text{subreg}}))$, being a simple current extension of $\cC_{\psi, 1D}(1,m)$ is also lisse and rational. It then follows from Proposition \ref{rationalextension} that $\cW_{\psi-2m-1}(\gs\go_{2m+3}, f_{\text{subreg}})$ is lisse and rational as well.
\end{proof}

Recall from \eqref{2d1d1o} in the case $n=1$ that
$$\text{Com}(\cH(1), \cW_{\psi-2m-1}(\gs\go_{2m+3}, f_{\text{subreg}}))^{\mathbb{Z}_2} \cong \text{Com}(\cH(1), \cW_{\psi'-m}(\mathfrak{osp}_{2|2m+2}))^{\mathbb{Z}_2},\qquad  \psi' = \frac{1}{2\psi}.$$ This is the $\mathbb{Z}_2$-invariant part of the duality \eqref{eq:CGN} proved in \cite{CGN}. We obtain

\begin{cor} \label{Brational2} For $\psi' = \frac{1 + m}{3 + 2 m + 2 r}$ such that $r\in \mathbb{N}$ and $m+1$ and $2r+1$ are coprime, $\cW_{\psi'-m}(\mathfrak{osp}_{2|2m+2})$ is lisse and rational.
\end{cor} 
The fact that $\cW_{\psi'-m}(\mathfrak{osp}_{2|2m+2})$ is lisse was also pointed out in \cite{CGN} as a consequence of \eqref{eq:CGN} together with the lisseness of $\cW_{\psi-2m-1}(\gs\go_{2m+3}, f_{\text{subreg}})$.

Next, we consider the case where $q = 2m+1$ and $p$ is odd.
\begin{thm} \label{Brational3} For $\psi = \frac{2 m + 2 r+1}{2m+1}$ such that $r\in \mathbb{N}$ and $r$ and $2m+1$ are coprime,  $\cW_{\psi-2m-1}(\gs\go_{2m+3}, f_{\text{subreg}})$ is lisse and rational.
\end{thm}
\begin{proof} By Theorem \ref{coinc:typeD-1D} and  \cite[Thm. 8.1 and Rem. 8.3]{CL1}, we have 
\begin{equation*} \begin{split} \cC_{\psi, 1D}(1,m) & \cong \text{Com}(\cH(1), \cW_{\psi-2m-1}(\gs\go_{2m+3}, f_{\text{subreg}}))^{\mathbb{Z}_2}
\\ & \cong \cW_s(\gs\go_{2r})^{\mathbb{Z}_2},\qquad s = -(2r-2) + \frac{2 r}{2 m  + 2 r +1}.\end{split} \end{equation*}
As above, $s$ is a nondegenerate admissible level for $\widehat{\gs\go}_{2r}$, so $\cC_{\psi, 1D}(1,m) $ is lisse and rational. Therefore $\text{Com}(\cH(1), \cW_{\psi-2m-1}(\gs\go_{2m+3}, f_{\text{subreg}}))$ is also lisse and rational, and so is $\cW_{\psi-2m-1}(\gs\go_{2m+3}, f_{\text{subreg}})$.
\end{proof}

\begin{cor} \label{Brational4} For $\psi' = \frac{2 m+1}{2(2 m + 2 r+1)}$ such that $r\in \mathbb{N}$ and $r$ and $2m+1$  are coprime, $\cW_{\psi'-m}(\mathfrak{osp}_{2|2m+2})$ is lisse and rational.
\end{cor} 

We now consider the case where $q = 2m+1$ and $p$ is even. By Theorem \ref{coinc:typeO-1D}, we have
\begin{equation} \label{iso:missingcases} \cC_{\psi, 1D}(1,m) \cong  \cW_{s}(\go\gs\gp_{1|2r})^{\mathbb{Z}_2},\qquad \psi = \frac{2 (m +  r+1)}{2 m+1},\qquad s = -(r+\frac{1}{2}) + \frac{m +  r+1}{1  + 2 r}.\end{equation}
For $r=1$, we have $\cC_{\psi, 1D}(1,m) \cong \cW_s(\go\gs\gp_{1|2})^{\mathbb{Z}_2} \cong \cC_{a,2}$, where
$$\psi =\frac{2 (2 + m)}{2 m+1},\qquad s = -\frac{3}{2} + \frac{2 + m}{3}, \qquad a = -2 + \frac{6}{2 m+1}.$$ Since $a$ is admissible for $\widehat{\gs\gl}_2$, it follows from Theorem \ref{thm:osp12} that 

\begin{cor} \label{Brational5} For $\psi = \frac{2 (2 + m)}{2 m+1}$, $\cW_{\psi-2m-1}(\gs\go_{2m+3}, f_{\text{subreg}})$ is lisse and rational. Similarly, for $\psi' = \frac{2 m+1}{4 (2 + m)}$, $\cW_{\psi'-m}(\mathfrak{osp}_{2|2m+2})$ is lisse and rational.
\end{cor}

\begin{remark} \label{rem:morerationalosp} If $\psi = \frac{2 (m +  r+1)}{ 2m+1}$,  $m+r+1$ and $2m+1$ are coprime, and $r>1$, we are not able to prove the rationality of $\cW_{\psi-2m-1}(\gs\go_{2m+3}, f_{\text{subreg}})$ using the methods of this paper. However, due to McRae's recent proof of the Kac-Wakimoto-Arakawa conjecture in full generality \cite{McRae2}, these algebras are indeed rational. In view of \eqref{iso:missingcases} together with Proposition \ref{rationalextension}, it follows that $\cW_{s}(\go\gs\gp_{1|2r})$ is rational when $s = -(r+\frac{1}{2}) + \frac{m +  r+1}{1  + 2 r}$. This proves an additional family of cases of Conjecture \ref{conj:ospcoset}.
\end{remark}

It is natural to ask whether the examples where $k = -(2m+1) + \frac{p}{q}$ is admissible and $q = 2m+2$ or $2m+1$, account for all cases where $\cW_k(\mathfrak{so}_{2m+3}, f_{\text{subreg}})$ is lisse and rational. It turns out that this is {\it not} the complete list. For example, we have isomorphisms 
\begin{equation} \begin{split} & \cC_{\psi, 1D}(1,m) \cong \cW_s(\go\gs\gp_{1|2})^{\mathbb{Z}_2} \cong \cC_{a,2},
\\ & \psi =\frac{2 m}{2 m-1},\qquad   s= -\frac{3}{2} + \frac{m}{2m-1}, \qquad a = 4m-4.\end{split} \end{equation} Since $a$ is a positive integer for $m >1$, we obtain

\begin{cor} \label{cor:newrationalb} For $ \psi = \frac{2 m}{2 m-1}$, $\cW_{\psi-2m-1}(\gs\go_{2m+3}, f_{\text{subreg}})$ is lisse and rational. Similarly, for $\psi' = \frac{2 m-1}{4 m}$, $\cW_{\psi'-m}(\mathfrak{osp}_{2|2m+2})$ is lisse and rational.
\end{cor}

\begin{remark} \label{rem:newtypebsubreg} The examples in Corollary \ref{cor:newrationalb} fit into the third family of coincidences in Theorem \ref{coinc:typeO-1D}, namely,
$$\cC_{\psi, 1D}(1,m) \cong  \cW_{s}(\go\gs\gp_{1|2r})^{\mathbb{Z}_2},\qquad \psi = \frac{2 (m - r +1)}{1 + 2 m - 2 r},\qquad s = -(r+\frac{1}{2}) +  \frac{m + 1 - r}{2m +1 -2 r }.$$ Since $ \frac{m + 1 - r}{2m +1 -2 r } = \frac{p+q}{2p}$ for $p= 2 m - 2 r + 1$ and $q=1$, Conjecture \ref{conj:ospcoset} would imply that $\cW_{s}(\go\gs\gp_{1|2r})^{\mathbb{Z}_2}$ is lisse and rational whenever $m \geq 2r-1$, and hence that the following algebras are lisse and rational:
\begin{equation} \begin{split} & \cW_{\psi-2m-1}(\gs\go_{2m+3}, f_{\text{subreg}}),\qquad \psi = \frac{2 (m - r +1)}{1 + 2 m - 2 r},\qquad m\geq 2r-1,
\\ & \cW_{\psi'-m}(\mathfrak{osp}_{2|2m+2}),\qquad  \frac{1 + 2 m - 2 r}{4 (m - r +1)}, \qquad m\geq 2r-1.\end{split} \end{equation}
\end{remark}

\subsection{Minimal $\cW$-algebras of type $C$}
Here we prove another case of the Kac-Wakimoto rationality conjecture, which involves the minimal $\cW$-algebras $\cW_{r-1/2}(\gs\gp_{2n+2}, f_{\text{min}})$ for all integers $r,n\geq 1$. Recall that in the case $m=1$, 
$\cW^{\psi}_{2C}(n,1) = \cW^{\psi -n-2}(\gs\gp_{2n+2}, f_{\text{min}})$, and has affine subalgebra $V^{\psi-n-3/2}(\gs\gp_{2n})$. If we specialize to the case $\psi = \frac{3+2n+2r}{2}$ for $r$ a positive integer, it was shown in \cite{ACKL} that we have an induced embedding of simple vertex algebras $L_r(\gs\gp_{2n}) \rightarrow \cW_{r-1/2}(\gs\gp_{2n+2}, f_{\text{min}})$. By \cite[Thm. 8.1]{CL1}, the coset $$\text{Com}(L_r(\gs\gp_{2n}), \cW_{r-1/2}(\gs\gp_{2n+2}, f_{\text{min}}))$$ is simple and coincides with the simple quotient $\cC_{\psi,2C}(n,1)$ of $\cC^{\psi}_{2C}(n,1)$.

\begin{thm} \label{mintypeC} For all positive integers $n,r$, $\cW_{r-1/2}(\gs\gp_{2n+2}, f_{\text{min}})$ is lisse and rational, and is an extension of $L_r(\gs\gp_{2n}) \otimes \cW_s(\gs\gp_{2r})$ for $s = -(r+1) + \frac{1+n+r}{3+2n+2r}$.
\end{thm}

\begin{proof} By Theorem \ref{coinc:typeC-2C} and \cite[Thm. 8.1]{CL1}, for $\psi = \frac{3+2n+2r}{2}$ and $r$ a positive integer we have 
\begin{equation} \label{minimaliso} \cC_{\psi,2C}(n,1) \cong \text{Com}(L_r(\gs\gp_{2n}), \cW_{r-1/2}(\gs\gp_{2n+2}, f_{\text{min}}) \cong \cW_s(\gs\gp_{2r}),\qquad s = -(r+1) + \frac{1+n+r}{3+2n+2r}.\end{equation}
Since $s$ is a nondegenerate admissible level for $\cW_s(\gs\gp_{2r})$, $\cC_{\psi,2C}(n,1) $ is lisse and rational. Therefore $\cW_{r-1/2}(\gs\gp_{2n+2}, f_{\text{min}})$ is an extension of $L_r(\gs\gp_{2n}) \otimes \cW_s(\gs\gp_{2r})$, and hence is also lisse and rational by Proposition \ref{rationalextension}.
\end{proof}
The isomorphism \eqref{minimaliso} was first conjectured in \cite{ACKL}, and was shown in \cite{KL} to be equivalent to the explicit truncation curve; see \cite[Conj. 7.4]{KL}. This curve is also given by $\Phi_{2C,n,m}(\psi)$ in the case $m=1$, which can be obtained from the formula \eqref{truncationcurve2B} for $\Phi_{2B,n,m}(\psi)$, together with \eqref{eq:trialities:parametrization} and \eqref{relationship3}.

\begin{remark}\label{rem:collapsing}  In the case $r=0$ and $\psi = \frac{3+2n}{2}$, $\cC_{\psi,2C}(n,1)$ is the simple quotient of $\cW^{I, \text{ev}}(c,\lambda)$, where $I$ is the maximal ideal generated by $c$ and $(\lambda + \frac{44 + 53 n + 12 n^2}{77 (20 + 29 n + 12 n^2)})$. Using \cite[Eq. (3.6) and (3.8)]{KL} and the recursive structure of the OPE algebra of $\cW^{\text{ev}}(c,\lambda)$ given by \cite[Thm. 3.9]{KL}, it is not difficult to check that the generators $L$ and $W^4$ of $\cW^{I, \text{ev}}(c,\lambda)$ lie in the maximal proper ideal, so $\cC_{\psi,2C}(n,1)\cong \mathbb{C}$. Therefore Theorem \ref{mintypeC} holds for $r=0$ as well. This provides an alternative proof of the fact that $L_0(\gs\gp_{2n})\hookrightarrow \cW_{-1/2}(\gs\gp_{2n+2}, f_{\text{min}})$ is a conformal embedding for all $n\geq 1$ \cite{AKMPP}.
\end{remark}

\subsection{Cosets of type $C$}\label{subsection:admissible}
It is a longstanding conjecture that if $\cA \subseteq \cV$ are both lisse and rational vertex algebras, the coset $\cC = \text{Com}(\cA, \cV)$ is also lisse and rational. This is a theorem if $\cA$ is a lattice vertex algebra \cite{CKLR}, but otherwise is it known only in isolated examples. In fact, there are even more general situations where coset vertex algebras can be lisse and rational. For example, \eqref{eqn:aclmain} implies that when $\gg$ is simply-laced, $\text{Com}(L_{k+1}(\gg), L_k(\gg) \otimes L_1(\gg))$ is lisse and rational for all admissible levels $k$. We expect the following generalization of this statement to hold.
\begin{conj} Let $\gg$ be a simple, finite-dimensional Lie algebra, $r$ a positive integer, and $k$ an admissible level for $\widehat{\gg}$. Then the coset $\text{Com}(L_{k+r}(\gg), L_k(\gg) \otimes L_r(\gg))$ is lisse and rational.
\end{conj}

This is known for all admissible levels $k$ in the special case $\gg = \gs\gl_2$ and $r = 2$ \cite{Ad,CFL}, and also when $k$ is a positive integer and  $r=2$ in the case of $E_8$ \cite{Lin}\footnote{Note that the argument of \cite{Lin} for admissible $k$ also applies if one uses \cite[Thm. 5.5]{CKM2}}. The next result gives another special case and will be useful for proving the rationality of other interesting cosets later.

\begin{thm}\label{cosetC1}
For $k \in \mathbb Z_{\geq 1}$, the coset $$\text{Com}(L_{k-1/2}(\gs\gp_{2n}), L_k(\gs\gp_{2n}) \otimes  L_{-1/2}(\gs\gp_{2n})) \cong \cW_{\ell} (\gs\gp_{2k}),$$ with $\ell = -(k + 1) + \frac{1 + n + k}{1 + 2 n + 2 k}$. In particular, this coset is lisse and rational.
\end{thm}
\begin{proof}
Note that 
\cite[Cor. 4.1]{KW1} tells us that $L_{k-1/2}(\gs\gp_{2n})$ embeds into $L_k(\gs\gp_{2n}) \otimes L_{-1/2}(\gs\gp_{2n})$ if $k$ is a positive integer. 
Hence we get that 
$\text{Com}(L_{k-1/2}(\gs\gp_{2n}), L_k(\gs\gp_{2n}) \otimes L_{-1/2}(\gs\gp_{2n}))$ is simple as well, again by \cite[Prop. 5.4]{CGN}, which applies since $L_k(\gs\gp_{2n}) \otimes L_{-1/2}(\gs\gp_{2n}))$ is an ordinary module for $L_{k-1/2}(\gs\gp_{2n})$, and that category is completely reducible \cite{Ar3}. Thus by Theorem \ref{coinc:typeC-2C} and \cite[Thm. 8.1]{CL1}
$$\text{Com}(L_{k-1/2}(\gs\gp_{2n}), L_k(\gs\gp_{2n}) \otimes \cS(n) ) \cong \cW_{\ell} (\gs\gp_{2k}),\qquad \ell = -(k + 1) + \frac{1 + n + k}{1 + 2 n + 2 k}.$$
The vertex algebra  $\cS(n)$ decomposes as
 $\cS(n)  \cong L_{-1/2}(\gs\gp_{2n}) \oplus L_{-1/2}(\omega_1)$ with $\omega_1$ the first fundamental weight of $\gs\gp_{2n}$, i.e., the top level of $L_{-1/2}(\omega_1)$ is the standard representation of $\gs\gp_{2n}$. Since $\omega_1$ is not in the root lattice of $\gs\gp_{2n}$, $L_{k-1/2}(\gs\gp_{2n})$ cannot be a submodule of $L_k(\gs\gp_{2n}) \otimes L_{-1/2}(\omega_1)$. It follows that 
$$\text{Com}(L_{k-1/2}(\gs\gp_{2n}), L_k(\gs\gp_{2n}) \otimes  L_{-1/2}(\gs\gp_{2n})) \cong \text{Com}(L_{k-1/2}(\gs\gp_{2n}), L_k(\gs\gp_{2n}) \otimes \cS(n) ),$$ which completes the proof.
\end{proof} 
The category of ordinary modules of $L_{k-1/2}(\gs\gp_{2n})$ is semisimple \cite{Ar3} and we denote by $P_k$ the set of weights such that $L_{k-1/2}(\lambda)$ is an ordinary module for $L_{k-1/2}(\gs\gp_{2n})$. We have
\begin{equation}\label{eq:decompC}
L_k(\gs\gp_{2n}) \otimes  L_{-1/2}(\gs\gp_{2n}) \cong  \bigoplus_{\lambda \in P_k \cap Q} L_{k-1/2}(\lambda) \otimes M(\lambda).
\end{equation}
Here each multiplicity space $M(\lambda)$ is either a direct sum of $\cW_{\ell} (\gs\gp_{2k})$-modules or zero. In fact $M(\lambda)$ can only be non-zero if $\lambda$ is in the root lattice $Q$ of $\gs\gp_{2n}$ and so we restrict the sum to $P_k \cap Q$. Finally, $M(0) \cong \cW_{\ell} (\gs\gp_{2k})$.

Let $ f_{\text{min}}$ be a minimal nilpotent element.
The minimal reduction functor $H_{k, f_{\text{min}}}$ at level $k$, see \eqref{Wmodule}, has the property that for an irreducible highest-weight module $L_k(\lambda)$ of the affine vertex algebra of $\gg$, the reduction 
$H_{k, f_{\text{min}}}(L_k(\lambda))$ is an irreducible ordinary module of the minimal $\cW$-algebra $\cW^k(\gg, f_{\text{min}})$ as long as $k$ is not a positive integer \cite{Ar4}.
We aim to determine the $\lambda$, such that $H_{k, f_{\text{min}}}( L_{k-1/2}(\lambda))  \cong  \cW_{k-1/2}(\gs\gp_{2n}, f_{\text{min}})$.
\begin{lemma}\label{lemma;min} For $\gg= \gs\gp_{2n}$ and $k \in \mathbb Z_{\geq 1}$,
$H_{k, f_{\text{min}}}( L_{k-1/2}(\lambda))  \cong  \cW_{k-1/2}(\gs\gp_{2n}, f_{\text{min}})$ implies $\lambda = m\omega_1$ with $m \in  \{ 0, 2k+1\}$.
\end{lemma} 
\begin{proof}
The minimal $\cW$-algebra has an affine subalgebra of type $\gs\gp_{2n-2}$.
The top level of $H_{k, f_{\text{min}}}( L_{k-1/2}(\lambda))$ is described in \cite[(66)]{Ar4} (see also \cite[(6.14)]{KW3}), and has highest weight $\lambda$ restricted to the Cartan subalgebra of $\gs\gp_{2n-2}$. The conformal weight of the top level is the conformal weight of the top level of $L_{k-1/2}(\lambda)$ minus $\lambda(x)$, where $x$ is in the Cartan subalgebra of the $\gs\gl_2$-triple for the quantum Hamiltonian reduction; see Section \ref{sec:W-algebras}.

In the case of  $\gg=\gs\gp_{2n}$, we embed the root system as usual in $\mathbb Z^n$ with orthonormal basis $\{\epsilon_1, \dots, \epsilon_n\}$. Then simple positive roots are $ \alpha_1 = \frac{1}{\sqrt{2}}(\epsilon_1 -\epsilon_2), \dots, \alpha_{n-1} = \frac{1}{\sqrt{2}}(\epsilon_{n-1} -\epsilon_n), \alpha_n = \sqrt{2} \epsilon_n$. The longest short co-root is $ \theta_s^\vee = {\sqrt{2}}(\epsilon_1+ \epsilon_2)$ and the Weyl vector is 
$\rho= \frac{1}{\sqrt{2}}(n\epsilon_1 + (n-1)\epsilon_2 + \dots + \epsilon_n)$. 
The $\gs\gl_2$-triple corresponds to the longest root, that is $\sqrt{2}\epsilon_1$ and so it follows that the top level of $H_{k, f_{\text{min}}}( L_{k-1/2}(\lambda))$ has $\gs\gp_{2n-2}$ weight zero if and only if $\lambda = m\omega_1$ is a multiple of the first fundamental weight $\omega_1 = \frac{\epsilon_1}{\sqrt{2}}$. Moreover $\lambda$ is an admissible weight if and only if $\lambda \theta_s^\vee \leq 2k+2n +1 - 2h = 2k +1$ with $h =2n$ the Coxeter number of $\gs\gp_{2n}$. The conformal weight of the top level is for $\lambda = m\omega_1$,
\[
\frac{\lambda(\lambda+2\rho)}{ 2k+2n+1} - \lambda \omega_1 = \frac{m}{2}\left( \frac{m+2n}{2k+2n+1} -1 \right), 
\]
i.e. it vanishes if either $m=0$ or $m=2k+1$. 
\end{proof}

Recall from Remark \ref{rem:collapsing} that $H_{-1/2, f_{\text{min}}}(L_{-1/2}(\gs\gp_{2n})) \cong \mathbb C$.
Applying Theorem \ref{Urod} to \eqref{eq:decompC} with $L = L_k(\gs\gp_{2n})$ and $V = L_{-1/2}(\gs\gp_{2n})$
 yields
\begin{equation}\label{eq:decompC2}
\begin{split}
L_k(\gs\gp_{2n}) &\cong L_k(\gs\gp_{2n}) \otimes \mathbb C \cong L_k(\gs\gp_{2n}) \otimes  H_{-1/2, f_{\text{min}}}(L_{-1/2}(\gs\gp_{2n})) \\
& \cong H_{k-1/2, f_{\text{min}}}(L_k(\gs\gp_{2n}) \otimes  (L_{-1/2}(\gs\gp_{2n})) \\
&\cong  \bigoplus_{\lambda \in P_k \cap Q} H_{k-1/2, f_{\text{min}}}( L_{k-1/2}(\lambda)) \otimes M(\lambda). 
\end{split}
\end{equation}
Since $(2k+1)\omega_1$ is not in the root lattice $Q$ of $\gs\gp_{2n}$, we can use Lemma \ref{lemma;min} to conclude that 
$$\text{Com}(\cW_{k-1/2}(\gs\gp_{2n},  f_{\text{min}}), L_k(\gs\gp_{2n})  ) \cong \cW_{\ell} (\gs\gp_{2k}),\qquad \ell = -(k + 1) + \frac{1 + n + k}{1 + 2 n + 2 k},$$
Recall Theorem \ref{mintypeC} saying that
\begin{equation} \nonumber 
 \text{Com}(L_k(\gs\gp_{2n-2}), \cW_{k-1/2}(\gs\gp_{2n}, f_{\text{min}})) \cong \cW_s(\gs\gp_{2k}),\qquad s = -(k+1) + \frac{n+k}{1+2n+2k}.\end{equation}
We can thus employ 
Corollary \ref{cor:triplecoset} with $V =  L_k(\gs\gp_{2n}), W_1 = \cW_{\ell} (\gs\gp_{2k}), W_2 = \cW_{k-1/2}(\gs\gp_{2n}, f_{\text{min}}),$ $W_3 = \cW_s(\gs\gp_{2k})$ and $L = L_k(\gs\gp_{2n-2})$ to conclude that 
\begin{cor}\label{cor:typeC-coset1}
For $k \in \mathbb Z_{\geq 1}$ and $n \in \mathbb Z_{\geq 2}$ the coset $\text{Com}( L_k(\gs\gp_{2n-2}),  L_k(\gs\gp_{2n}))$ is lisse and rational and is an extension of $\cW_{\ell} (\gs\gp_{2k}) \otimes \cW_{s} (\gs\gp_{2k})$ with $ \ell = -(k + 1) + \frac{1 + n + k}{1 + 2 n + 2 k}$ and $s = -(k+1) + \frac{n+k}{1+2n+2k}$.
\end{cor}
\begin{remark}
By Remark \ref{remark:Urodembedding} the embedding of  $L_k(\gs\gp_{2n-2})$ in  $L_k(\gs\gp_{2n})$  is the standard one described in Remark \ref{remark: embedding} (and $m=1$).

The standard coset conformal vector of the coset $\text{Com}( L_k(\gs\gp_{2n-2}),  L_k(\gs\gp_{2n}))$ is the difference of the Sugawara vectors  of $L_k(\gs\gp_{2n})$ and $L_k(\gs\gp_{2n-2})$. 
Note that the conformal vector of $\cW_{\ell} (\gs\gp_{2k}) \otimes \cW_{s} (\gs\gp_{2k})$ is not the standard coset conformal vector. The contragredient dual and being of CFT-type depend on the choice of a conformal vector, e.g. the coset $\text{Com}( L_k(\gs\gp_{2n-2}),  L_k(\gs\gp_{2n}))$  is self-contragredient and of CFT-type with the standard coset conformal vector. On the other hand, neither the lisse nor rationality properties depend on a choice of conformal vector. 
\end{remark}
We can iterate, i.e. apply Corollary \ref{cor:triplecoset} with $V = L_k(\gs\gp_{2n}), L = L_k(\gs\gp_{2n-2(m+1)}), W_1= \text{Com}(L_k(\gs\gp_{2n-2m}), L_k(\gs\gp_{2n})),  W_2 = L_k(\gs\gp_{2n-2m}), 
W_3 = \text{Com}(L_k(\gs\gp_{2n-2(m+1)}), L_k(\gs\gp_{2n-2m}))$. Then the induction hypothesis is that $W_1= \text{Com}(L_k(\gs\gp_{2n-2m}), L_k(\gs\gp_{2n}))$ is rational and lisse. The base case has just been proven and the induction step is again Corollary \ref{cor:triplecoset}.
\begin{cor}\label{cor:typeC-coset}
For $k \in \mathbb Z_{\geq 1}, n, m \in \mathbb Z_{\geq 1}$  and $n > m$ the coset $\text{Com}( L_k(\gs\gp_{2n-2m}),  L_k(\gs\gp_{2n}))$ is lisse and rational and is an extension of 
\begin{equation} \label{GT:type C} \bigotimes_{i=1}^m \left(\cW_{\ell_i} (\gs\gp_{2k}) \otimes \cW_{s_i} (\gs\gp_{2k})\right)  \end{equation}
 with $ \ell_i  = -(k + 1) + \frac{2 + n - i + k}{3+ 2 n -2i  + 2 k}$ and $s_i = -(k+1) + \frac{1+ n - i +k}{3+2n -2i +2k}$.
\end{cor}
As above, the embedding of  $L_k(\gs\gp_{2n-2m})$ in  $L_k(\gs\gp_{2n})$  is the standard one described in Remark \ref{remark: embedding}.

This coset is isomorphic to another interesting coset via level-rank duality. For this we use that $L_k(\gs\gp_{2n})$ and $L_n(\gs\gp_{2k})$ form a commuting pair in $\cF(4nk)$ \cite[Prop. 2]{KP}; a detailed proof is given in the appendix of \cite{ORS}. We can thus apply the idea of the proof of  \cite[Thm. 13.1]{ACL}, namely
\begin{equation} \nonumber
\begin{split}
\text{Com}\left( L_k(\gs\gp_{2n-2m}),  L_k(\gs\gp_{2n})\right) 
&\cong \text{Com}\left( L_k(\gs\gp_{2n-2m}),   \text{Com}\left(L_n(\gs\gp_{2k}), \cF(4nk) \right)\right) \\
&\cong \text{Com}\left( L_k(\gs\gp_{2n-2m}) \otimes L_n(\gs\gp_{2k}),  \cF(4nk) \right) \\
&\cong \text{Com}\left(   L_n(\gs\gp_{2k}),  \text{Com}\left(L_k(\gs\gp_{2n-2m}), \cF(4nk) \right) \right) \\
&\cong \text{Com}\left(   L_n(\gs\gp_{2k}),  \text{Com}\left(L_k(\gs\gp_{2n-2m}), \cF(4(n-m)k) \right) \otimes \cF(4mk) \right) \\
&\cong \text{Com}\left(   L_n(\gs\gp_{2k}),  L_{n-m}(\gs\gp_{2k}) \otimes \cF(4mk) \right).
\end{split}
\end{equation}
Corollary \ref{cor:typeC-coset} thus gives us 
\begin{cor}\label{cosetC2}
For $k, n, m \in \mathbb Z_{\geq 1}$ and $n > m$, the coset  $\text{Com}\left(   L_n(\gs\gp_{2k}),  L_{n-m}(\gs\gp_{2k}) \otimes \cF(4mk) \right)$ is lisse and rational and is an extension of $$\bigotimes_{i=1}^m \left(\cW_{\ell_i} (\gs\gp_{2k}) \otimes \cW_{s_i} (\gs\gp_{2k})\right)   $$
 with $ \ell_i  = -(k + 1) + \frac{2 + n - i + k}{3+ 2 n -2i  + 2 k}$ and $s_i = -(k+1) + \frac{1+ n - i +k}{3+2n -2i +2k}$.
 \end{cor}

\subsection{Gelfand-Tsetlin algebras in types $B$, $C$ and $D$}  \label{Gelfand-Tsetlin}
Consider the sequence of upper left corner inclusions $\gg\gl_1 \subseteq \gg\gl_2\subseteq \cdots \subseteq \gg\gl_{n+1}$, and let $\cZ_i$ denote the center of $U(\gg_i)$. The Gelfand-Tsetlin algebra is the commutative subalgebra of $U(\gg\gl_{n+1})$ which is generated by $\{\cZ_i|\ i = 1,\dots, n+1\}$ \cite{DFO}. In the setting of affine Lie algebras, the analogous object needs to be defined in a different way because the center of the universal enveloping algebra is trivial for noncritical levels. In \cite{ACL}, it was shown that for all $k,n\in \mathbb{Z}_{\geq 1}$, the coset $\text{Com}(L_k(\gg\gl_{n-1}), L_k(\gg\gl_n))$ is isomorphic to $\cW_{\ell}(\gg\gl_k)$ for $\ell = -k + \frac{k+n-1}{k+n}$. Iterating this construction shows that $L_k(\gg\gl_n)$ is an extension of $\bigotimes_{i=1}^n \cW_{\ell_i}(\gg\gl_k)$ with $\ell_i =- k = \frac{k+n-i}{k+n-i+1}$. This was regarded in \cite{ACL} as a noncommutative, affine analogue of the Gelfand-Tsetlin subalgebra $\Gamma$ of $U(\gg\gl_n)$. Even though $\bigotimes_{i=1}^n \cW_{\ell_i}(\gg\gl_k)$ is noncommutative, its Zhu algebra is commutative, and it maps to $\Gamma$ via the Zhu functor \cite{Z}.

For types $B$ and $D$, a similar observation appears in \cite{CKoL}. For $k,n\in \mathbb{N}$, the cosets \begin{equation} \label{def:bdcosets} \cD_k(n) = \text{Com}(L_k(\gs\go_{2n}), L_k(\gs\go_{2n+1}))^{\mathbb{Z}_2},\quad  \cE_k(n) = \text{Com}(L_k(\gs\go_{2n+1}), L_k(\gs\go_{2n+2}))^{\mathbb{Z}_2},\end{equation} were called generalized parafermion algebras  of orthogonal types. Just as the realization of $L_k(\gg\gl_n)$ as an extension of $\bigotimes_{i=1}^n \cW_{\ell_i}(\gg\gl_k)$ comes from the chain of inclusions $\gg\gl_1 \subseteq \gg\gl_2\subseteq \cdots \subseteq \gg\gl_{n}$, the chains of inclusions $\gs\go_2 \subseteq \gs\go_3 \subseteq \cdots \subseteq \gs\go_{2n+2}$ and $\gs\go_2 \subseteq \gs\go_3 \subseteq \cdots \subseteq \gs\go_{2n+1}$ imply that $L_k(\gs\go_{2n+2})$ and $L_k(\gs\go_{2n+1})$ are extensions of 
\begin{equation}\label{GT:typeBandD} \begin{split} & \cH \otimes \cD_k(1) \otimes \cE_k(1)  \otimes \cD_k(2) \otimes \cE_k(2) \otimes \cdots \otimes \cD_k(n-1)\otimes \cE_k(n-1)\otimes \cD_k(n)\otimes \cE_k(n),
\\ & \cH \otimes \cD_k(1) \otimes \cE_k(1)  \otimes \cD_k(2) \otimes \cE_k(2) \otimes \cdots \otimes \cD_k(n-1)\otimes \cE_k(n-1) \otimes \cD_k(n).\end{split} \end{equation}
Therefore the algebras \eqref{GT:typeBandD} can be regarded as analogues of the Gelfand-Tsetlin algebra, and they also have commutative Zhu algebras. Since $\cD_k(m) \cong \cC_{k+2m-1, 1D}(m,0)$ and $\cE_k(m) \cong \cC_{k+2m, 1B}(m,0)$, it follows from Theorems \ref{coinc:typeD-1B} and \ref{coinc:typeD-1D} that for $r\in \mathbb{N}$ and $k = 2r$,
\begin{equation} \begin{split}
& \cD_{2r}(m) \cong \cW_s(\gs\go_{2r})^{\mathbb{Z}_2},\qquad s = -(2r-2) + \frac{2m+2r-2}{2m + 2r -1},
\\ & \cE_{2r}(m) \cong \cW_s(\gs\go_{2r})^{\mathbb{Z}_2},\qquad s = -(2r-2) + \frac{2m+2r - 1}{2m+2r}. \end{split} \end{equation}
In particular, if $k = 2r$ is even, the Gelfand-Tsetlin subalgebras of $L_{2r}(\gs\go_{2n+1})$ and $L_{2r}(\gs\go_{2n+2})$ are tensor products of rational vertex algebras of the form $\cW_s(\gs\go_{2r})^{\mathbb{Z}_2}$.

Similarly, for $r\in \mathbb{N}$ and $k = 2r+1$, it follows from Theorems \ref{coinc:typeO-1B} and \ref{coinc:typeO-1D} that
\begin{equation} \label{genparaoddlevel} \begin{split}
& \cD_{2r+1}(m) \cong \cW_s(\go\gs\gp_{1|2r})^{\mathbb{Z}_2},\qquad s = -(r + \frac{1}{2}) + \frac{m+r}{2m+2r -1},
\\ & \cE_{2r+1}(m) \cong \cW_s(\go\gs\gp_{1|2r})^{\mathbb{Z}_2},\qquad s = -(r + \frac{1}{2}) + \frac{m+r}{2m + 2r + 1}. \end{split} \end{equation}
So if $k = 2r+1$ is odd, the Gelfand-Tsetlin subalgebras of $L_{2r+1}(\gs\go_{2n+1})$ and $L_{2r+1}(\gs\go_{2n+2})$ are tensor products of algebras of the form $\cW_s(\go\gs\gp_{1|2r})^{\mathbb{Z}_2}$, which are expected to be rational by Conjecture \ref{conj:ospcoset}.

In type $C$, it follows from Corollary \ref{cor:typeC-coset1} that for all $k,n \in \mathbb{Z}_{\geq 1}$, $L_k(\gs\gp_{2n})$ is an extension of the rational vertex algebra 
$$ \bigotimes_{i=1}^n \left(\cW_{\ell_i} (\gs\gp_{2k}) \otimes \cW_{s_i} (\gs\gp_{2k})\right) $$ with $ \ell_i  = -(k + 1) + \frac{2 + n - i + k}{3+ 2 n -2i  + 2 k}$ and $s_i = -(k+1) + \frac{1+ n - i +k}{3+2n -2i +2k}$. Again, we regard this as a Gelfand-Tsetlin subalgebra of $L_k(\gs\gp_{2n})$, and its Zhu algebra is commutative.

\appendix

\section{Explicit truncation curve for $\cC^{\psi}_{2B}(n,m)$} \label{appendixA}
Here we give the explicit parametrization of the truncation curve for $\cC^{\psi}_{2B}(n,m)$. For all $n,m$ with $n+m \geq 1$, $$\cC^{\psi}_{2B}(n,m) \cong \cW^{\rm ev}_{I_{2B,n,m}}(c,\lambda),$$ where the ideal $I_{2B, n,m}$ is described explicitly via the parametrization 
\begin{equation} \label{truncationcurve2B} \begin{split} c_{2B,n,m}(\psi) = & -\frac{(-m + n - \psi + 2 m \psi) (1 - 2 m + 2 n + 4 m \psi) (-1 - 2 m + 2 n + 2 \psi + 4 m \psi)}{2 \psi (2 \psi-1)},
\\ \lambda_{2B,n,m}(\psi) = & - \frac{2\psi(2\psi-1) f}{ 7 (-m + n + \psi + 2 m \psi) (-1 - 2 m + 2 n + 4 m \psi) (1 - 2 m + 2 n - 2 \psi + 4 m \psi) g h},
\\ f & = -19 m + 80 m^3 - 16 m^5 + 19 n - 240 m^2 n + 80 m^4 n + 240 m n^2 - 
 160 m^3 n^2 - 80 n^3 
 \\ & + 160 m^2 n^3 - 80 m n^4 + 16 n^5 + 49 \psi + 
 114 m \psi - 364 m^2 \psi - 640 m^3 \psi + 160 m^5 \psi - 76 n \psi 
 \\ & + 728 m n \psi + 
 1440 m^2 n \psi - 640 m^4 n \psi - 364 n^2 \psi - 960 m n^2 \psi + 
 960 m^3 n^2 \psi + 160 n^3 \psi
 \\ & - 640 m^2 n^3 \psi + 160 m n^4 \psi - 196 \psi^2 - 
 380 m \psi^2 + 2184 m^2 \psi^2 + 2240 m^3 \psi^2 - 640 m^5 \psi^2 
 \\ & + 228 n \psi^2 - 
 2912 m n \psi^2 - 3840 m^2 n \psi^2 + 1920 m^4 n \psi^2 + 728 n^2 \psi^2  + 
 1920 m n^2 \psi^2  
 \\ & - 1920 m^3 n^2 \psi^2 - 320 n^3 \psi^2 + 640 m^2 n^3 \psi^2 + 
 392 \psi^3 + 760 m \psi^3 - 4368 m^2 \psi^3 - 4480 m^3 \psi^3 
 \\ & + 1280 m^5 \psi^3 - 
 304 n \psi^3 + 2912 m n \psi^3 + 5760 m^2 n \psi^3 - 2560 m^4 n \psi^3 - 
 1920 m n^2 \psi^3 
 \\ & + 1280 m^3 n^2 \psi^3 - 392 \psi^4 - 912 m \psi^4 + 
 2912 m^2 \psi^4 + 5120 m^3 \psi^4 - 1280 m^5 \psi^4 + 304 n \psi^4
 \\ &  - 3840 m^2 n \psi^4 + 1280 m^4 n \psi^4 + 608 m \psi^5 - 2560 m^3 \psi^5 + 
 512 m^5 \psi^5,
\\ g & =  -7 + 4 m^2 - 8 m n + 4 n^2 + 14 \psi - 16 m^2 \psi + 16 m n \psi - 28 \psi^2 + 
 16 m^2 \psi^2,
 \\ h & = 5 m - 20 m^3 - 5 n + 60 m^2 n - 60 m n^2 + 20 n^3 + 49 \psi - 20 m \psi + 
 120 m^3 \psi + 10 n \psi 
 \\ & - 240 m^2 n \psi + 120 m n^2 \psi - 98 \psi^2 + 40 m \psi^2 - 
 240 m^3 \psi^2 - 20 n \psi^2 + 240 m^2 n \psi^2 - 40 m \psi^3 
 \\ & + 160 m^3 \psi^3.
\end{split} \end{equation}

Using this explicit parametrization \eqref{truncationcurve2B} as well as the truncation curve for $\cW^s(\gs\gp_{2r})$ which appears in Appendix A of \cite{KL}, it is easy to verify that these curves intersect at the point $(c,\lambda)$ given by
\begin{equation} \label{app:intersectionpoint}
\begin{split} c = & -\frac{r (-1 - 2 m + 4 n - 4 m r + 4 n r) (1 + 2 m + 2 n + 2 r - 4 m r + 4 n r)}{2 (n + r) (1 + 2 m + 2 r)},
\\ \lambda = & - \frac{2 (n + r) (1 + 2 m + 2 r) f}{7 (1 + 2 r) (2 n + r - 2 m r + 2 n r) (1 + 2 m - 4 m r + 4 n r) gh},
\\ f & = -68 n - 408 m n - 816 m^2 n - 544 m^3 n + 136 n^2 + 544 m n^2 + 544 m^2 n^2 + 96 n^3 + 192 m n^3
\\ &  - 49 r - 256 m r - 360 m^2 r + 64 m^3 r + 304 m^4 r - 212 n r - 1000 m n r - 1456 m^2 n r 
\\ & - 608 m^3 n r + 92 n^2 r - 1296 m n^2 r - 2960 m^2 n^2 r + 1824 n^3 r + 3264 m n^3 r - 576 n^4 r - 196 r^2 
\\ & - 632 m r^2 - 176 m^2 r^2 + 608 m^3 r^2 - 772 n r^2 - 3000 m n r^2 - 496 m^2 n r^2 + 4832 m^3 n r^2 
\\ & + 640 n^2 r^2 - 5792 m n^2 r^2 - 9664 m^2 n^2 r^2 + 4176 n^3 r^2 + 6432 m n^3 r^2 - 1600 n^4 r^2 - 392 r^3
 \\ & - 328 m r^3 + 2368 m^2 r^3 + 2272 m^3 r^3 - 1280 m^4 r^3 - 1544 n r^3 - 5824 m n r^3 + 928 m^2 n r^3 
 \\ & + 3840 m^3 n r^3 +  2240 n^2 r^3 - 4512 m n^2 r^3 - 4480 m^2 n^2 r^3 + 1312 n^3 r^3 + 2560 m n^3 r^3 - 640 n^4 r^3 
 \\ & - 392 r^4 + 608 m r^4 + 2912 m^2 r^4 - 1280 m^3 r^4 - 912 n r^4 - 2784 m n r^4 + 1600 m^2 n r^4 
 \\ & - 640 m^3 n r^4 - 128 n^2 r^4 + 640 m n^2 r^4 + 1920 m^2 n^2 r^4 - 960 n^3 r^4 - 1920 m n^3 r^4 + 640 n^4 r^4 
 \\ & + 608 m r^5  - 1216 m^2 r^5 - 128 m^3 r^5 + 256 m^4 r^5 - 608 n r^5 + 
 2432 m n r^5 + 384 m^2 n r^5 
 \\ & - 1024 m^3 n r^5 - 1216 n^2 r^5 - 
 384 m n^2 r^5 + 1536 m^2 n^2 r^5 + 128 n^3 r^5 - 1024 m n^3 r^5 + 
 256 n^4 r^5,
\\ g & = -7 - 28 m - 28 m^2 + 14 n + 28 m n - 24 n^2 - 14 r - 28 m r - 28 n r - 16 m n r + 16 n^2 r 
\\ & - 28 r^2 + 16 m^2 r^2 - 32 m n r^2 + 16 n^2 r^2,
\\ h & =   -44 n - 88 m n - 49 r - 108 m r - 20 m^2 r - 78 n r + 20 m n r + 40 n^2 r - 98 r^2 - 20 m r^2 
\\ & + 40 n r^2 - 120 m n r^2 + 120 n^2 r^2  - 40 m r^3 + 80 m^2 r^3 + 40 n r^3 - 160 m n r^3 + 80 n^2 r^3.
\end{split}
\end{equation}

In the next three Appendices, we classify coincidences between the simple quotients $\cC_{\psi,1B}(n,m)$, $\cC_{\psi,1D}(n,m)$, $\cC_{\psi,2B}(n,m)$, and $\cC_{\psi,2C}(n,m)$ and the algebras $\cW_s(\gs\gp_{2r})$, $\cW_s(\gs\go_{2r})^{\mathbb{Z}_2}$, and $\cW_s(\go\gs\gp_{1|2r})^{\mathbb{Z}_2}$. There are coincidences at central charges $c = 0, 1, -24, -\frac{22}{5}, \frac{1}{2}$, where the algebra degenerates; see \cite[Thm. 8.1]{KL}. Aside from these points, all additional coincidences correspond to intersection points on the truncation curves. This follows from \cite[Cor. 8.2]{KL}, together with a case-by-case analysis to rule out possible additional coincidences at points where the formula for $\lambda$ is not defined. The details are omitted since the argument is similar to the proof of special cases appearing in Section 9 of \cite{KL}. Via our triality results, similar coincidences can be found for $\cC_{\psi,1C}(n,m)$, $\cC_{\psi,2D}(n,m)$, $\cC_{\psi,1O}(n,m)$, and $\cC_{\psi,2O}(n,m)$ and these are also omitted.

\section{Coincidences with type $C$ principal $\cW$-algebras} \label{appendixB}

\begin{thm} \label{coinc:typeC-1B}  ({\bf Type 1B}) We have the following coincidences.
$$\cC_{\psi,1B}(n,m) \cong \cW_s(\gs\gp_{2r}),$$ for $m,n \geq 0$ and $r\geq 1$.

\begin{enumerate}

\item $\ \displaystyle \psi = \frac{1 + m + n + r}{1 + m},\qquad s = -(r+1) + \frac{1 + m + n + r}{2 (n + r)}$, 
\smallskip

\item $\ \displaystyle  \psi = \frac{2 (m + n)}{1 + 2 m + 2 r},\qquad s = -(r+1) + \frac{1 - 2 n + 2 r}{2 (1 + 2 m + 2 r)}$, 

\smallskip

\item $\ \displaystyle  \psi = \frac{1 + 2 m + 2 n + 2 r}{2 m} ,\qquad s = -(r+1) +\frac{1 + 2 n + 2 r}{2 (1 + 2 m + 2 n + 2 r)}$, 

\smallskip

\item $\ \displaystyle  \psi =  \frac{1 + m + n}{1 + m + r} ,\qquad s = -(r+1) +  \frac{1 + m + r}{2 (r -n)}, \qquad r\neq n$,

\smallskip

\item $\ \displaystyle  \psi =  \frac{2 (m + n - r)}{1 + 2 m - 2 r},\qquad s = -(r+1) + \frac{r -m - n}{2 r - 2 m -1}$, 

\smallskip

\item $\ \displaystyle  \psi =  \frac{1 + 2 m + 2 n - 2 r}{2 (m - r)} ,\qquad s = -(r+1) +  \frac{r -m}{2 r - 2 m - 2 n - 1 }, \qquad r \neq m$.
\end{enumerate}
\end{thm}

\begin{thm} \label{coinc:typeC-1D}  ({\bf Type 1D}) We have the following coincidences  
$$\cC_{\psi, 1D}(n,m) \cong \cW_s(\gs\gp_{2r}),$$ for $m,n \geq 0$ and $r\geq 1$.

\begin{enumerate}

\item $\ \displaystyle \psi = \frac{m + n + r}{m},\qquad s = -(r+1) + \frac{n + r}{2 (m + n + r)}$,

\smallskip

\item $\ \displaystyle \psi = \frac{2 m + 2 n -1}{2 m + 2 r +1},\qquad s =  -(r+1) +\frac{1 - n + r}{1 + 2 m + 2 r}$, 

\smallskip

\item $\ \displaystyle \psi = \frac{1 + 2 m + 2 n + 2 r}{2 (1 + m)},\qquad s =  -(r+1) + \frac{1 + 2 m + 2 n + 2 r}{2 (2 r+ 2 n -1)}$,

\smallskip

\item $\ \displaystyle  \psi = \frac{1 + 2 m + 2 n}{2 (1 + m + r)},\qquad s =  -(r+1) + \frac{1 + m + r}{1 - 2 n + 2 r}$,

\smallskip

\item $\ \displaystyle \psi = \frac{2 m + 2 n - 2 r -1}{2 m - 2 r + 1} ,\qquad s =  -(r+1) + \frac{1 - 2 m - 2 n + 2 r}{2 (2 r  - 2 m -1)}$, 

\smallskip

\item $\ \displaystyle \psi = \frac{m + n - r}{m - r},\qquad s =  -(r+1) + \frac{r -m}{2 (r -m - n)}, \qquad r \neq m, \ n+m$.

\end{enumerate}
\end{thm}

\begin{thm}  \label{coinc:typeC-2B} ({\bf Type 2B}) We have the following coincidences  
$$\cC_{\psi, 2B}(n,m) \cong \cW_s(\gs\gp_{2r}),$$
for $m,n \geq 0$ and $r\geq 1$.

\begin{enumerate}

\item $\ \displaystyle \psi = \frac{1 + 2 m - 2 n + 2 r}{2 (1 + 2 m)} ,\qquad s = -(r+1) + \frac{1 + 2 m - 2 n + 2 r}{4 (r -n)}, \qquad r \neq n$, 

\smallskip

\item $\ \displaystyle \psi = \frac{1 + 2 m - 2 n}{2 (1 + 2 m + 2 r)},  \qquad s = -(r+1) + \frac{1 + 2 m + 2 r}{4 (n + r)}$, 
\smallskip

\item $\ \displaystyle \psi = \frac{m - n + r}{ 2 m -1},\qquad s = -(r+1) + \frac{1 - 2 n + 2 r}{4 (m - n + r)}, \qquad r \neq n-m$,

\smallskip

\item $\ \displaystyle \psi = \frac{m - n - r}{ 2 m - 2 r -1},\qquad s = -(r+1) + \frac{1 - 2 m + 2 r}{4 (n -m + r)},\qquad r \neq m-n$, 
\smallskip

\item $\ \displaystyle \psi = \frac{2 m - 2 n - 2 r -1}{4 (m - r)},\qquad s = -(r+1) + \frac{1 - 2 m + 2 n + 2 r}{4 (r -m)},\qquad r\neq m$, 
\smallskip

\item $\ \displaystyle \psi = \frac{2 m - 2 n - 1}{4 (m + r)} ,\qquad s = -(r+1) + \frac{1 + 2 n + 2 r}{4 (m + r)}$. 

\end{enumerate}
\end{thm}

\begin{thm}  \label{coinc:typeC-2C} ({\bf Type 2C}) We have the following coincidences  
$$\cC_{\psi,2C}(n,m)  \cong \cW_s(\gs\gp_{2r}),$$
for $m,n \geq 0$ and $r\geq 1$.

\begin{enumerate}

\item $\ \displaystyle \psi = \frac{1 + m + n + r}{1 + 2 m} ,\qquad s = -(r+1) + \frac{1 + m + n + r}{1 + 2 n + 2 r}$,

\smallskip

\item $\ \displaystyle \psi = \frac{1 + m + n}{1 + 2 m + 2 r} ,\qquad s = -(r+1) + \frac{1 + 2 m + 2 r}{2 (2 r  - 2 n -1)}$, 

\smallskip

\item $\ \displaystyle \psi = \frac{1 + 2 m + 2 n + 2 r}{2 (2 m -1)} ,\qquad s = -(r+1) + \frac{1 + n + r}{1 + 2 m + 2 n + 2 r}$, 

\smallskip

\item $\ \displaystyle \psi = \frac{m + n}{2 (m + r)} ,\qquad s = -(r+1) + \frac{r -n}{2 (m + r)}$, 

\smallskip

\item $\ \displaystyle \psi = \frac{m + n - r}{2 (m - r)},\qquad s = -(r+1) + \frac{r -m - n}{2 (r -m)},\qquad r \neq m$, 

\smallskip

\item $\ \displaystyle \psi = \frac{1 + 2 m + 2 n - 2 r}{2 ( 2 m - 2 r -1)},\qquad s = -(r+1) + \frac{1 - 2 m + 2 r}{2 (2 r  - 2 m - 2 n -1)}$.

\end{enumerate}
\end{thm}

\section{Coincidences with orbifolds of type $D$ principal $\cW$-algebras} \label{appendixC}

\begin{thm} \label{coinc:typeD-1B} ({\bf Type 1B}) We have the following coincidences 
$$\cC_{\psi, 1B}(n,m) \cong  \cW_s(\gs\go_{2r})^{\mathbb{Z}_2},$$ for $m,n \geq 0$ and $r\geq 2$.
\begin{enumerate}

\item $\ \displaystyle \psi = \frac{2 (m + n + r)}{1 + 2 m} ,\qquad s = -(2r-2) + \frac{2 n + 2 r -1}{2 (m + n + r)}$,

\smallskip

\item $\ \displaystyle \psi = \frac{1 + 2 m + 2 n}{2 (m + r)} ,\qquad s = -(2r-2) + \frac{2 r - 2 n -1}{2 (m + r)}$, 

\smallskip

\item $\ \displaystyle \psi = \frac{1 + m + n - r}{1 + m - r} ,\qquad s = -(2r-2) + \frac{r - m - n  -1}{r - m -1},\qquad r \neq m+1$.
\end{enumerate}
\end{thm}

\begin{thm} \label{coinc:typeD-1D}  ({\bf Type 1D}) We have the following coincidences  
$$\cC_{\psi, 1D}(n,m) \cong \cW_s(\gs\go_{2r})^{\mathbb{Z}_2},$$ for $m,n \geq 0$ and $r\geq 2$.

\begin{enumerate}

\item $\ \displaystyle \psi = \frac{2 m + 2 n + 2 r -1}{1 + 2 m},\qquad s = -(2r-2) + \frac{2 (n + r-1)}{2 m + 2 n + 2 r -1}$,

\smallskip

\item $\ \displaystyle \psi = \frac{m + n}{m + r} ,\qquad s = -(2r-2) +\frac{r -n}{m + r}$, 

\smallskip

\item $\ \displaystyle \psi = \frac{1 + 2 m + 2 n - 2 r}{2 (1 + m - r)},\qquad s = -(2r-2) + \frac{2r - 2 m - 2 n  -1}{2 (r - m  -1)},\qquad r \neq m+1$. 
\end{enumerate}
\end{thm}

\begin{thm} \label{coinc:typeD-2B} ({\bf Type 2B}) We have the following coincidences  
$$\cC_{\psi, 2B}(n,m) \cong \cW_s(\gs\go_{2r})^{\mathbb{Z}_2},$$
for $m,n \geq 0$ and $r\geq 2$.
\begin{enumerate}

\item $\ \displaystyle \psi = \frac{2 m - 2 n + 2 r - 1}{4 m},\qquad s = -(2r-2) + \frac{2 r  - 2 n -1}{2 m - 2 n + 2 r -1}$,
\smallskip

\item $\ \displaystyle \psi =\frac{1 + 2 m - 2 n - 2 r}{2 (1 + 2 m - 2 r)}, \qquad s = -(2r-2) +\frac{2 r  - 2 m - 1}{2 n + 2 r - 2 m - 1}$, 
\smallskip

\item $\ \displaystyle \psi = \frac{m - n}{2 m + 2 r - 1},\qquad s = -(2r-2) + \frac{2 n + 2 r -1}{2 m + 2 r -1}$. 

\end{enumerate} \end{thm}

\begin{thm} \label{coinc:typeD-2C} ({\bf Type 2C}) We have the following coincidences  
$$\cC_{\psi, 2C}(n,m) \cong \cW_s(\gs\go_{2r})^{\mathbb{Z}_2},$$
for $m,n \geq 0$ and $r\geq 2$.
\begin{enumerate}

\item $\ \displaystyle \psi =\frac{m + n + r}{2 m} ,\qquad s = -(2r-2) + \frac{n + r}{m + n + r}$, 

\smallskip

\item $\ \displaystyle \psi = \frac{1 + 2 m + 2 n}{2 (2 m + 2 r -1)},\qquad s = -(2r-2) + \frac{2 ( r -n -1)}{2 m + 2 r -1}$, 

\smallskip

\item $\ \displaystyle \psi = \frac{1 + m + n - r}{1 + 2 m - 2 r},\qquad s = -(2r-2) +\frac{2 (r - m - n - 1)}{2 r - 2 m - 1}$.
\end{enumerate}
\end{thm}

\section{Coincidences with orbifolds of $\cW_{s}(\go\gs\gp_{1|2r})^{\mathbb{Z}_2}$} \label{appendixD}

\begin{thm} \label{coinc:typeO-1B} ({\bf Type 1B}) We have the following coincidences 
$$\cC_{\psi,1B}(n,m) \cong  \cW_{s}(\go\gs\gp_{1|2r})^{\mathbb{Z}_2},$$ for $m,n \geq 0$ and $r\geq 1$.
\begin{enumerate}

\item $\ \displaystyle \psi = \frac{1 + 2 m + 2 n + 2 r}{1 + 2 m} ,\qquad s = -(r+\frac{1}{2}) + \frac{n + r}{1 + 2 m + 2 n + 2 r}$,
\smallskip

\item $\ \displaystyle \psi = \frac{1 + 2 m + 2 n}{1 + 2 m + 2 r} ,\qquad s = -(r+\frac{1}{2}) + \frac{r -n}{1 + 2 m + 2 r}$, 
\smallskip

\item $\ \displaystyle \psi = \frac{1 + 2 m + 2 n - 2 r}{1 + 2 m - 2 r},\qquad s = -(r+\frac{1}{2}) + \frac{2 r - 2 m - 2 n - 1}{2 (2 r - 2 m - 1)}$. 
\end{enumerate}
\end{thm}

\begin{thm} \label{coinc:typeO-1D} ({\bf Type 1D}) We have the following coincidences 
$$\cC_{\psi, 1D}(n,m) \cong  \cW_{s}(\go\gs\gp_{1|2r})^{\mathbb{Z}_2},$$ for $m,n \geq 0$ and $r\geq 1$.
\begin{enumerate}

\item $\ \displaystyle \psi = \frac{2 (m + n + r)}{1 + 2 m} ,\qquad s = -(r+\frac{1}{2}) + \frac{m + n + r}{2 n + 2 r - 1}$, 

\smallskip

\item $\ \displaystyle \psi = \frac{2 (m + n)}{1 + 2 m + 2 r} ,\qquad s = -(r+\frac{1}{2}) +  \frac{1 - 2 n + 2 r}{2 (1 + 2 m + 2 r)}$,

\smallskip

\item $\ \displaystyle \psi = \frac{2 (m + n - r)}{1 + 2 m - 2 r} ,\qquad s = -(r+\frac{1}{2}) +  \frac{r -m - n}{2 r - 2 m - 1}$.
\end{enumerate}
\end{thm}

\begin{thm} \label{coinc:typeO-2B} ({\bf Type 2B}) We have the following coincidences 
$$\cC_{\psi, 2B}(n,m) \cong  \cW_{s}(\go\gs\gp_{1|2r})^{\mathbb{Z}_2},$$ for $m,n \geq 0$ and $r\geq 1$.
\begin{enumerate}

\item $\ \displaystyle \psi = \frac{m - n + r}{2 m} ,\qquad s = -(r+\frac{1}{2}) +  \frac{r -n }{2 (m - n + r)},\qquad r \neq n-m$, 
\smallskip

\item $\ \displaystyle \psi = \frac{m - n - r}{2 (m - r)},\qquad s = -(r+\frac{1}{2}) + \frac{r - m}{2 (n -m + r)},\qquad r \neq m,\ m-n$,
\smallskip

\item $\ \displaystyle \psi = \frac{m - n}{2 (m + r)} ,\qquad s = -(r+\frac{1}{2})  + \frac{n + r}{2 (m + r)}$.
\end{enumerate}
\end{thm}

\begin{thm} \label{coinc:typeO-2C} ({\bf Type 2C}) We have the following coincidences 
$$\cC_{\psi, 2C}(n,m) \cong  \cW_{s}(\go\gs\gp_{1|2r})^{\mathbb{Z}_2},$$ for $m,n \geq 0$ and $r\geq 1$.
\begin{enumerate}

\item $\ \displaystyle \psi = \frac{1 + 2 m + 2 n + 2 r}{4 m},\qquad s = -(r+\frac{1}{2}) + \frac{1 + 2 n + 2 r}{2 (1 + 2 m + 2 n + 2 r)}$, 

\smallskip

\item $\ \displaystyle \psi =  \frac{1 + 2 m + 2 n}{4 (m + r)} ,\qquad s = -(r+\frac{1}{2}) + \frac{m + r}{2 r - 2 n - 1}$, 
\smallskip

\item $\ \displaystyle \psi =  \frac{1 + 2 m + 2 n - 2 r}{4 (m - r)}, \qquad s = -(r+\frac{1}{2}) + \frac{r -m}{2 r - 2 m - 2 n - 1},\qquad r \neq m$. 
\end{enumerate}
\end{thm}

\begin{cor} All isomorphisms $\cW_{k}(\go\gs\gp_{1|2m})^{\mathbb{Z}_2} \cong \cW_{\ell}(\go\gs\gp_{1|2n})^{\mathbb{Z}_2}$ occur in the following list:
\begin{equation} \begin{split} & k = -(m+\frac{1}{2}) + \frac{m + n}{2 m},\qquad k= -(m+\frac{1}{2}) + \frac{m}{2 (m+n)},
\\ & \qquad \ell =  -(n+\frac{1}{2}) + \frac{m + n}{2 n},\qquad \ell = -(n+\frac{1}{2}) + \frac{n}{2 (m+n)}.\end{split} \end{equation}
 This has central charge $$c = -\frac{(1 + 2 m) (1 + 2 n) (2 m n-m - n )}{2 (m + n)}.$$
\end{cor}

\end{document}